\documentclass[10pt]{amsart}
\usepackage{latexsym}
\usepackage{amsmath,amssymb,amsthm,ascmac}
\usepackage{wrapfig}
\usepackage[all]{xy}
\usepackage{proof}

\usepackage[pdftex]{graphicx}

\usepackage{moreverb}
\usepackage{fancybox}
\usepackage{fancyvrb}
\usepackage{comment}
\usepackage{color}
\usepackage{multirow}

\theoremstyle{plain}
\newtheorem{theorem}{Theorem}[section]
\newtheorem{lemma}[theorem]{Lemma}

\newtheorem{cor}[theorem]{Corollary}
\newtheorem{prop}[theorem]{Proposition}

\newtheorem{question}{Question}

\newtheorem{fact}{Fact}
\newtheorem{obs}[theorem]{Observation}

\theoremstyle{definition}
\newtheorem{definition}[theorem]{Definition}
\newtheorem{example}[theorem]{Example}

\theoremstyle{definition}

\newtheorem*{claim}{Claim}

\newtheorem*{ack}{Acknowledgement}

\newcommand{\upto}{\upharpoonright}
\newcommand{\fr}{\mbox{}^\smallfrown}

\newcommand{\om}{\omega}

\newcommand{\ep}{\varepsilon}

\newcommand{\qf}[1]{\bar{\sf #1}}

\newcommand{\pair}[1]{\langle #1 \rangle}
\newcommand{\comp}[1]{\langle #1 \rangle}

\title[The Arithmetical Hierarchy]{The Arithmetical Hierarchy:\\ A Realizability-Theoretic Perspective}

\author{Takayuki Kihara}

\begin{document}
\maketitle


\begin{abstract}
In this article, we investigate the arithmetical hierarchy from the perspective of realizability theory.
An experimental observation in classical computability theory is that the notion of degrees of unsolvability for {\em natural} arithmetical decision problems only plays a role in counting the number of quantifiers, jumps, or mind-changes.
In contrast, we reveal that when the realizability interpretation is combined with many-one reducibility, it becomes possible to classify natural arithmetical problems in a very nontrivial way.
\end{abstract}

\section{Introduction}\label{sec:introduction}
The main research theme of computation theory from birth to the present day is to analyze the computational complexity (degrees of computational difficulty) of various problems.
The most traditional type of problem is called a {\em decision problem}, which refers to a problem that asks to determine if a statement with one parameter is true or false for each parameter.
In computability theory, the basic tools for measuring the complexity of decision problems are reducibility notions such as many-one reducibility and Turing reducibility \cite{OdiBook}.

However, these traditional approaches are too rough to classify ``{\em natural}'' decision problems, and often only count the number of alternations of quantifiers that appear in the problem.
This experimental observation that the degree structure of natural decision problems is too simple is also suggested by outstanding problems in computability theory, such as the problem of finding a natural solution of Post's problem and the Martin conjecture \cite{Mon19}.
This ``natural-degree'' problem is often the subject of debate among computability theorists \cite{Sim07}; however, despite many years of research, no definitive method has been found for measuring computability-theoretic complexity of ``natural'' decision problems (beyond counting the number of quantifiers, jumps, or mind-changes).

One of the aims of this article is to provide a novel method of measuring the computability-theoretic complexity of natural problems that goes beyond counting.
Our approach can be summed up in a slogan as follows: ``{\em A (decision) problem is a formula, not a subset of natural numbers or strings}.''
Of course, in theory, any problem should be describable as a formula (in the sense of formal logic), so people who have never studied computation theory will easily accept this slogan, and may even think it is a self-evident truth that does not need to be stated.
The difference between formulas and subsets is that the former is intensional, while the latter is extensional, and the impact of this difference is enormous.
Introducing the notion of a decision problem as a formula makes it possible to talk about the notion of {\em witness} for the truth (e.g., an existential witness) of a formula.

\begin{definition}[see Definition \ref{def:many-one-formula-def} for the rigorous definition]\label{def:intro-many-one-1}
A formula $\varphi$ is {\em (realizability-theoretic) many-one reducible to $\psi$} if there exists a computable function $h$ such that, for any instance $x$, $\varphi(x)$ is true iff $\psi(h(x))$ is true, and moreover, any witness for the truth of $\varphi(x)$ is effectively converted into a witness for the truth of $\psi(h(x))$ and vice versa.
\end{definition}

Formally, the notion of witness is formulated using Kleene's {\em realizability interpretation} \cite{Kle45,Tro98}.
In computational complexity theory, the polytime version of Definition \ref{def:intro-many-one-1} has been introduced by Levin \cite{Lev73} (formulated as a search problem rather than a decision problem), and known as {\em Levin reducibility}.
Nevertheless, this notion has not been studied in computability theory before the author's article \cite{Kihara}.
Definition \ref{def:intro-many-one-1} can also be obtained as the realizability interpretation of many-one reducibility in constructive logic \cite{Vel09,FoFe23}.

When the notion of witness is combined with many-one reducibility, each problem now becomes one that not only requires us to determine whether it is true or false, but also to provide a witness for the truth if it is true.
Then, interestingly, it becomes possible to classify natural problems in a very nontrivial way.
For example, the author \cite{Kihara} has classified natural $\Sigma_2$-decision problems as follows:
\begin{itemize}
\item Boundedness for countable posets is $\forall^\infty$-complete.
\item Finiteness of width for countable posets is $\forall^\infty\forall$-complete.
\item Non-density for countable linear orders is $\exists\forall$-complete.
\end{itemize}

Here, ``$\forall m\exists n\geq m$ (there are infinitely many $n$ such that ...)'' is abbreviated as ``$\exists^\infty n$'', and ``$\exists m\forall n\geq m$ (for all but finitely many $n$ ...)'' is abbreviated as ``$\forall^\infty n$''.
By combining these quantifiers, one can introduce various classes of formulas (decision problems), which often have natural complete problems as above.
Also, a countable structure can be presented by some computational method (as in computable structure theory); see Section \ref{sec:countable-structure}.

The above three problems are all classically $\Sigma_2$-complete (so classically indistinguishable), but our reducibility notion can distinguish the computability-theoretic complexity of these problems.
In other words, the difficulty of searching for existential witnesses differs in these problems.
And this difference is caused by the difference in the ``quantifier-patterns (i.e., combinations and order of appearance of $\exists,\forall,\exists^\infty,\forall^\infty$, etc.)''~used to describe these problems.

By analyzing further natural decision problems, we reveal the effectiveness of classification using quantifier-patterns (where, note that the author had not yet recognized the importance of classifying decision problems using quantifier-patterns when writing \cite{Kihara}).
In order to see this, it is important to analyze a formula and its dual (see Section \ref{sec:quantifier-dual}) simultaneously.
Then the following stronger reducibility notion is useful.
\begin{definition}[see also Definition \ref{def:di-reducible}]
A formula $\varphi$ is {\em many-one di-reducible to $\psi$} if both $\varphi$ and its dual are many-one reducible to $\psi$ and its dual via a common $h$.
\end{definition}

Then one can introduce the notion of dicompleness in a straightforward manner.
Then, in fact, all of the above examples are dicomplete w.r.t.~the corresponding classes.

In this article, we mainly focus on problems that are classically $\Sigma_3$- or $\Pi_3$-complete.
For instance, this article reveals that the following ``$\Pi_3$-complete'' problems, which cannot be distinguished in traditional computability theory, actually have different complexities.
\begin{itemize}
\item {\sf Lattice}: Being lattice for countable posets is $\forall\forall^\infty$-dicomplete.
\item {\sf Atomic}: Atomicity for countable posets is $\forall\forall^\infty$-dicomplete.
\item {\sf LocFin}: Local finiteness for countable graphs is $\forall\forall^\infty\forall$-dicomplete.
\item ${\sf FinBranch}$: Being finitely branching for countable trees is $\forall\forall^\infty\forall$-dicomplete.
\item {\sf Compl}: Complementedness for countable posets is $\forall\exists\forall$-dicomplete.
\item ${\sf InfDiam}_{\sf conn}$: Unboundedness of the diameters of connected components for countable graphs is $\exists^\infty\exists\forall$-dicomplete.
\item {\sf Cauchy}: Cauchyness for rational sequences is $\forall^\downarrow\forall^\infty$-dicomplete.
\item {\sf SimpNormal}: Simple normality in base $2$ for real numbers is $\forall^\downarrow\forall^\infty$-dicomplete.
\item ${\sf Perfect}_{\sf bin}$: Perfectness for countable binary trees is $\forall(\forall\to\exists\forall)$-dicomplete.
\end{itemize}

\begin{figure}[t]
\begin{center}
\includegraphics[width=30mm]{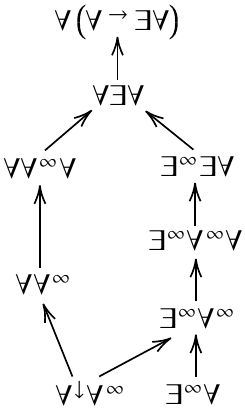}
\qquad
\qquad
\includegraphics[width=50mm]{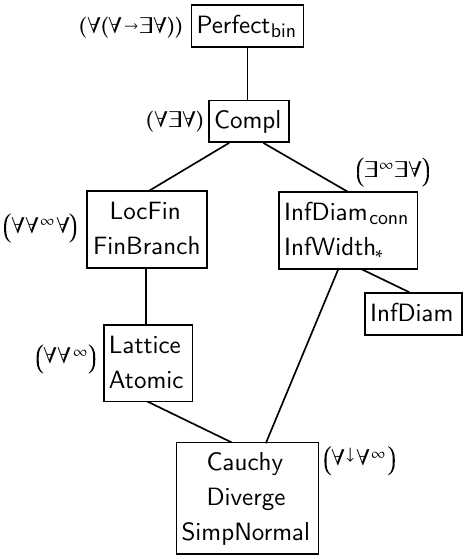}
\end{center}
\caption{Complexity of $\Pi_3$-complete problems}\label{fig:natural}
\end{figure}

The true complexity of these ``classically $\Pi_3$-complete'' problems is desplayed as in Figure \ref{fig:natural}.
Here, for quantifier-patters ${\sf P},{\sf Q}$, the arrow ${\sf P}\to {\sf Q}$ means that a ${\sf P}$-complete problem is many-one direducible to a ${\sf Q}$-complete problem.
No further arrows can be added.

Let us summarize our key idea behind this classification.
When we look at natural $\Pi_3$-complete problems, we find that they often contain quantifiers of the form $\forall^\infty$ and $\exists^\infty$, rather than just $\forall\exists\forall$.
The order in which these four quantifiers $\forall,\exists,\forall^\infty,\exists^\infty$ occur in a formula is closely related to the complexity of the corresponding problem.
Therefore, it is important to determine how many quantifier-patterns (i.e., finite strings on $\{\forall,\exists,\forall^\infty,\exists^\infty\}$) can exist at each level of the arithmetical hierarchy.
In this article, we also show the following:
\begin{itemize}
\item The number of (di-)many-one equivalence classes of $\Pi_2$ quantifier-patterns is exactly $1$.
\[\forall\exists.\]
\item The number of (di-)many-one equivalence classes of $\Sigma_2$ quantifier-patterns is exactly $3$.
\[\exists\forall,\forall^\infty\forall,\forall^\infty.\]
\item The number of many-one equivalence classes of $\Sigma_3$ quantifier-patterns is exactly $3$.
\[\exists\forall\exists,\forall^\infty\exists^\infty,\forall^\infty\exists.\]
\item The number of many-one equivalence classes of $\Pi_3$ quantifier-patterns is exactly $5$.
\[\forall\exists\forall,\exists^\infty\forall^\infty\forall,\exists^\infty\forall,\forall\forall^\infty\forall,\forall\forall^\infty.\]
\item The number of di-many-one equivalence classes of $\Pi_3$ quantifier-patterns is exactly $7$.
\[\forall\exists\forall,\exists^\infty\exists\forall,\exists^\infty\forall^\infty\forall,\exists^\infty\forall^\infty,\exists^\infty\forall,\forall\forall^\infty\forall,\forall\forall^\infty.\]
\end{itemize}


Let us comment on one more important point.
Although we have presented our results here in the context of many-one reducibility, the results of this article are all applicable to {\em Wadge reducibility} (a continuous version of many-one reducibility), a well-studied notion in descriptive set theory \cite{Wad83,Wad12}.
In fact, the study of Wadge reducibility in intuitionistic descriptive set theory \cite{VeldmanPhD,Vel90,Vel09,Vel22} is what triggered our study.

Our results are applicable to both classical and constructive logics; this article provides a new perspective and powerful techniques for research on the arithmetical hierarchy (the Borel hierarchy) in both logics.
As for the latter, there has been a growing interest in constructive logic, and in recent years, there has been a rise in the number of studies on the arithmetical hierarchy in constructive logic, e.g.~ \cite{ABHK04,Ber06,BeSt17,Bur00,Bur04,FoFe23,FuKu21,Fuku22,Fuku23,Nakata,SUZ16}.
Developments linked to these are also expected.



\section{Basic definitions}
For the basics of computability theory, see \cite{OdiBook}.
For an abstract foundation of this research, see the author's previous article \cite{Kihara}.
\subsection{Quantifier}
In this article, we consider only arithmetical quantifiers; that is, bounded variables range over the natural numbers.
However, some free variables may range over (indices of) total functions.
We also deal with quantifiers $\exists^\infty$ and $\forall^\infty$, which express ``{\em for infinitely many}'' and ``{\em for all but finitely many (for cofinitely many)},'' as well as the existential quantifier $\exists$ and the universal quantifier $\forall$.
Formally, quantifiers $\exists^\infty$ and $\forall^\infty$ are defined as follows.
\begin{align*}
\exists^\infty n.\ \varphi(n)&\equiv\forall m\exists n\geq m.\ \varphi(n),\\
\forall^\infty n.\ \varphi(n)&\equiv\exists m\forall n\geq m.\ \varphi(n).
\end{align*}

From here on, the term ``quantifier'' refers to one of $\exists$, $\forall$, $\exists^\infty$, or $\forall^\infty$.
A finite sequence $\langle{\sf Q}_0,{\sf Q}_1,\dots,{\sf Q}_\ell\rangle\in\{\exists,\forall,\exists^\infty,\forall^\infty\}^\ast$ of quantifiers is often abbreviated to ${\sf Q}_0{\sf Q}_1\dots{\sf Q}_\ell$, and called a {\em quantifier-pattern}.

\begin{definition}
Let ${\sf Q}_0{\sf Q}_1\dots{\sf Q}_\ell$ be a quantifier-pattern.
A formula of the following form is called a {\em ${\sf Q}_0{\sf Q}_1\dots{\sf Q}_\ell$-formula}.
\[{\sf Q}_0n_0{\sf Q}_1n_1\dots,{\sf Q}_\ell n_\ell\;\theta(n_0,n_1,\dots,n_\ell,x).\]

Here, $\theta$ is a bounded formula.
\end{definition}

\begin{example}
A $\exists\forall\exists$-formula is exactly a $\Sigma_3$-formula.
\end{example}

We introduce a notion for comparing quantifier-patterns.
For quantifier-patterns $\bar{\sf P}$ and $\bar{\sf R}$ and a quantifier ${\sf Q}$, consider the following rewriting rules:
\[\bar{\sf P}\exists^\infty\bar{\sf R}\to \bar{\sf P}\forall\exists\bar{\sf R};\quad\bar{\sf P}\forall^\infty\bar{\sf R}\to \bar{\sf P}\exists\forall\bar{\sf R};\quad\bar{\sf P}\exists\exists\bar{\sf R}\to\bar{\sf P}\exists\bar{\sf R};\quad\bar{\sf P}\forall\forall\bar{\sf R}\to\bar{\sf P}\forall\bar{\sf R};\quad\bar{\sf P}\bar{\sf R}\to\bar{\sf P}{\sf Q}\bar{\sf R}.\]

\begin{definition}
A quantifier-pattern $\bar{\sf Q}$ is {\em absorbable into $\bar{\sf Q}'$} if one can obtain $\bar{\sf Q}'$ from $\bar{\sf Q}$ by applying the above rewriting rules a finite number of times.
\end{definition}

In this case, we write $\bar{\sf Q}\to^\ast\bar{\sf Q}'$, or $\bar{\sf Q}\to\bar{\sf Q}'$ for brevity.

\begin{example}
$\forall\forall^\infty\forall$ is absorbable into $\forall\exists\forall$:
\[\forall\forall^\infty\forall\to\forall\exists\forall\forall\to\forall\exists\forall.\]
\end{example}

\begin{example}\label{exa:Sigma03-absorb}
All of the following are absorbable into $\exists\forall\exists$:
\begin{multline*}
\forall^\infty\exists,\exists\exists^\infty,\forall^\infty\exists^\infty,\exists\forall^\infty\exists,\forall^\infty\forall\exists,\exists\forall\exists^\infty,\exists\exists^\infty\exists,\exists\forall^\infty\exists^\infty,\forall^\infty\forall\exists^\infty,\\
\forall^\infty\exists^\infty\exists,\exists\forall^\infty\forall\exists,\exists\forall\exists^\infty\exists,\exists\forall^\infty\forall\exists^\infty,\forall^\infty\forall\exists^\infty\exists,\exists\forall^\infty\exists^\infty\exists,\exists\forall^\infty\forall\exists^\infty\exists.
\end{multline*}

In Figure \ref{fig:arith}, a solid arrows indicate an absorption relation.
A dotted arrow is a di-reducibility relation obtained by Proposition \ref{prop:ex-ex-infty-bm-reducible} below.
\end{example}

\begin{figure}[t]
\begin{center}
\includegraphics[width=60mm]{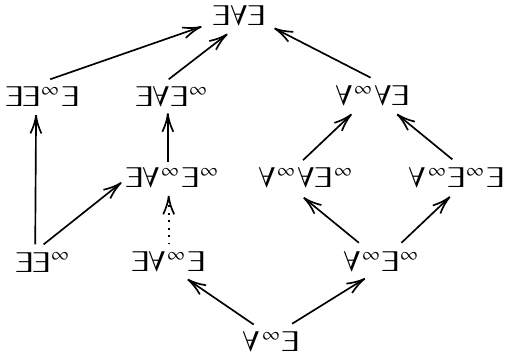}
\end{center}
\caption{Example of absorption relations for $\Sigma_3$ quantier-patterns.}\label{fig:arith}
\end{figure}

\begin{definition}[Arithmetical hierarchy]\label{def:arithmetical-hierarchy}
The classes $\Sigma_n$ and $\Pi_n$ of quantifier-patterns $\bar{\sf Q}$ are defined inductively as follows.
\begin{enumerate}
\item $\exists,\forall,\exists^\infty,\forall^\infty$ are  $\Sigma_1,\Pi_1,\Pi_2,\Sigma_2$ respectively.
\item If $\bar{\sf Q}$ is $\Sigma_n$, then $\exists\bar{\sf Q},\forall\bar{\sf Q},\exists^\infty\bar{\sf Q},\forall^\infty\bar{\sf Q}$ are $\Sigma_n,\Pi_{n+1},\Pi_{n+1},\Sigma_{n+2}$, respectively.
\item If $\bar{\sf Q}$ is $\Pi_n$, then $\exists\bar{\sf Q},\forall\bar{\sf Q},\exists^\infty\bar{\sf Q},\forall^\infty\bar{\sf Q}$ are $\Sigma_{n+1},\Pi_{n},\Pi_{n+2},\Sigma_{n+1}$, respectively.
\end{enumerate}
\end{definition}

Obviously, this definition is consistent with the classical definition of the arithmetic hierarchy.
\begin{obs}
For $\Gamma\in\{\Sigma_n,\Pi_n\}$, if a quantifier-pattern $\bar{\sf Q}$ is $\Gamma$, then a $\bar{\sf Q}$-formula is a $\Gamma$-formula (in the classical sense).
\end{obs}

\begin{obs}\label{obs:arit-hier-preserve}
Let $\Gamma$ be either $\Sigma$ or $\Pi$, and let $\check{\Gamma}$ be the other.
If $\qf{P}$ is $\Gamma_n$ and is absorbable into $\qf{Q}$, then $\qf{Q}$ is either $\Gamma_m$ for some $m\geq n$ or $\check{\Gamma}_m$ for some $m>n$.
\end{obs}

For quantifier-patterns $\bar{\sf P},\bar{\sf Q}$, we say that $\bar{\sf P}=({\sf P}_i)_{i<k}$ is a {\em subpattern of $\bar{\sf Q}=({\sf Q}_j)_{j<\ell}$} if there exists a strictly increasing map $h\colon k\to\ell$ such that ${\sf P}_i={\sf Q}_{h(i)}$ holds for any $i<k$.

\begin{prop}\label{prop:Sigma3-bluteforce-list}
A quantifier-pattern $\bar{\sf Q}$ is $\Sigma_3$ iff there exists a pattern $\bar{\sf Q}'$ in Example \ref{exa:Sigma03-absorb} such that $\bar{\sf Q}$ is absorbable into $\bar{\sf Q}'$ and vice versa.
\end{prop}

\begin{proof}
($\Leftarrow$)
It is clear that the quantifier-petterns presented in Example \ref{exa:Sigma03-absorb} are all $\Sigma_3$.
By Observation \ref{obs:arit-hier-preserve}, $\bar{\sf Q}\leftrightarrow\bar{\sf Q}'\in\Sigma_3$ implies $\bar{\sf Q}\in\Sigma_3$.

($\Rightarrow$)
Let $\bar{\sf Q}$ be a ${\Sigma_3}$-pattern.
Let ${\sf P}\bar{\sf Q}'$ be the longest tail of $\bar{\sf Q}$ (i.e., $\bar{\sf Q}$ is of the form $\qf{R}{\sf P}\bar{\sf Q}'$) which is ${\Sigma_3}$.
Looking at Definition \ref{def:arithmetical-hierarchy}, the only quantifiers that yield $\Sigma_3$ are $\exists$ and $\forall^\infty$, so ${\sf P}\in\{\exists,\forall^\infty\}$.
If ${\sf P}=\exists$, then $\bar{\sf Q}'$ is $\Pi_2$.
If ${\sf P}=\forall^\infty$ then $\bar{\sf Q}'$ is either $\Pi_2$ or $\Sigma_1$.
It is easy to see that a $\Pi_2$-pattern contains $\exists^\infty$ or $\forall\exists$ as a subpattern.
A $\Sigma_1$-pattern clearly contains $\exists$.
To summarize the above, $\bar{\sf Q}$ contains one of $\exists\exists^\infty,\exists\forall\exists,\forall^\infty\exists^\infty,\forall^\infty\exists$ as a subpattern.
Furthermore, all of these subpatterns are $\Sigma_3$.

If the addition of a quantifier to a $\Sigma_3$-pattern still maintains $\Sigma_3$, then it can only be one of the following:
Add $\exists$ after $\exists^\infty$ or before $\forall^\infty$; add $\forall$ before $\exists^\infty$ or after $\forall^\infty$; insert $\exists^\infty$ between $\forall$ and $\exists$; or insert $\forall^\infty$ between $\exists$ and $\forall$.
Regardless of which of $\exists\exists^\infty,\exists\forall\exists,\forall^\infty\exists^\infty,\forall^\infty\exists$ is the starting point, the insertion result saturates in the form of $\exists^i\forall^\infty\forall^j\exists^\infty\exists^k$.
Since the duplication of $\exists$ ($\forall$, respectively) can be mutually absorbed into a single $\exists$ ($\forall$, respectively), it reaches $\exists\forall^\infty\forall\exists^\infty\exists$.
Therefore, we only need to consider subpatterns of $\exists\forall^\infty\forall\exists^\infty\exists$, but if we eliminate the duplication of $\exists$ and $\forall$, these are all covered by the quantifier-patterns presented in Example \ref{exa:Sigma03-absorb}.
\end{proof}

%
%

\noindent
{\em Declaration:}
In this article, an (arithmetical) formula always refers to a $\bar{\sf Q}$-formula $\varphi(\bar{x})$ for a quantifier-pattern $\bar{\sf Q}$.
Here, $\varphi$ is assumed to contain no parameters other than $\bar{x}$, and $\bar{x}$ is a sequence of natural numbers or (indices of) total computable functions.
A computable function parameter is always a free variable; that is, quantification over a computable function parameter never appears in a formula, even if a computable function is identified with its index (which is a natural number).

The reason we deal with function parameters is because we analyze decision problems on countable structures, where a computable structure is coded by a total computable function (see Section \ref{sec:countable-structure}).
By the Kreisel-Lacombe-Shoenfield theorem \cite{KLS59}, there is no difference between total computability over total computable functions and total computability over their indices, so we simply consider each function parameter $x_i$ as a total computable function (or a computable element in $\om^\om$) rather than its index.

\subsection{Realizability}
We now consider the notion of witness for a formula, and also a transformation of a given witness $\alpha$ into another witness $\beta$.
As for the latter, the modern approach is to think of a witness $\alpha$ as being given as an oracle, and then transforming it into another witness $\beta$ in a computable way --- this is the approach of ``{\em topological objects, computable morphisms} \cite{Bau00}.''
Formally, we consider Kleene's functional realizability interpretation for arithmetical formulas (see \cite{Tro98}).

In the following, we use the identifications $\om\times\om^\om\simeq\om^\om$ and $(\om^\om)^\om\simeq\om^\om$ without mentioning.
So, for $\alpha\in\om^\om\simeq(\om^\om)^\om$, $\alpha(n)$ is still an element in $\om^\om$; that is, $\alpha(n)(m)$ is identified with $\alpha(\pair{n,m})$.

\begin{definition}[Kleene]
For $\alpha\in\om^\om$ and a formula $\varphi$, the binary relation $\alpha\Vdash\varphi(\bar{x})$ is inductively defined as follows:
\begin{align*}
\alpha\Vdash\theta(\bar{x})&\iff \alpha=\bar{x}\mbox{ and }\theta(\bar{x})\quad\mbox{(for bounded $\theta$)}\\
\pair{t,\alpha}\Vdash\exists n\varphi(n,\bar{x})&\iff \alpha\Vdash\varphi(t,\bar{x})\\
\alpha\Vdash\forall n\varphi(n,\bar{x})&\iff \alpha(n)\Vdash\varphi(n,\bar{x})\mbox{ for all $n\in\om$}\\
\pair{t,\alpha}\Vdash\forall^\infty n\varphi(n,\bar{x})&\iff \alpha(n)\Vdash \varphi(n,\bar{x})\mbox{ for all $n\geq t$}\\
\alpha\Vdash\exists^\infty n\varphi(n,\bar{x})&\iff \pi_1\circ\alpha(n)\Vdash\varphi(\pi_0\circ\alpha(n),\bar{x})\\
&\qquad\mbox{ and }\pi_0\circ\alpha(n)\geq n\mbox{ for all $n$.}
\end{align*}

If $\alpha\Vdash\varphi(\bar{x})$, then we say that $\alpha$ is a {\em witness for} $\varphi(\bar{x})$.
\end{definition}

We only consider $\Sigma_3$- and $\Pi_3$-formulas in this article, so the definition of functional realizability can be simplified.
This is because a witness for the inner $\Pi_2$-subformula can be restored in a computable way without explicitly giving it, so it can be omitted.
For instance, a witness for a $\Sigma_3$ formula can be replaced with a sequence of existential witnesses for the outermost block of existential quantifiers.

\begin{example}
A (simplified) witness for a $\exists\forall^\infty\exists^\infty$-formula $\exists a\forall^\infty b\exists^\infty c\varphi(a,b,c,x)$ is a pair $\pair{a,b'}$ such that the $\Pi_2$-subformula $\forall b\geq b'\exists^\infty c\varphi(a,b,c,x)$ is true.
\end{example}

\begin{example}
A (simplified) witness for a $\forall\exists^\infty\forall$-formula $\forall a\exists^\infty b\forall c\varphi(a,b,c,x)$ is a function $h$ such that, for any $a$ and $b$, $h(a,b)\geq b$ and the $\Pi_1$-subformula $\forall c\varphi(a,h(a,b),c,x)$ is true.
\end{example}

%
%
%

As mentioned in Section \ref{sec:introduction}, our key idea for distinguishing various natural decision problems is {\em not} to identify a {\em decision problem} with a {\em subset}.
In other words, we consider a formula itself to be a decision problem.
Based on this perspective, various computability-theoretic notions can be redefined as operations on formulas.

\begin{definition}[see \cite{Kihara}]\label{def:many-one-formula-def}
A $\bar{\sf Q}$-formula $\varphi(\bar{x})$ is {\em many-one reducible to} a $\bar{\sf Q}'$-formula $\psi(\bar{x})$ if there exist computable functions $\eta,r_-,r_+$ such that the following holds:
\begin{enumerate}
\item $\varphi(\bar{x})$ is true iff $\psi(\eta(\bar{x}))$ is true.
\item $\alpha\Vdash\varphi(\bar{x})$ implies $r_-(\alpha,\bar{x})\Vdash\psi(\eta(\bar{x}))$.
\item $\alpha\Vdash\psi(\eta(\bar{x}))$ implies $r_+(\alpha,x)\Vdash\varphi(\bar{x})$.
\end{enumerate}

In this case, we write $\varphi(\bar{x})\leq_{\sf m}\psi(\bar{x})$.
Often, the free variable part $\bar{x}$ is omitted, and $\varphi(\bar{x})\leq_{\sf m}\psi(\bar{x})$ is simply written as $\varphi\leq_{\sf m}\psi$.
\end{definition}

In other words, $\varphi\leq_{\sf m}\psi$ iff there exists a computable function $\eta$ such that $\varphi(\bar{x})\leftrightarrow\psi(\eta(\bar{x}))$ is realizable.
Here, $r_-$ is a realizer for the forward implication $\varphi(\bar{x})\to\psi(\eta(\bar{x}))$, and $r_+$ is a realizer for the backward implication $\varphi(\bar{x})\leftarrow\psi(\eta(\bar{x}))$.
Its polytime version has been introduced by Levin \cite{Lev73}; see also \cite{Kihara}.

\begin{definition}
An arithmetical formula $\varphi$ is {\em strictly $\qf{Q}$-complete} if $\varphi$ is a $\qf{Q}$-formula and $\psi\leq_{\sf m}\varphi$ for any $\qf{Q}$-formula $\psi$.
A formula $\varphi$ is {\em $\qf{Q}$-complete} if it is ${\sf m}$-equivalent to a strictly $\qf{Q}$-complete formula.
\end{definition}



\begin{prop}\label{prop:Q-complete}
A $\bar{\sf Q}$-complete formula exists for any quantifier-pattern $\bar{\sf Q}$.
\end{prop}

\begin{proof}
Given a quantifier ${\sf P}$ and an arithmetic formula $\varphi$, the formula ${\sf P}\varphi$ is defined as follows.
For each $\bar{x}=(x_n)_{n\in\om}$,
\[({\sf P}\varphi)(\bar{x},\bar{y})\equiv{\sf P} n.\,\varphi(x_n,\bar{y})\]

The free variable part $\bar{y}$ is omitted below.
We show that $\varphi\leq_{\sf m}\psi$ implies ${\sf P}\varphi\leq_{\sf m}{\sf P}\psi$.
Assume $\varphi\leq_{\sf m}\psi$ via $\eta,r_-,r_+$.
Then define $\eta'(\bar{x})=(\eta(x_n))_{n\in\om}$ for $\bar{x}=(x_n)_{n\in\om}$.

(${\sf P}=\exists$):
A witness for $(\exists\varphi)(\bar{x})$ is of the form $(t,a)$.
Then $\alpha$ witnesses $\varphi(x_t)$, so $r_-(\alpha,x_t)$ witnesses $\psi(\eta(x_t))$.
Hence, $(t,r_-(\alpha,x_t))$ witnesses $(\exists\psi)(\eta'(x))$.
Therefore, $r'_-\colon(t,\alpha)\mapsto (t,r_-(\alpha,x_t))$ gives a realizer for the forward direction.
A similar argument applies to the conversion from a witness for $(\exists\psi)(\eta'(x))$ to a witness for $(\exists\varphi)(x)$.

(${\sf P}=\forall$):
If $\alpha$ witnesses $(\forall\varphi)(\bar{x})$ then $\alpha(n)$ witnesses $\varphi(x_n)$, so $r_-(\alpha(n),x_n)$ witnesses $\psi(\eta(x_n))$.
Hence, $\lambda n.r_-(\alpha(n),x_n)$ witnesses $(\forall\psi)(\eta'(\bar{x}))$.
Therefore, $r'_-\colon \alpha\mapsto \lambda n.r_-(\alpha(n),x_n)$ gives a realizer for the forward direction.
A similar argument applies to the conversion from a witness for $(\forall\psi)(\eta'(\bar{x}))$ to a witness for $(\forall\varphi)(\bar{x})$.

(${\sf P}=\forall^\infty$):
A witness for $(\forall^\infty\varphi)(\bar{x})$ is of the form $(t,\alpha)$.
For any $n\geq t$, $\alpha(n)$ witnesses $\varphi(x_n)$, so $r_-(\alpha(n),x_n)$ witnesses $\psi(\eta(x_n))$.
Hence, $(t,\lambda n.r_-(\alpha(n),x_n))$ witnesses $(\forall^\infty\psi)(\eta'(\bar{x}))$.
A similar argument applies to the conversion from a witness for $(\forall^\infty\psi)(\eta'(\bar{x}))$ to a witness for $(\forall^\infty\varphi)(\bar{x})$.

(${\sf P}=\exists^\infty$):
If $\alpha$ witnesses $(\exists^\infty\varphi)(\bar{x})$ then for $\alpha_i(n)=\pi_i(\alpha(n))$ we have $\alpha_0(n)\geq n$ and $\alpha_1(n)$ witnesses $\varphi(x_{\alpha_0(n)})$.
Therefore, $r_-(\alpha_1(n),x_{\alpha_0(n)})$ witnesses $\psi(\eta(x_{\alpha_0(n)}))$.
Hence, $\lambda n.\pair{\alpha_0(n),r_-(\alpha_1(n),x_{\alpha_0(n)})}$ witnesses $(\exists^\infty\psi)(\eta'(\bar{x}))$.
A similar argument applies to the conversion from a witness for $(\exists^\infty\psi)(\eta'(\bar{x}))$ to a witness for $(\exists^\infty\varphi)(\bar{x})$.

Now, inductively assume that a $\bar{\sf Q}$-complete formula $\comp{\qf{Q}}$ is given.
If we define $\comp{{\sf P}\bar{\sf Q}}={\sf P}\comp{\bar{\sf Q}}$, then by the above discussion, $\comp{{\sf P}\bar{\sf Q}}$ is ${\sf P}\bar{\sf Q}$-complete.
By induction, this shows that there exists a $\bar{\sf Q}$-complete  formula for any quantifier-pattern $\bar{\sf Q}$.
\end{proof}

\begin{example}
The following is an example of a $\exists^\infty\forall^\infty\forall$-formula:
\[\exists^\infty n\forall^\infty m\forall k.\;x(n,m,k)=0\]
\end{example}

We fix a $\bar{\sf Q}$-complete problem $\comp{\bar{\sf Q}}$.

\begin{obs}
If a quantifier-pattern $\bar{\sf Q}$ is absorbable into $\bar{\sf Q}'$, then $\pair{\bar{\sf Q}}\leq_{\sf m}\pair{\bar{\sf Q}'}$.
\end{obs}

Now, let us think about $\Sigma_2$ formulas.
Analyzing natural examples of $\Sigma_2$ problems, we find that typical examples are described by one of $\forall^\infty,\forall^\infty\forall,\exists\forall$.
If we consider $\Sigma_2$ problems to be $\Sigma_2$ sets, we will not be able to distinguish between them at all, but by directly analyzing the complexity of decision problems as formulas rather than subsets, we can understand the differences between them.
In fact, the author \cite{Kihara} has shown that the $\Sigma_2$-patterns $\forall^\infty$, $\forall^\infty\forall$ and $\exists\forall$ each yield different levels of complexity:

\begin{fact}[\cite{Kihara}]\label{fact:Sigma-2-sep}
$\comp{\forall^\infty}<_{\sf m}\comp{\forall^\infty\forall}<_{\sf m}\comp{\exists\forall}$.
\end{fact}

\subsection{Dual quantifier}\label{sec:quantifier-dual}

The {\em dual} ${\sf Q}^{\sf d}$ of a quantifier ${\sf Q}$ is defined as follows:
\[
\exists^{\sf d}=\forall,\quad\forall^{\sf d}=\exists,\quad(\exists^\infty)^{\sf d}=\forall^\infty,\quad(\forall^\infty)^{\sf d}=\exists^\infty
\]

The dual of a quantifer-pattern $\bar{\sf Q}={\sf Q}_0{\sf Q}_1\dots{\sf Q}_\ell$ is defined as $\bar{\sf Q}^{\sf d}={\sf Q}_0^{\sf d}{\sf Q}_1^{\sf d}\dots{\sf Q}_\ell^{\sf d}$.
The dual $\varphi^{\sf d}$ of a $\bar{\sf Q}$-formula $\varphi=\bar{\sf Q}\bar{n}\theta(\bar{n},x)$ is defined as $\bar{\sf Q}^{\sf d}\bar{n}\neg\theta(\bar{n},x)$; that is,
\[
\big({\sf Q}_0n_0{\sf Q}_1n_1\dots{\sf Q}_\ell n_\ell\ \theta(n_0,n_1,\dots,n_\ell,x)\big)^{\sf d}
={\sf Q}_0^{\sf d}n_0{\sf Q}_1^{\sf d}n_1\dots{\sf Q}_\ell^{\sf d} n_\ell\ \neg\theta(n_0,n_1,\dots,n_\ell,x),
\]
where $\theta$ is a bounded formula.
Of course, classically, the dual $\varphi^{\sf d}$ of an arithmetical formula $\varphi$ is merely the negation $\neg\varphi$.

\begin{obs}
The dual $\varphi^{\sf d}$ of a $\bar{\sf Q}$-formula $\varphi$ is a $\bar{\sf Q}^{\sf d}$-formula.
\end{obs}


\begin{definition}
For pairs $(\varphi,\varphi')$ and $(\psi,\psi')$ of formulas, we say that $(\varphi,\varphi')$ is {\em many-one reducible to $(\psi,\psi')$} if $\varphi\leq_{\sf m}\psi$ and $\varphi'\leq_{\sf m}\psi'$ via a common $\eta$.
Here, realizers $r_-$ and $r_+$ can be different.
In this case, we write $(\varphi,\varphi')\leq_{\sf m}(\psi,\psi')$.
\end{definition}

\begin{definition}\label{def:di-reducible}
A formula $\varphi$ is {\em many-one di-reducible to} $\psi$ if $(\varphi,\varphi^{\sf d})\leq_{\sf m}(\psi,\psi^{\sf d})$.
In this case, we write $\varphi\leq_{\sf dm}\psi$.
A formula $\varphi$ is {\em strictly $\qf{Q}$-dicomplete} if $\varphi$ is a $\qf{Q}$-formula and $\psi\leq_{\sf dm}\varphi$ for any $\bar{\sf Q}$-formula $\psi$.
A formula $\varphi$ is {\em $\qf{Q}$-dicomplete} if it is ${\sf dm}$-equivalent to strict $\qf{Q}$-dicomplete formula.
\end{definition}

\begin{prop}\label{prop:quantifier-bi-complete}
A $\bar{\sf Q}$-dicomplete formula exists for any quantifier-pattern $\bar{\sf Q}$.
\end{prop}

\begin{proof}
In fact, the proof of Proposition \ref{prop:Q-complete} shows that $\varphi\leq_{\sf m}\psi$ implies ${\sf P}\varphi\leq_{\sf m}{\sf P}\psi$, but since the reduction $\eta$ is the same regardless of ${\sf P}$, so in fact the previous proof shows that $\varphi\leq_{\sf bm}\psi$ implies ${\sf P}\varphi\leq_{\sf bm}{\sf P}\psi$.
Therefore, the $\bar{\sf Q}$-complete formula $\comp{\bar{\sf Q}}$ constructed in Proposition \ref{prop:Q-complete} is actually $\bar{\sf Q}$-dicomplete.
\end{proof}

\begin{obs}\label{obs:absorb-bm-relation}
If a quantifier-pattern $\bar{\sf Q}$ is absorbable into $\bar{\sf Q}'$ then $\pair{\bar{\sf Q}}\leq_{\sf bm}\pair{\bar{\sf Q}'}$.
\end{obs}

Let us extract the following useful lemma from the proof of Proposition \ref{prop:quantifier-bi-complete}.

\begin{lemma}\label{lem:quantifier-add-lemma}
For quantifier-patterns $\bar{\sf P},\bar{\sf Q},\bar{\sf Q}'$, 
$\comp{\bar{\sf Q}}\leq_{\sf m}\comp{\bar{\sf Q}'}$ implies $\comp{\bar{\sf P}\bar{\sf Q}}\leq_{\sf m}\comp{\bar{\sf P}\bar{\sf Q}'}$.
Similarly, $\comp{\bar{\sf Q}}\leq_{\sf dm}\comp{\bar{\sf Q}'}$ implies $\comp{\bar{\sf P}\bar{\sf Q}}\leq_{\sf dm}\comp{\bar{\sf P}\bar{\sf Q}'}$.
\end{lemma}

\begin{proof}
As mentioned in the proof of proposition \ref{prop:quantifier-bi-complete}, $\varphi\leq_{\sf m}\psi$ implies ${\sf P}\varphi\leq_{\sf m}{\sf P}\psi$, and $\varphi\leq_{\sf dm}\psi$ implies ${\sf P}\varphi\leq_{\sf dm}{\sf P}\psi$.
\end{proof}

Using this lemma, we can obtain some reductions that cannot be obtained immediately from the absorption relation.

\begin{prop}\label{prop:reduction-AE-E-infty}
$\pair{\forall\exists}\leq_{\sf m}\pair{\exists^\infty}$ holds; hence $\pair{\bar{\sf Q}\forall\exists}\leq_{\sf m}\pair{\bar{\sf Q}\exists^\infty}$ for any quantifier-pattern $\qf{Q}$.
\end{prop}

\begin{proof}
For the first assertion, given $\bar{x}=(x_n)_{n\in\om}$, we construct $y=\eta(\bar{x})$ satisfying the following:
\[\forall m\exists k.\ x_m(k)\not=0\iff\exists^\infty t.\ y(t)\not=0.\]

We inductively construct $m[s]$.
First put $m[0]=0$.
Assume that $m[s]$ has already been constructed at stage $s$.
If $x_{m[s]}(k)\not=0$ for some $k\leq s$, then put $y(s)=1$ and $m[s+1]=m[s]+1$.
Otherwise, put $y(s)=0$ and $m[s+1]=m[s]$.
One can easily see that this gives the desired reduction.
Then the second assertion follows from Lemma \ref{lem:quantifier-add-lemma}.
\end{proof}

As mentioned in Fact \ref{fact:Sigma-2-sep}, the dual of the above result does not hold.
For direducibility, we need to weaken the statement as follows:

\begin{prop}\label{prop:ex-ex-infty-bm-reducible}
$\pair{\exists}\leq_{\sf dm}\pair{\exists^\infty}$ holds; hence $\pair{\bar{\sf Q}\exists}\leq_{\sf dm}\pair{\bar{\sf Q}\exists^\infty}$.
\end{prop}

\begin{proof}
For the first assertion, given $x$, we construct $y=\eta(x)$ satisfying the following:
\[\exists t.\ x(t)\not=0\iff \forall s\exists t\geq s.\ y(t)\not=0\]

To be explicit, put $y(t,u)=x(t)$.
Then obviously, the above equivalence holds, and the corresponding realizers can be easily obtained.
For the dual reduction, since the dual of the left side does not involve a witness (since it is a $\forall$-formula), it is sufficient to show that whenever $x$ satisfies the dual of the left side, we can obtain a witness for the dual of the right side.
Suppose that $x(t)=0$ for any $t$.
In this case, $y(t,u)=0$ for any $(t,u)$, so $s=0$ is a witness for the right-hand side.
Then the second assertion follows from Lemma \ref{lem:quantifier-add-lemma}.
\end{proof}

\begin{figure}[t]
\begin{center}
\includegraphics[width=60mm]{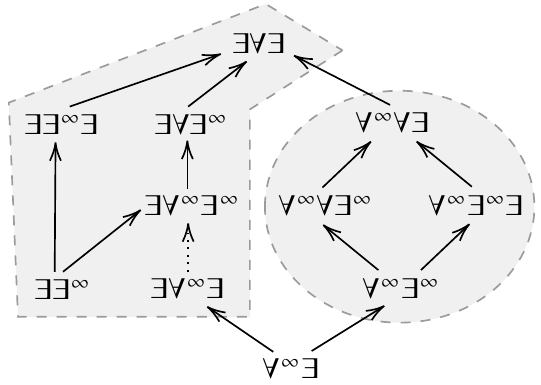}
\quad
\includegraphics[width=57mm]{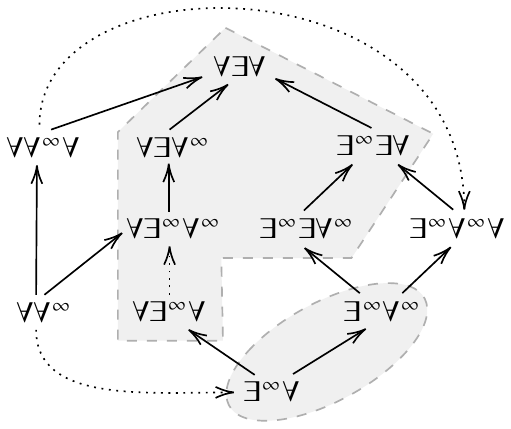}
\end{center}
\caption{The many-one classification of $\Sigma_3$- and $\Pi_3$-patterns}\label{fig:arith2}
\end{figure}

To make it easier to understand the discussion from here on, let us illustrate the relationship between $\Sigma_3$- and $\Pi_3$-patterns in advance as Figure \ref{fig:arith2}.
Here, the arrow $\bar{\sf P}\to\bar{\sf Q}$ implies $\comp{\bar{\sf P}}\leq_{\sf m}\comp{\bar{\sf Q}}$, regardless of whether it is a dotted line or a solid line.
Later we will see that all classes of quantifier-patterns that belong to a region enclosed by a polygon or ellipse are $\equiv_{\sf m}$-equivalent.

\section{Natural complete problems}\label{sec:natural-problem}

Here, we perform a detailed analysis of classical $\Sigma_3$- and $\Pi_3$-complete problems.

\subsection{Countable structure}\label{sec:countable-structure}
In this article, we often deal with countable structures such as countable partial orders and countable graphs.
Since they are all countable binary relations, we present how to handle countable binary relations.
A countable binary relation we deal with in this article is a pair $(X,R)$ where $X\subseteq\om$ and $R\subseteq X^2$.
A code for a binary relation $(X,R)$ is $(p,q)\in\om^\om\times\om^{\om\times\om}\simeq\om^\om$ such that $X=\{n\in\om:p(n)=1\}$ and $R=\{(a,b)\in\om^2:q(a,b)=1\}$.
From now on, a binary relation is always identified with its code.

The set of all finite sequences from $X$ is written as $X^{<\om}$.
A tree refers to a acyclic directed graph with a root, but in this article, a tree is treated as a subset of $\om^{<\om}$ that is closed under taking initial segment.

\subsection{$\forall\forall^\infty\forall$: Local finiteness}\label{sec:local-finiteness}

Let us take a closer look at various $\Pi_3$-complete problems.
Interestingly, the form that appears most often does not seem to be $\forall\exists\forall$.
First, let us take a look at $\forall\forall^\infty\forall$-formulas (where note that the quantifier-pattern $\forall\forall^\infty\forall$ is $\Pi_3$).
The following $\forall\forall^\infty\forall$-dicomplete formula is useful for gaining an intuition of $\forall\forall^\infty\forall$.

\begin{obs}\label{obs:all-bdd-dicomplete}
The following formula is $\forall\forall^\infty\forall$-dicomplete:
\[\forall n\exists k\forall t.\;x(n,t)\leq k.\]
\end{obs}

\begin{proof}
The boundedness for a sequence, ${\sf Bdd}\equiv\exists k\forall t.\;x(t)\leq k$, is $\forall^\infty\forall$-complete \cite{Kihara}.
For the dual, a witness for a $\forall\exists$-formula is always computable, so the dicompleteness follows automatically.
That is, ${\sf Bdd}\equiv_{\sf bm}\pair{\forall^\infty\forall}$.
By the same discussion as in Lemma \ref{lem:quantifier-add-lemma}, we get $\forall{\sf Bdd}\equiv_{\sf bm}\pair{\forall\forall^\infty\forall}$.
The formula in the assertion is nothing other than $\forall{\sf Bdd}$, so the proof is complete.
\end{proof}

The $\forall\forall^\infty\forall$-dicomplete formula described above, in a nutshell, expresses the property ``being bounded everywhere.''
Thus, typical examples of $\forall\forall^\infty\forall$-formulas are those related to {\em local finiteness}.

A partial order is locally finite if every interval contains finitely many elements.
A graph is locally finite if every vertex has finite degree.
A locally finite tree is often called a finitely branching tree.
To be precise, we consider the following properties for a partial order $P$, a graph $(V,E)$ and a tree $T$:
\begin{align*}
{\sf LocFin}_{\sf PO}:&\quad\forall a,b\in P\ \exists k\ |\{c\in P:a<_Pc<_Pb\}|\leq k.\\
{\sf LocFin}_{\sf G}:&\quad\forall a\in V\ \exists k\ |\{b\in V:(a,b)\in E\}|\leq k.\\
{\sf FinBranch}:&\quad\forall\sigma\in T\ \exists k\ |\{n\in\om:\sigma\fr n\in T\}|\leq k.
\end{align*}

Here, for a set $A$, $|A|\leq k$ denotes that the cardinality of $A$ is at most $k$.
That is, $|A|\leq k$ is an abbreviation of the following formula:
\[\forall c_0,c_1,\dots,c_k\in A\ \exists i,j\leq k\ (i\not=j\ \land\ c_i=c_j).\]

The above three formulas are all just $\Pi_3$ formulas in the classical sense.
Note that the $\exists k$ in these definitions can be replaced with $\forall^\infty k$, so they can be considered $\forall\forall^\infty\forall$-formulas.

\begin{prop}\label{prop:complete-locfin-PO}
${\sf LocFin}_{\sf PO}$ is $\forall\forall^\infty\forall$-dicomplete.
\end{prop}

\begin{proof}
As in the above argument, one can easily see ${\sf LocBdd}_{\sf PO}\leq_{\sf dm}\comp{\forall\forall^\infty\forall}$.
For $\forall\forall^\infty\forall$-dicompleteness, by Observation \ref{obs:all-bdd-dicomplete}, it remains to show $\forall{\sf Bdd}\leq_{\sf dm}{\sf LocFin}_{\sf PO}$.
In order to show this, given $x=(x_n)_{n\in\om}$, we construct a partial order $P=\eta(x)$ such that
\[
\forall n\exists k\forall t.\ x_n(t)\leq k\iff P\mbox{ is locally finite.}
\]

The construction proceeds as follows.
First, put the bottom element $\bot$ and infinitely many pairwise incomparable elements $\{a_n\}_{n\in\om}$ in $P$.
Moreover, for any $n$, if $x_n(t)\geq k$ for some $t$, then insert $c^n_{k,t}$ between $\bot$ and $a_n$ for the least such $t$.
Here, if $\langle k,t\rangle\not=\langle \ell,s\rangle$ then $c^n_{k,t}$ is incomparable with $c^n_{\ell,s}$.
Formally, $\bot,a_n,c^n_{k,t}$ can be coded as $\pair{0,0,0,0},\pair{1,n,0,0},\pair{2,n,k,t}$, respectively, for example.
This construction gives a reduction $\eta$ for $\forall\forall^\infty\forall$-dicompleteness.

Let us analyze this construction.
The only comparable elements in $P$ are $\bot<_Pc^n_{k,t}<_Pa_n$, and all others are incomparable.
Moreover, there is nothing in the intervals $[\bot,c^n_{k,t}]$ and $[c^n_{k,t},a_n]$ except for the endpoints, so the only intervals that can contain multiple elements are $[\bot,a_n]$.
Thus, given $(p,q)$ is of the form $(\bot,a_n)$, only elements of the form $c^n_{k,t}$ are enumerated in the interval $[\bot,a_n]$, but at most one such element is enumerated for each $k$.
Therefore, that $k$ is an upper bound for $x_n$ is equivalent to that the cardinality of the interval $(\bot,a_n)$ is less than or equal to $k$.

In order to show $\forall\forall^\infty\forall$-completeness, let $n\mapsto k_n$ be a witness for $\forall{\sf Bdd}(x)$.
We need to find a witness for local finiteness of $P=\eta(x)$.
Given $p,q\in P$, if $(p,q)$ is not of the form $(\bot,a_n)$, then there is nothing in this interval as discussed above, so $u_{p,q}=0$ is an upper bound for the cardinality of this interval.
If $(p,q)$ is of the form $(\bot,a_n)$, the given upper bound $k_n$ for $x_n$ is also an upper bound for the cardinality of the interval $(\bot,a_n)$ as discussed above.
In this case, put $u_{p,q}=k_n$, and then $(p,q)\mapsto u_{p,q}$ is a witness for ${\sf LocFin}_{\sf PO}(P)$.

Conversely, let $(p,q)\mapsto u_{p,q}$ be a witness for ${\sf LocFin}_{\sf PO}(P)$.
In particular, for any $n$, $u_{\bot,a_n}$ is an upper bound for the cardinality of the interval $(\bot,a_n)$, which is also an upper bound for $x_n$ by the above argument.
Hence, $n\mapsto u_{\bot,a_n}$ is a witness for $\forall{\sf Bdd}(x)$.

Next we show the $\exists\exists^\infty\exists$-completeness of the dual $({\sf LocFin}_{\sf PO})^{\sf d}$.
Let $n$ be a witness for $(\forall{\sf Bdd})^{\sf d}(x)$.
This means that $x_n$ has no upper bound, so the interval $[\bot,a_n]$ has infinitely many elements; hence $(\bot,a_n)$ is an witness for $({\sf LocFin}_{\sf PO})^{\sf d}(P)$.

Conversely, let $(a,b)$ be a witness for $({\sf LocFin}_{\sf PO})^{\sf d}(P)$.
Then this must be of the form $(\bot,a_n)$ and such $n$ is a witness for $(\forall{\sf Bdd})^{\sf d}(x)$ by the above argument.
\end{proof}

\begin{prop}
${\sf LocFin}_{\sf G}$ is $\forall\forall^\infty\forall$-dicomplete.
\end{prop}

\begin{proof}
As before, ${\sf LocFin}_{\sf G}\leq_{\sf m}\comp{\forall\forall^\infty\forall}$.
In order to show $\forall{\sf Bdd}\leq_{\sf dm}{\sf LocFin}_{\sf G}$, given $x=(x_n)_{n\in\om}$, we construct a graph $G=\eta(x)$ such that
\[
\forall n\exists k\forall t.\ x_n(t)\leq k\iff G\mbox{ is locally finite.}
\]

The construction proceeds as follows.
First put a special vertex $u$ in $G=(V,E)$.
Moreover, for any $n$, if $x_n(t)\geq k$ for some $t$, then put a vertex $v^n_{k,t}$ which is adjacent only to $u$ for the least such $t$.
That is, $(u,v^n_{k,t})\in E$ for such $v^n_{k,t}$.
Formally, $u,v^n_{k,t}$ can be coded as $\pair{0,0,0,0},\pair{1,n,k,t}$, respectively, for example.
This construction gives a reduction $\eta$ for $\forall\forall^\infty\forall$-dicompleteness.
The proof of the dicompleteness is the same as before.
\end{proof}

In exactly the same way, one can also show the following.

\begin{prop}
${\sf FinBranch}$ is $\forall\forall^\infty\forall$-dicomplete.
\end{prop}

\subsection{$\forall\forall^\infty$: Local finiteness for codes}\label{sec:locally-finite-AAinfty}
Local finiteness discussed in Section \ref{sec:local-finiteness} can also be expressed in a different way.
Consider the following properties for a partial order $P$, a graph $(V,E)$, and a tree $T$:
\begin{align*}
{\sf LocCFin}_{\sf PO}:&\quad\forall a,b\in P\ \forall^\infty c\in P\ (a<_Pb\to \neg (a<_Pc<_Pb)).\\
{\sf LocCFin}_{\sf G}:&\quad\forall a\in V\ \forall^\infty b\in V.\ (a,b)\not\in E.\\
{\sf CFinBranch}:&\quad\forall\sigma\in T\ \forall^\infty n.\ \sigma\fr n\not\in T.
\end{align*}

The difference from Section \ref{sec:local-finiteness} is that a witness is not an upper bound for the cardinality, but an upper bound for the code.
The existence of an upper bound for the cardinality of a set $A$ is classically equivalent to the existence of an upper bound for the codes $\{\dot{a}\in\om:a\in A\}$ of the elements of $A$ (where $\dot{a}$ is a code of $a$), but the difficulty of finding them can be different.
Of course, the coded versions are a little unnatural, but let us analyze these notions as well.

\begin{prop}
${\sf LocCFin}_{\sf PO}$ is $\forall\forall^\infty$-dicomplete.
\end{prop}

\begin{proof}
Clearly, ${\sf LocCFin}_{\sf PO}$ is a $\forall\forall^\infty$-formula.
Hence, it suffices to show $\comp{\forall\forall^\infty}\leq_{\sf dm}{\sf LocCFin}_{\sf PO}$.
Given $x=(x_n)_{n\in\om}$, we construct a partial order $P=\eta(x)$:
\[
\forall n\exists s\forall t\geq s.\ x_n(t)=0\iff P\mbox{ is locally finite.}
\]

The construction is almost the same as in Proposition \ref{prop:complete-locfin-PO}.
First, put the bottom element $\bot$ and infinitely many pairwise incomparable elements $\{a_n\}_{n\in\om}$ in $P$.
Moreover, for any $n$, if $x_n(k)\not=0$, then insert $c^n_{k}$ between $\bot$ and $a_n$.
Now, the codes are important.
Formally, we assume that $\bot,a_n,c^n_{k}$ are coded as $\pair{0,0,0},\pair{1,n,0},\pair{2,n,k}$, respectively.
This construction gives a reduction $\eta$ for $\forall\forall^\infty\forall$-dicompleteness.

In order to show $\forall\forall^\infty$-completeness, let $n\mapsto s_n$ be a witness for $x\in\comp{\forall\forall^\infty}$; that is, $x_n(t)=0$ for any $t\geq s_n$.
Then take the largest $t<s_n$ such that $x_n(t)\not=0$.
Given $p,q\in P$, if $(p,q)$ is not of the form $(\bot,a_n)$, then there is nothing in this interval, so take any number $c_{p,q}=0$.
If $(p,q)$ is of the form $(\bot,a_n)$, the construction enumerate nothing into the interval after $c^n_t$, so let $c_{p,q}$ be the code of $c^n_t$; that is, $c_{p,q}=\pair{2,n,t}$.
Then $(p,q)\mapsto c_{p,q}$ is a witness for ${\sf LocCFin}_{\sf PO}(P)$.

Conversely, let $(p,q)\mapsto c_{p,q}$ be a witness for ${\sf LocCFin}_{\sf PO}(P)$.
This means that if $a\geq c_{p,q}$ then $a$ is not contained in the interval $(p,q)$.
For each $n$, a sufficiently large $s_n$ satisfies $c_{\bot,a_n}\leq\langle 2,n,s_n\rangle$.
This is a code of $c^n_{s_n}$, so $\bot<_Pc^n_{s_n}<_P a_n$ fails.
By definition, this implies that $x_n(t)\not=0$ for any $t\geq s_n$.
Hence, $n\mapsto s_n$ is a witness for $\comp{\forall\forall^\infty}(x)$.

Next we show the $\exists\exists^\infty$-completeness of the dual $({\sf LocCFin}_{\sf PO})^{\sf d}$.
Let $n$ be a witness for $\comp{\exists\exists^\infty}(x)$, which means that $x_n(t)\not=0$ for infinitely many $t$.
Then the interval $[\bot,a_n]$ contains infinitely many elements, so $(\bot,a_n)$ is a witness for $({\sf LocCFin}_{\sf PO})^{\sf d}(P)$.

Conversely, let $(p,q)$ be a witness for $({\sf LocCFin}_{\sf PO})^{\sf d}(P)$.
Then this must be of the form $(\bot,a_n)$ and such $n$ is a witness for $x\in\comp{\exists\exists^\infty}$ as in Proposition \ref{prop:complete-locfin-PO}.
\end{proof}

\begin{prop}
${\sf LocFinite}({\sf G})$ is $\forall\forall^\infty$-dicomplete.
\end{prop}

\begin{proof}
As before, ${\sf LocCFin}_{\sf G}\leq_{\sf m}\comp{\forall\forall^\infty\forall}$.
In order to show $\comp{\forall\forall^\infty}\leq_{\sf dm}{\sf LocCFin}_{\sf G}$, given $x=(x_n)_{n\in\om}$, we construct a graph $G=\eta(x)$ such that
\[
\forall n\exists s\forall t\geq s.\ x_n(t)=0\iff G\mbox{ is locally finite.}
\]

The construction proceeds as follows.
First put a special vertex $u$ in $G=(V,E)$.
Moreover, for any $n$, if $x_n(k)\not=0$, then put a vertex $v^n_{k}$ which is adjacent only to $u$.
That is, $(u,v^n_{k})\in E$ for such $v^n_{k}$.
Formally, we assume that $u,v^n_{k}$ are coded as $\pair{0,0,0},\pair{1,n,k}$, respectively.
This construction gives a reduction $\eta$ for $\forall\forall^\infty$-dicompleteness.
The proof of the dicompleteness is the same as before.
\end{proof}

In exactly the same way, one can also show the following.

\begin{prop}
${\sf FinBranch}$ is $\forall\forall^\infty$-dicomplete.
\end{prop}

\subsection{$\forall\exists!\forall$: Universality}
As an example of a $\forall\forall^\infty$-complete formula, in Section \ref{sec:locally-finite-AAinfty}, we have given an alternative presentation of local finiteness, but it is somewhat unnatural, as it relies on how to code a structure (i.e., it uses the order on the codes).
Therefore, here we give another natural example of a $\forall\forall^\infty$-complete formula.
What we present here is a formula related to some kind of universality.
However, this does not appear in the form of a $\forall\forall^\infty$-formula, but rather in the form of a $\forall\exists!\forall$-formula, where $\exists!$ expresses the unique existence.
Let us look at a concrete example.

A partial order $P$ is a lattice if any two elements $a,b\in P$ have the infimum $a\land b$ and the supremum $a\lor b$.
Formally, this is the following formula:
\begin{align*}
{\sf Lattice}:\ \forall a,b\in P\exists c,d\in P\forall e\in P&\;[c\leq_Pa,b\leq_Pd\;\land\;\\
&(e\leq_P a,b\to e\leq_P c)\;\land\;(a,b\leq_P e\to d\leq_P e)].
\end{align*}

The elements $c$ and $d$ in the above formula represent the infimum $a\land b$ and the supremum $a\lor b$, respectively.
The key point of this formula is that the infimum and supremum are unique if they exist.
Thus, the existential quantifier $\exists$ in the above formula can be replaced with the unique existential quantifier $\exists!$.
In other words, the formula expressing a lattice is not only of the $\forall\exists\forall$-type, but also of the $\forall\exists!\forall$-type.

\begin{definition}
A $\forall\exists\forall$-formula $\forall a\exists b\forall c\,\theta(a,b,c,x)$ fulfills the $\forall\exists!\forall$-condition if the following holds:
\[\forall a\exists b\forall c\,\theta(a,b,c,x)\leftrightarrow\forall a\exists! b\forall c\,\theta(a,b,c,x)\]
\end{definition}

%
%
%
%

\begin{lemma}\label{lem:unique-witness-property}
If a $\forall\exists\forall$-formula $\varphi$ fulfills the $\forall\exists!\forall$-condition, then $\varphi\leq_{\sf bm}\comp{\forall\forall^\infty}$.
\end{lemma}

\begin{proof}
Assume that $\varphi(x)\equiv\forall n\exists k\forall \ell\,\theta(n,k,\ell,x)$ fulfills the $\forall\exists!\forall$-condition.
To show $\varphi\leq_{\sf bm}\comp{\forall\forall^\infty}$, given $x$, we construct $y=\eta(x)$ such that
\begin{align*}
\forall n\exists k\forall \ell.\ \theta(n,k,\ell,x)\iff\forall n\forall^\infty t.\ y_n(t)=0.
\end{align*}

The construction proceeds as follows.
First, $y_n$ has a unique guess $k_s$ for $k$ for each stage $s$.
We first set $k_0=0$.
We inductively assume that $k_s$ has been constructed.
If the guess is found to be wrong at stage $s$, that is, if there exists $\ell\leq s$ such that $\theta(n,k_s,\ell,x)$ is false, then put $y_n(s)=1$ and $k_{s+1}=k_s+1$.
If the guess is still correct, then put $y_n(s)=0$ and $k_{s+1}=k_s$.

Let $n\mapsto k(n)$ be a witness for $\varphi(x)$.
By the $\forall\exists!\forall$-condition, $k(n)$ is the only one that makes $\theta(n,k,\ell,x)$ true for any $\ell$.
Since other guesses are incorrect, we will eventually reach $k_{s_n}=k(n)$.
For the least such $s_n$, the guess $k_{s_n}$ cannot be refuted after this stage, so we have $y_n(t)=0$ for any $t\geq s$.
Therefore, $n\mapsto s_n$ is a witness for $\comp{\forall\forall^\infty}(y)$.

Conversely, let $n\mapsto s_n$ be a witness for $\comp{\forall\forall^\infty}(y)$; that is, $y_n(t)=0$ for any $t\geq s_n$.
This means that the guess $k_{s_n}$ is not refuted after stage $s_n$, so $\theta(n,k_s,\ell,x)$ holds for any $\ell$.
Therefore, $n\mapsto k_{s_n}$ is a witness for $\varphi(x)$.

The same argument applies to the dual.
\end{proof}

\begin{theorem}
${\sf Lattice}$ is $\forall\forall^\infty$-dicomplete.
\end{theorem}

\begin{proof}
As seen above, ${\sf Lattice}$ fulfills the $\forall\exists!\forall$-condition.
Thus, by Lemma \ref{lem:unique-witness-property}, we get ${\sf Lattice}\leq_{\sf dm}\comp{\forall\forall^\infty}$.
It remains to show the converse direction $\comp{\forall\forall^\infty}\leq_{\sf dm}{\sf Lattice}$.
Given $x=(x_n)_{n\in\om}$ we construct a partial order $P=\eta(x)$ such that
\[
\forall n\exists s\forall t\geq s.\ x_n(t)=0\iff P\mbox{ is a lattice.}
\]

The construction proceeds as follows.
First put the bottom element $\bot$, the top element $\top$ and infinitely many pairwise incomparable elements $\{a_n,b_n\}_{n\in\om}$ in $P$.
Moreover, for each $n$, insert an increasing sequence $C_n=\{c^n_k:x_n(k)\not=0\}$ between $\bot$ and $\{a_n,b_n\}$; that is,
 if  $c^n_k,c^n_\ell$ are defined for $k<\ell$, then
\[\bot<_Pc^n_k<_Pc^n_\ell<_Pa_n,b_n.\]

This construction gives a reduction $\eta$ for $\forall\forall^\infty$-dicompleteness.
Let us analyze this construction.
We call $\{a_n,b_n\}\cup C_n$ the $n$th block.
For a pair $\{u,v\}$ of elements in $P$, if $u$ and $v$ belong to different blocks, then $u\land v=\bot$ and $u\lor v=\top$, respectively.
Assume that $u$ and $v$ both belong to the $n$th block.
If $\{u,v\}\not=\{a_n,b_n\}$, then $u$ and $v$ are comparable, and one can check which one, $u\leq_P v$ or $v\leq_P u$, holds true.
If the former holds, $u\land v=u$ and $u\lor v=v$, and if the latter holds, $u\land v=v$ and $u\lor v=u$.
If $\{u,v\}=\{a_n,b_n\}$, then $a_n\lor b_n=\top$, so only $a_n\land b_n$ depends on $x$.
For the infimum $a_n\land b_n$, only the increasing sequence $C_n$ is inserted between $\bot$ and $\{a_n,b_n\}$, so $a_n\land b_n$ is the maximum element of $C_n$.
If $C_n$ is an infinite set, then there is no maximum element, so $a_n\land b_n$ also do not exist.
Therefore, that $P$ is a lattice is equivalent to that $\{a_n,b_n\}$ has a infimum for any $n$, which is equivalent to that $C_n$ is finite for any $n$.

First, we show the $\forall\forall^\infty$-completeness of ${\sf Lattice}$.
Let $n\mapsto s_n$ be a witness for $\comp{\forall\forall^\infty}(x)$; that is, $x_n(t)=0$ for any $t\geq s_n$.
By checking the values of $x_n(t)$ for $t\leq s_n$, we can compute the smallest witness $s_n'\leq s_n$.
In this case, the largest element in $C_n$ is $c^n_{s_n'-1}$.
Therefore, the infimum of $\{a_n,b_n\}$ is $c^n_{s_n'-1}$.
Conversely, if a witness for $P$ being a lattice is given, in particular, we have $(a_n,b_n)\mapsto a_n\land b_n$.
As discussed above, such an $a_n\land b_n$ is in the form of $c^n_{k(n)}$, and $c^n_{k(n)}$ is the maximum element of $C_n$, so $n\mapsto k(n)$ gives a witness for $\comp{\forall\forall^\infty}(x)$.

Next, we show that the $\exists\exists^\infty$-completeness of the dual ${\sf Lattice}^{\sf d}$.
Let $n$ be a witness for $x\in\comp{\exists\exists^\infty}$; that is, there are infinitely many $t$ such that $x_n(t)\not=0$.
In this case, $C_n$ is an infinite set, so there is no infimum of $\{p_n,q_n\}$.
Therefore, $\{p_n,q_n\}$ is a witness for ${\sf Lattice}^{\sf d}(P)$.
Conversely, if a witness for ${\sf Lattice}^{\sf d}(P)$ is given, it is always in the form of $\{p_n,q_n\}$.
Recall that $\{p_n,q_n\}$ is a witness for ${\sf Lattice}^{\sf d}(P)$ iff $C_n$ is an infinite set, which means that there are infinitely many $t$ such that $x_n(t)\not=0$.
Therefore, $n$ is witness for $\comp{\exists\exists^\infty}(x)$.
\end{proof}

\subsection{$\forall\forall^\infty$: Verifiability}
Next, let us look at an example that does not look like a $\forall\forall^\infty$-type at first glance, but actually is.

A partial order $P$ is atomic if every non-bottom element $a\in P$ bounds a minimal element.
Formally, this is the following formula:
\[{\sf Atomic}:\quad\forall a\exists b\forall c\;[a>_P\bot\to (\bot<_Pb\leq_P a\ \land\ \neg(\bot<_Pc<_Pb)].\]

Of course, the part $\exists b\forall c$ in the above formula cannot be replaced with either $\exists b\forall c\geq b$ or $\exists! b\forall c$, so at first glance, this is neither the $\forall\forall^\infty$-type nor the $\forall\exists!\forall$-type.
Now the key observation is that this problem has the property of verifiability:
If we have at least one witness $a\mapsto b_a$ for ${\sf Atomic}(x)$, then for any pair $(a',b')$, one can verify whether or not $b'$ is a witness at $a'$ (that is, $b'$ is a minimal element below $a'$).

\begin{definition}
A $\forall\exists\forall$-formula $\varphi\equiv\forall a\exists b\forall c\ \theta(a,b,c,x)$ is {\em verifiable} if there exists a partial computable function ${\sf verify}$ such that for any witness $w$ for $x\in\varphi$ the following holds for any $a$ and $b$:
\[
{\sf verify}(w,a,b,x)=
\begin{cases}
1&\mbox{ if }\forall c\ \theta(a,b,c,x),\\
0&\mbox{ otherwise.}
\end{cases}
\]
\end{definition}

\begin{example}\label{exa:atomic-verifiable}
${\sf Atomic}$ is verifiable.
To see this, let $w$ be a witness for ${\sf Atomic}(x)$.
If $b$ is minimal, $w(b)=b$.
If $b$ is not minimal, then either $b=\bot$ or $w(b)<b$.
Therefore, $b$ is a witness at $a$ (that is, $b$ is a minimal element below $a$) iff either $a=\bot$ or $w(a)=a=b$ or $\bot<_Pw(b)=b<_Pa$.
This decision is clearly computable.
\end{example}

\begin{lemma}\label{lem:verifiable-reducible}
If a $\forall\exists\forall$-formula $\varphi$ is verifiable, then $\varphi\leq_{\sf dm}\comp{\forall\forall^\infty}$.
\end{lemma}

\begin{proof}
Let $\varphi$ be of the form $\forall n\exists m\forall k. x(n,m,k)=0$. 
Slightly modifying the reduction $\eta$ in the proof of Proposition \ref{prop:reduction-AE-E-infty}, given $x$, one can construct $y=\eta(x)$ such that
\[\forall n\exists m\forall k.\ x(n,m,k)=0\iff\forall n\forall^\infty t.\ y_n(t)=0.\]

Recall that this reduction is defined by the process that, given $n$, searches for the least $m$ such that $x(n,m,k)=0$ for any $k$.
Hence, if $(s_n)_{n\in\om}$ is a witness for $\comp{\forall\forall^\infty}(y)$, then the $m$ seen at the $s_n$th stage of the search actually provides the least witness at $n$.
Conversely, if $\bar{m}=(m_n)_{n\in\om}$ is a witness for $x\in \varphi$, then for each $n'$, search for the least $m'$ such that ${\sf verify}(\bar{m},n',m',x)=1$.
Then simulate the construction of $y=\eta(x)$ until the stage $s_{n'}$ at which the search process arrives at $m'$.
Then $n'\mapsto s_{n'}$ is a witness for $\comp{\forall\forall^\infty}(y)$.

For the dual, we only need to consider the $n$ part, but since the reduction is a parallel process for each $n$, it is easy to verify that $n$ is a witness for $\varphi^{\sf d}(x)$ iff it is a witness for $\comp{\exists\exists^\infty}(\eta(x))$.
\end{proof}

\begin{theorem}
${\sf Atomic}$ is $\forall\forall^\infty$-dicomplete.
\end{theorem}

\begin{proof}
By Example \ref{exa:atomic-verifiable} and Lemma \ref{lem:verifiable-reducible}, we have ${\sf Atomic}\leq_{\sf dm}\pair{\forall\forall^\infty}$.
In order to show $\pair{\forall\forall^\infty}\leq_{\sf dm}{\sf Atomic}$, given $x=(x_n)_{n\in\om}$, we construct a partial order $P=\eta(x)$ such that
\[
\forall n\exists s\forall t\geq s.\ x_n(t)=0\iff P\mbox{ is atomic.}
\]

First put the bottom element $\bot$ in $P$.
For each $n,s$, define $P_{n,s}=\{t\geq s:x_n(t)\not=0\}$.
Then add the following elements to $P$.
\[\{p_{n,s_1,t_1,\dots,s_\ell,t_\ell,s}:n,s\in\om\mbox{ and }t_i\in P_{n,s_i}\mbox{ for each }i\leq\ell\}.\]

Here, if $\sigma$ and $\tau$ are incomparable, then $p_\sigma$ and $p_\tau$ are also incomparable, and if $\tau$ properly extends $\sigma$, we put $\bot<_Pp_\tau<_Pp_\sigma$.

Now, let $n\mapsto s_n$ be a witness for $\comp{\forall\forall^\infty}(x)$; that is, $x_n(t)=0$ for any $t\geq s_n$.
To find a witness for $P$ being atomic, let $p=p_{n,\sigma,s}\in P$ be given.
If $s\geq s_n$ then $P_{n,s}=\emptyset$, so $a(p):=p_{n,\sigma,s}$ is a minimal element.
If $s<s_n$ then check if $x_n(t)=0$ for any $s\leq t<s_n$.
If so, we know $P_{n,s}=\emptyset$; hence $a(p):=p_{n,\sigma,s}$ is a minimal element.
Otherwise, take $t$ such that $s\leq t<s_n$ and $x_n(t)\not=0$.
The latter means $t\in P_{n,s}$, so we get $\bot<_Pp_{n,\sigma,s,t,s_n}<_Pp_{n,\sigma,s}$, and $p_{n,\sigma,s,t,s_n}$ is a minimal element as in the above argument.
Then put $a(p):=p_{n,\sigma,s,t,s_n}$.
In any case, $\bot<_Pa(p)\leq_Pp$ and $a(p)$ is a minimal element.
Hence, $p\mapsto a(p)$ witnesses that $P$ is atomic.

Conversely, let $p\mapsto a(p)$ witness that $P$ is atomic.
Then, given $n$, compute $a(p_{n,0})$.
If $a(p_{n,0})=p_{n,0}$ then $p_{n,0}$ is minimal, which means $P_{n,0}=\emptyset$, so $x_n(t)=0$ for any $t$.
Then put $s_n=0$.
Otherwise, $a(p_{n,0})$ is of the form $p_{n,\sigma,s}$, which must be minimal, so this means $P_{n,s}=\emptyset$; that is, $x_n(t)=0$ for any $t\geq s$.
Then put $s_n=s$.
In any case, $n\mapsto s_n$ is a witness for $\comp{\forall\forall^\infty}(x)$.

For the dual, let $n$ be a witness for $\comp{\exists\exists^\infty}(x)$; that is, $x_n(t)\not=0$ for infinitely many $t$.
In particular, $P_{n,s}\not=\emptyset$ for any $s$.
In particular, there is no minimal element below $p_{n,0}$.
Hence $p_{n,0}$ is a witness for ${\sf Atomic}^{\sf d}(P)$.

Conversely, let $p_{n,\sigma,s}$ be a witness for ${\sf Atomic}^{\sf d}(P)$.
Then there is no minimal element below $p_{n,\sigma,s}$.
In particular, some $t\in P_{n,s}$ exists.
Then $p_{n,\sigma,s,t,s'}<_Pp_{n,\sigma,s}$ for any $s'$.
By the assumption, this is not minimal, too, so $P_{n,s'}\not=\emptyset$.
Therefore, for any $s'$ there exists $t\geq s'$ such that $x_n(t)\not=0$.
This means that $n$ is a witness for $\comp{\exists\exists^\infty}(x)$.
\end{proof}


\subsection{$\forall\exists\forall$-dicomplete}
All of the $\Pi_3$-problems discussed so far in Section \ref{sec:natural-problem} are actually $\Pi_3$-complete in the classical sense, whereas they are not $\forall\exists\forall$-complete under $\leq_{\sf m}$.
Therefore, it would be interesting to see whether there are natural $\forall\exists\forall$-(di)complete problems. 

As an example of a $\forall\exists\forall$-dicomplete problem, we here propose the complementedness problem for partial orders.
We assume that a poset $P$ is bounded; that is, it has the top element $\top$ and the bottom element $\bot$.
A complement of an element $a\in P$ is an element $b\not=a$ such that $a\lor b=\top$ and $a\land b=\bot$.
To be precise, there is no $c$ such that either $a,b\leq_Pc<_P\top$ or $\bot<_Pc\leq_Pa,b$ holds.
Such an element $b$ is not necessarily unique.
A bounded poset $P$ is complemented if every $a\in P$ has a complement.
Formally,
\[{\sf Compl}:\quad\forall a\in P\,\exists b\in P\,\forall c\in P\;\neg [(a,b\leq_Pc<_P\top)\lor(\bot<_Pc\leq_Pa,b)].\]

\begin{theorem}
${\sf Compl}$ is $\forall\exists\forall$-dicomplete.
\end{theorem}

\begin{proof}
Before starting the proof, let us explain the basic idea of our construction.
We first put the order $q<_Sp_0,p_1$ on the set $S=\{p_0,p_1,q\}$ of three symbols.
Let $Q$ be the set of all finite subsets of $\om$ ordered by inclusion.
Then, let us analyze the structure of the product order $S\times Q$.  
\[
(u,D)\leq_{S\times Q}(v,E)\iff u\leq_Sv\mbox{ and }D\subseteq E.
\]

This order has the least element $(q,\emptyset)$, which is written as $\bot$.
We write the poset obtained by adding the top element $\top$ to $S\times Q$ as $P$.

Note that the complement of $(p_i,D)$ is $(p_{1-i},\emptyset)$.
This is because $p_i$ and $p_{1-i}$ are incomparable, so the common upper bound of $(p_i,D)$ and $(p_{1-i},\emptyset)$ is only $\top$, and moreover, $(p_{1-i},\emptyset)$ only bounds $\bot=(q,\emptyset)$.
If $D\not=\emptyset$, $(q,D)$ does not have a complement.
This is because, for any $(u,E)$, the non-top element $(u,D\cup E)<_P\top$ bounds $(q,D)$ and $(u,E)$.
Indeed, this is the least upper bound, so this construction always gives a lattice.

Now,  given $x$, we construct $\eta(x)=P$ such that
\[\forall a\exists b\forall t.\ x(a,b,t)=0\iff P\mbox{ is complemented.}\]

As a modification of the above construction, we prepare for a new symbol $\ep^a_b$ for each $a,b$.
Here, we assume that $\ep^a_b$ is coded in a way that allows $a,b\in\om$ to be extracted.
We say that a finite set $E\subseteq\om$ is an {\em $(a,b)$-refuter} if either $a\not\in E$ or $x(a,b,t)\not=0$ for some $t$.
In this case, let $t_{a,b}$ be the least $t$ such that $x(a,b,t)\not=0$, and put $E^a_b=\pair{E,t_{a,b}}$.
Here, if $a\not\in E$ then $E^a_b=\pair{E,0}$.
Note that if both $D$ and $E$ are $(a,b)$-refuters then any subset of $D\cup E$ is an $(a,b)$-refuter.
In particular, the empty set $\emptyset$ is always an $(a,b)$-refuter.
Then we consider the following set:
\[{\sf Ref}=\{(\ep^a_b,D^a_b):D\mbox{ is an $(a,b)$-refuter}\}.\]

The decision of whether $E$ is an $(a,b)$-refuter or not is not necessarily $x$-computable, but the decision of $(u,\pair{E,t})\in{\sf Ref}$ is $x$-computable (using the information on $t$).
Now consider $S=\{q\}\cup\{\ep^a_b:a,b\in\om\}$.
Here, $q$ is the bottom element of $S$, and other elements of $S$ are pairwise incomparable.
A rough idea our proof strategy is to consider $S\times Q$ in the same way as before, but with a few corrections.  
\[P=(\{q\}\times Q)\cup{\sf Ref}\cup\{\top\}\]

The order on $\{q\}\times Q$ is the product order.
The order on the remaining parts are generated by the following relations:
\begin{align*}
&D\subseteq E\implies(q,D)\leq_P(\ep^a_b,D^a_b)\leq_{P}(\ep^a_b,E^a_b),\\
&u\in P\implies u\leq_P\top.
\end{align*}

Here, $D$ and $E$ in the first line are $(a,b)$-refuters.
Then let us analyze a complement of each element of $P$.
If $(\ep^a_b,D^a_b)\in{\sf Ref}$, then its complement is $(\ep^c_d,\emptyset^c_d)$ for $(c,d)\not=(a,b)$.
Here, as noted above, $\emptyset$ is always a $(c,d)$-refuter, so $(\ep^c_d,\emptyset^c_d)\in{\sf Ref}\subseteq P$.
Note that $\ep^a_b$ and $\ep^c_d$ are incomparable for $(a,b)\not=(c,d)$, so the upper bound of $(\ep^a_b,D^a_b)$ and $(\ep^c_d,\emptyset^c_d)$ is only $\top$, and the lower bound is only $\bot=(q,\emptyset)$.

Next consider $(q,D)$ for $D\not=\emptyset$.
First, an element of the form $(q,E)$ cannot be a complement of $(q,D)$, since they have an upper bound $(q,D\cup E)$.
Therefore, we only need to discuss whether an element of the form $(\ep^a_b,E^a_b)\in{\sf Ref}$ is a complement.

If $D$ is not an $(a,b)$-refuter, then we must have $a\in D$.
Also, any $C\supseteq D$ is not an $(a,b)$-refuter, so in particular, $(\ep^a_b,C^a_b)$ is undefined.
Therefore, the upper bound of $(q,D)$ and $(\ep^a_b,E^a_b)$ is only $\top$.
The infimum is $(q,D\cap E)$.
Thus, the infimum of $(q,D)$ and $(\ep^a_b,\emptyset^a_b)$ is $\bot=(q,\emptyset)$, where  we have $(\ep^a_b,\emptyset^a_b)\in{\sf Ref}$ since $\emptyset$ is always an $(a,b)$-refuter as noted above.
Hence, $(\ep^a_b,\emptyset^a_b)$ is a complement of $(q,D)$.

If $D$ is an $(a,b)$-refuter, then $(q,D)$ and $(\ep^a_b,E^a_b)$ have an upper bound $(\ep^a_b,(D\cup E)^a_b)<_P\top$, since the union of two refuters is again a refuter as noted above.
Hence, $(\ep^a_b,E^a_b)$ is not a complement of $(q,D)$.
Consequently, $(q,D)$ has no complement iff $D$ is an $(a,b)$-refuter for any $a,b$ (where we only need to consider $a\in D$).

Let $a\mapsto b_a$ be a witness for $\comp{\forall\exists\forall}(x)$; that is, $x(a,b_a,t)=0$ for any $a,t$.
To find a witness for $P=\eta(x)$ being complemented, it suffices to find a complement of each element of the form $(q,D)\in P$.
For each $a\in D$, $D$ is not an $(a,b_a)$-refuter, so we can take $(\ep^a_{b_a},\emptyset^a_{b_a})$ as a complement of $(q,D)$ as discussed above.

Conversely, assume that a witness for ${\sf Compl}(P)$ is given.
In particular, we have information on a complement of each $(q,\{a\})$.
As discussed above, $(q,\{a\})$ has a complement only if $\{a\}$ is not an $(a,b)$-refuter for some $b$.
Moreover, its complement is necessarily of the form $(\ep^a_b,E^a_b)$, and $(a,b)$ can be extracted from $\ep^a_b$.
Since $\{a\}$ is not an $(a,b)$-refuter, we have $x(a,b,t)=0$ for any $t$.
Putting $b_a=b$, the map $a\mapsto b_a$ gives a witness for $\comp{\forall\exists\forall}(x)$.

For the dual, let $a$ be a witness for $\comp{\exists\forall\exists}(x)$.
Then for any $b$ there is $t$ such that $x(a,b,t)\not=0$, so in particular, $\{a\}$ is an $(a,b)$-refuter for any $b$.
Hence, $(q,\{a\})$ has no complement in $P$.

Conversely, assume that a witness for $P$ not being complemented is given.
Such a witness must be of the form $(q,D)$ for some $D\not=\emptyset$.
By the above argument, $(q,D)$ has no complement only if $D$ is an $(a,b)$-refuter for any $(a,b)$.
This means that if $a\in D$ then, for any $b$, there is $t$ such that $x(a,b,t)\not=0$.
Therefore, for instance, $a=\min D$ is a witness for $\comp{\exists\forall\exists}(x)$.
\end{proof}

\subsection{$\forall^{\downarrow}\forall^\infty$: Asymptotic behavior}

Interestingly, there is also an example of a complete problem which is not dicomplete.
As a concrete example, let us consider the divergence problem ${\sf Diverge}$.
This is a problem that asks whether a given sequence of natural numbers $(x_n)_{n\in\om}$ diverges to infinity.
\[{\sf Diverge}:\quad\forall n\exists s\forall t\geq s.\ x_t\geq n.\]

The divergence problem is $\forall\forall^\infty$-complete, but the dual problem is not $\exists\exists^\infty$-complete, but $\forall^\infty\exists^\infty$-complete.

\begin{prop}\label{prop:diverge-AAinfty}
${\sf Diverge}$ is $\forall\forall^\infty$-complete.
\end{prop}

\begin{proof}
Clearly, ${\sf Diverge}$ is a $\forall\forall^\infty$-formula.
To show $\forall\forall^\infty$-completeness, given $x=(x_n)_{n\in\om}$, we construct a sequence $y=\eta(x)$ of natural numbers such that
\begin{align*}
\forall n\forall^\infty t.\ x_n(t)=0&\iff y=(y(s))_{s\in\om}\mbox{ diverges.}
\end{align*}

We construct $y$ as follows:
\[y(s)=\min\{n\leq s:x_n(s)\not=0\}\cup\{s\}.\]

Let $n\mapsto s_n$ be a witness for $\comp{\forall\forall^\infty}(x)$; that is, $x(t)=0$ for any $t\geq s_n$.
Then, for $s_n'=\max_{m<n}s_m$, we have $y(t)\geq n$ for any $t\geq s_n'$.
Therefore, $n\mapsto s_n'$ is a witness for divergence of $y$.

Conversely, let $n\mapsto t_n$ be a witness for divergence of $y$; that is, $y(s)\geq n$ for any $s\geq t_n$.
This means that $x_m(s)=0$ for any $m<n$.
In particular, for any $s\geq t_{n+1}$ we have $x_n(s)=0$, so $n\mapsto t_{n+1}$ is a witness for $\comp{\forall\forall^\infty}(x)$.
\end{proof}

\begin{prop}\label{prop:div-A8E8-comp}
${\sf Diverge}^{\sf d}$ is $\forall^\infty\exists^\infty$-complete.
\end{prop}

\begin{proof}
For each sequence $x=(x(s))_{s\in\om}$ of natural numbers, we have the following equivalences:
\[x=(x(s))_{s\in\om}\mbox{ diverges}\iff\forall n\forall^\infty s.\ x(s)\geq n\iff \exists^\infty n\forall^\infty s.\ x(s)\geq n.\]

This means ${\sf Diverge}\leq_{\sf m}\comp{\exists^\infty\forall^\infty}$, and thus ${\sf Diverge}^{\sf d}\leq_{\sf m}\comp{\forall^\infty\exists^\infty}$.
To show $\comp{\forall^\infty\exists^\infty}\leq_{\sf m}{\sf Diverge}^{\sf d}$, given $x=(x_n)_{n\in\om}$, we construct a sequence $y=\eta(x)$ of natural numbers such that
\begin{align*}
\forall^\infty n\exists^\infty t.\ x_n(t)\not=0&\iff y=(y(s))_{s\in\om}\mbox{ does not diverge.}
\end{align*}

For each $n,s$, we inductively define a parameter $c_n[s]$.
Put $c_n[0]=n$, and inductively assume that $c_n[s]$ is defined for each $n\leq s$.
Then check if the following holds for $c=c_n[s]$.
\begin{align}\label{equ:diverge}
\forall m\;(n\leq m\leq c\ \to\ |\{t\leq s:x_m(t)\not=0\}|\geq c).
\end{align}

Choose the least $n\leq s$ satisfying (\ref{equ:diverge}), and put $y(s)=n$, and $c_m[s+1]=c_m[s]+1$ for $n\leq m\leq s+1$.
If there is no such $n$, put $y(s)=s$.
For $m\leq s+1$, put $c_m[s+1]=c_m[s]$ if this has not yet been defined.
Clearly, $(n,s)\mapsto c_n[s]$ is computable.

To see $\forall^\infty\exists^\infty$-completeness of ${\sf Diverge}^{\sf d}$, let $n$ be a witness for $\comp{\forall^\infty\exists^\infty}(x)$; that is, for any $m\geq n$, there are infinitely many $t$ such that $x_n(t)\not=0$.
Hence, for any fixed $c$, the property (\ref{equ:diverge}) holds for almost all $s$.
Also, since the value of $c=c_n[s]$ does not change until this holds, (\ref{equ:diverge}) must hold.
Then, our action ensures $y(s)\leq n$.
This occurs for all $c$, which means that $y(s)\leq n$ for infinitely many $s$.
Therefore, $n+1$ is a witness for non-divergence of $y$.

Conversely, let $n$ be a witness for non-divergence of $y$; that is, $y(s)\leq n$ for infinitely many $s$.
This means that (\ref{equ:diverge}) holds infinitely many times for $n$.
Since the value of $c=c_n[s]$ increases each time (\ref{equ:diverge}) holds for $n$, for this to be the case, there must be infinitely many $t$ such that $x_m(t)\not=0$ for any $m\geq n$.
Hence, $n$ is a witness for $\comp{\forall^\infty\exists^\infty}(x)$.
\end{proof}

Although we need to use some separation results proven later, let us present the following results in advance.

\begin{cor}
${\sf Diverge}$ is not $\forall\forall^\infty$-dicomplete, and ${\sf Diverge}^{\sf d}$ is not $\forall^\infty\exists^\infty$-dicomplete.
\end{cor}

\begin{proof}
By Theorem \ref{thm:formula-class-collapse} and Theorem \ref{thm:separation3}, we have $\comp{\forall^\infty\exists^\infty}<_{\sf m}\comp{\exists\exists^\infty}$, so by Proposition \ref{prop:div-A8E8-comp}, ${\sf Diverge}^{\sf d}$ cannot be $\exists\exists^\infty$-complete, which verifies the latter assertion.
By Theorem \ref{thm:formula-class-collapse} and Proposition \ref{prop:separation1}, we have $\comp{\forall\forall^\infty}<_{\sf m}\comp{\exists^\infty\forall^\infty}$, so Proposition \ref{prop:diverge-AAinfty}, ${\sf Diverge}$ cannot be $\exists^\infty\forall^\infty$-complete, which verifies the former assertion.
\end{proof}

The question, therefore, is whether ${\sf Diverge}$ can be positioned as a dicomplete problem for some class of formulas.
The answer is, in a sense, affirmative, by viewing ${\sf Diverge}$ as the intersection of a descending sequences of $\forall^\infty$-formulas.

\begin{definition}
A formula $\psi(a,x)$ is {\em descending} if the following holds:
\[a'\leq a\mbox{ and }\psi(a,x)\implies \psi(a',x).\]

A $\forall^{\downarrow}\bar{\sf Q}$-formula is a formula of the form $\forall a\psi(a,x)$ for some descending $\bar{\sf Q}$-formula $\psi(a,x)$.
The dual of a $\forall^{\downarrow}\bar{\sf Q}$-formula is called a $\exists^{\uparrow}\bar{\sf Q}$-formula.
\end{definition}

\begin{theorem}
${\sf Diverge}$ is $\forall^{\downarrow}\forall^\infty$-dicomplete.
\end{theorem}

\begin{proof}
Clearly, ${\sf Diverge}$ is a $\forall^{\downarrow}\forall^\infty$-formula.
Let $\varphi(x)$ be a $\forall^{\downarrow}\forall^\infty$-formula, presented as $\forall n\forall^\infty t.x(n,t)=0$ satisfying the following descending condition:
\[n'\leq n\mbox{ and }\forall^\infty t.\ x(n,t)=0\implies\forall^\infty t.\ x(n',t)=0.\]

To show the $\forall^{\downarrow}\forall^\infty$-completeness of ${\sf Diverge}$, we can directly adopt the proof of Proposition \ref{prop:diverge-AAinfty} (since $\varphi$ is also a $\forall\forall^\infty$-formula); that is, consider the following $y$:
\[y(s)=\min\{n\leq s:x(n,s)\not=0\}\cup\{s\}.\]

For the dual, let $n$ be a witness for $\varphi^{\sf d}(x)$.
Then we have $x(n,t)\not=0$ for infinitely many $t$, so $y(t)\leq n$ for such $t$.
In particular, $n+1$ is a witness for non-divergence of $y$.

Conversely, let $n$ be a witness for non-divergence of $y$.
Then $y(t)<n$ for infinitely many $t$, so for such $t$, we get $x(m,t)\not=0$ for some $m<n$.
By the pigeon hole principle, there is $m<n$ such that $x(m,t)\not=0$ for infinitely many $t$.
By the contrapositive of the descending condition, we must have $x(n,t)\not=0$ for infinitely many $t$.
This means that $n$ is a witness for $\varphi^{\sf d}(x)$.
\end{proof}

In fact, there are many natural problems which are $\forall^{\downarrow}\forall^\infty$-dicomplete.
Most people first encounter $\Pi_3$ formulas in elementary analysis; for instance, formulas describing convergence, Cauchy sequences, etc.
Analyzing these formulas carefully, we realize that they are not just $\Pi_3$ formulas, but $\forall^{\downarrow}\forall^\infty$-formulas.
To be explicit, we consider the formula describing Cauchyness of a sequence $(x_n)_{n\in\om}$ of rational numbers:
\[
{\sf Cauchy}\colon \forall k\exists N\forall n,m\geq N.\ |x_n-x_m|\leq\frac{1}{k+1}
\]

Another example is related to the frequency of occurrence of $0$ and $1$ in a binary sequence.
For a binary string $\sigma\in 2^{<\om}$, the frequency of occurrence of $1$ in $\sigma$ is written as ${\sf Freq}_1(\sigma)$.
\[{\sf Freq}_1(\sigma)=\frac{\#\{k<n:\sigma(k)=1\}}{n}.\]

An infinite binary seqeunce $x\in 2^\om$ is simply normal in base $2$ if the frequencies of occurrence of $0$ and $1$ in $x\upto n$ converge to $\frac{1}{2}$; that is, $\lim_{n\to\infty}{\sf Freq}_1(x\upto n)=\frac{1}{2}$.
The asymptotic density of a set $A\subseteq \om$ is defined as follows:
\[\delta(A)=\lim_{n\to\infty}{\sf Freq}_1(\chi_A\upto n)=\lim_{n\to\infty}\frac{\#\{k<n:k\in A\}}{n}\]

For example, the formula expressing that the asymptotic density is $0$ is also $\Pi_3$.
These notions are formally expressed as follows:
\begin{align*}
{\sf SimpNormal}\colon &\quad\forall k\exists s\forall t\geq s.\ \left|{\sf Freq}_1(x\upto t)-\frac{1}{2}\right|\leq\frac{1}{k+1},\\
{\sf AsympDen}_0\colon &\quad\forall k\exists s\forall t\geq s.\ {\sf Freq}_1(x\upto t)\leq\frac{1}{k+1}.
\end{align*}

\begin{theorem}
${\sf Cauchy}$, ${\sf SimpNormal}$, ${\sf AsympDen}_0$ are all $\forall^{\downarrow}\forall^\infty$-dicomplete.
\end{theorem}

\begin{proof}
It is clear that all of these formulas are $\forall^\downarrow\forall^\infty$.

${\sf Diverge}\leq_{\sf dm}{\sf Cauchy}$:
Given $x=(x_n)_{n\in\om}$, we construct a sequence $y=\eta(x)$ of natural numbers such that
\begin{align*}
x=(x_t)_{t\in\om}\mbox{ diverges}\iff y=(y_t)_{t\in\om}\mbox{ is a Cauchy sequence.}
\end{align*}

Given $x=(x_t)\in\om^\om$, for each $t$, if the same value as $x_t$ has already appeared in $(x_s)_{s<t}$ an even number of times, then put $y_t=\frac{1}{2x_t+1}$, and if it has appeared an odd number of times, then $y_t=\frac{1}{2x_t+2}$.
In other words, we make a distinction between cases where $|\{s<t:x_s=x_t\}|$ is even or odd.
The values of the sequence $(y_t)_{t\in\om}$ are of the form $\frac{1}{2n+i}$ for $i\in\{1,2\}$, so if this converges, it is of the form $0$ or $\frac{1}{2n+i}$.
In order for the sequence to converge to $\frac{1}{2n+i}$, almost all of $y_t$ must be this value, but in this case, almost all of $x_t$ be the value $n$.
However, if $x_t$ stabilizes to $n$, then $y_t$ returns the values $\frac{1}{2n+1}$ and $\frac{1}{2n+2}$ alternately, so it does not have a limit.
Therefore, the sequence $(y_t)_{t\in\om}$ cannot converge to a value other than $0$.

Let $n\mapsto s_n$ be a witness for ${\sf Diverge}(x)$; that is, $x_t\geq n$ for any $t\geq s_n$.
In this case, we have $y_t\leq \frac{1}{2n+1}$ for any $t\geq s_n$.
Thus, $n\mapsto s_n$ is also a witness for ${\sf Cauchy}(y)$.
Conversely, let $n\mapsto s_n$ be a witness for ${\sf Cauchy}(y)$; that is, for any $k,\ell\geq s_n$ we have $|y_k-y_\ell|\leq \frac{1}{n}$.
In particular, $y$ has a limit, but as discussed above, in this case, $y$ must converge to $0$, so we get $y_t\leq \frac{2}{n}$ for any $t\geq s_n$.
Considering $t\geq s_{4(n+1)}$, since $\frac{1}{2x_t+2}\leq y_t\leq \frac{1}{2n+2}$, we get $x_t\geq n$.
Hence, $n\mapsto s_{4(n+1)}$ is a witness for ${\sf Diverge}(x)$.

For the dual, let $n$ be a witness for ${\sf Diverge}^{\sf d}(x)$; that is, $x_t<n$ for infinitely many $t$.
By the pigeon hole principle, there is $m<n$ such that $x_t=m$ for infinitely many $t$.
In this case, $(y_t)_{t\in\om}$ takes $\frac{1}{2m+1}$ and $\frac{1}{2m+2}$ for infinitely many $t$.
In particular, for any $N$, there exist $k,\ell\geq N$ such that
\[
|y_k-y_\ell|=\frac{1}{2m+1}-\frac{1}{2m+2}=\frac{1}{(2m+1)(2m+2)}>\frac{1}{(2n+1)(2n+2)}.
\]

Hence, $(2n+1)(2n+2)$ is a witness for ${\sf Cauchy}^{\sf d}(y)$.
Conversely, let $n$ be a witness for ${\sf Cauchy}^{\sf d}(y)$.
In this case, for $m$ with $n\leq 2m+1$, for any $N$, there exist $k,\ell\geq N$ such that $|y_k-y_\ell|>\frac{1}{n}\geq\frac{1}{2m+1}$.
If $x_k,x_\ell\geq m$ then our construction ensures $y_k,y_\ell\leq\frac{1}{2m+1}$, and in particular, $|y_k-y_\ell|\leq\frac{1}{2m+1}$, but this is impossible.
Therefore, either $x_k<m$ or $x_\ell<m$ holds.
Since $N$ is arbitrary, there are infinitely many such $k,\ell$; hence, $m$ is a witness for ${\sf Diverge}^{\sf d}(x)$.

${\sf Diverge}\leq_{\sf dm}{\sf AsympDen}_0$:
Given $x=(x_n)_{n\in\om}$, we construct a sequence $y=\eta(x)$ of natural numbers such that
\begin{align*}
x=(x_t)_{t\in\om}\mbox{ diverges}\iff \mbox{the asymptotic density of $\{n:y(n)=1\}$ is $0$.}
\end{align*}

Put $u(0)=1$ and $u(s+1)=(s!+1)\cdot u(s)$.
At step $0$, put $y(n)=0$ for any $n<u(2)$.
For $s\geq 1$, inductively assume that $y\upto u(s)$ has already been defined at the beginning of step $s$.
For $k=\min\{x_s+2,s\}$, note that $\frac{s!}{k}$ is a natural number since $k\leq s$.
Then, consider the extension $y\upto u(s+1)$ of $y\upto u(s)$ where the last $\frac{s!}{k}u(s)$ bits of $y\upto u(s+1)$ are all set to $1$ and the rest are set to $0$; that is,
\[
y(t)=
\begin{cases}
0&\mbox{ if }u(s)\leq t<\left(1+s!-\frac{s!}{k}\right)u(s)\\
1&\mbox{ if }\left(1+s!-\frac{s!}{k}\right)u(s)\leq t<(1+s!)u(s)
\end{cases}
\]

Now, let us calculate the number of occurrences of $1$ in $y\upto u(s+1)$.
In the above construction, exactly $\frac{s!}{k}u(s)$ many new $1$'s are added, and the number of $1$'s in $y\upto u(s)$ is at most $u(s)$, so we get the following inequality:
\[
\frac{s!}{k}u(s)\leq\#\{t<u(s+1):y(t)=1\}\leq u(s)+\frac{s!}{k}u(s)
\]

Therefore, the frequency of occurrence of $1$ is the result of dividing this by $u(s+1)=(s!+1)u(s)$.
Calculating this value, we first obtain the following for the left-hand side:
\[
\frac{s!}{k}u(s)\cdot\frac{1}{u(s+1)}=\frac{s!}{k(s!+1)}=\frac{s!}{ks!+k}>\frac{s!}{ks!+s!}=\frac{1}{k+1}.
\]

Next, for $k>1$ we get the following for the right-hand side:
\[
\left(u(s)+\frac{s!}{k}u(s)\right)\cdot\frac{1}{u(s+1)}
=\frac{k+s!}{k(s!+1)}
=\frac{(k-1)s!+(k-1)k}{(k-1)k(s!+1)}
<\frac{1}{k-1}.
\]

Here, note $(k-1)s!+(k-1)k\leq (k-1)s!+s!=ks!$ for the last inequality.
Summarizing the above, we obtain the following inequality:
\[\frac{1}{k+1}<{\sf Freq}_1(y\upto u(s+1))<\frac{1}{k-1}.\]

Also, in the construction of $y\upto u(s+1)$, $1$'s are added as a tail.
Hence, if $u(s)<t\leq u(s+1)$, then we get ${\sf Freq}_1(y\upto t)\leq{\sf Freq}_1(y\upto u(s+1))<\frac{1}{k-1}$.

Let $n\mapsto s_n$ be a witness for ${\sf Diverge}(x)$; that is, $x_s\geq n$ for any $s\geq s_n$.
Put $t_n=\max\{s_n,n+1\}$.
If $s\geq t_n$ then $k=\min\{x_s+2,s\}\geq n+1$.
By the above argument, if $u(s)<t\leq u(s+1)$ then we get ${\sf Freq}_1(y\upto t)<\frac{1}{k-1}\leq\frac{1}{n}$.
This holds for any $t>u(t_n)$.
Hence, $n\mapsto u(t_n)+1$ is a witness for ${\sf AsympDen}_0(y)$.

Let $n\mapsto v_n$ be a witness for ${\sf AsympDen}_0(y)$; that is, ${\sf Freq}_1(y\upto t)\leq\frac{1}{n}$ for any $t\geq v_n$.
Calculate $s_n$ such that $u(s_n)\geq v_n$.
In particular, we have ${\sf Freq}_1(y\upto u(s+1))\leq\frac{1}{n}$ for any $s\geq s_n$.
For $k=\min\{x_s+2,s\}$, by the above argument, we have $\frac{1}{k+1}<{\sf Freq}_1(y\upto u(s+1))\leq\frac{1}{n}$, so we get $n<k+1\leq x_s+3$.
This means $x_s\geq n-2$ for any $s\geq s_n$.
Hence, $n\mapsto s_{n+2}$ is a witness for ${\sf Diverge}(x)$.

Let $n$ be a witness for $x\in{\sf Diverge}^{\sf d}$; that is, $x_s<n$ for infinitely many $s$.
For such a large $s>n$, we have $k=\min\{x_s+2,s\}=n+2$.
By the above argument, we have ${\sf Freq}_1(y\upto u(s+1))<\frac{1}{k-1}=\frac{1}{n+1}$.
Hence, $n+1$ is a witness for ${\sf AsympDen}_0^{\sf d}(x)$.

Let $n$ be a witness for ${\sf AsympDen}_0^{\sf d}(x)$; that is, ${\sf Freq}_1(y\upto t)\geq\frac{1}{n}$ for infinitely many $t$.
For such $t$, take $s$ such that $u(s)<t\leq u(s+1)$.
By the above argument, we get ${\sf Freq}_1(y\upto t)\geq{\sf Freq}_1(y\upto u(s+1))$.
Also, for $k=\min\{x_s+2,s\}$, we have ${\sf Freq}_1(y\upto u(s+1))<\frac{1}{k-1}$.
Summarizing the above, we get $\frac{1}{n}\leq {\sf Freq}_1(y\upto u(s+1))<\frac{1}{k-1}$, so $k<n+1$.
If $t$ is sufficiently large, $s\geq n+1$.
In this case, $k\not=s$, so $k=x_s+2$ and thus, $x_s<n-1$.
Hence, $n-1$ is a witness for ${\sf Diverge}^{\sf d}(x)$.

${\sf AsympDen}_0\leq_{\sf dm}{\sf SimpNormal}$:
Given $x\in 2^\om$, consider the sequence $y=\eta(x)$ obtained by replacing the $(2n+1)$st occurrence of $0$ in $x$ with $1$.
If $a_t$ is the number of $1$'s in $x\upto t$, then the number of $1$'s in $y$ is approximately $\frac{t-a_t}{2}+a_t=\frac{t+a_t}{2}$.
To be precise, this value is not necessarily a natural number, but it only differs from the actual number by at most $\frac{1}{2}$.
Therefore, calculating the error between ${\sf Freq}_1(y\upto t)$ and $\frac{1}{2}$, we get the following:
\[
{\sf Freq}_1(y\upto t)-\frac{1}{2}\approx\frac{t+a_t}{2t}-\frac{1}{2}=\frac{a_t}{2t}=\frac{{\sf Freq}_1(x\upto t)}{2}.
\]

Here, $\approx$ is the equivalence with errors at most $\frac{1}{2t}$.
Using the above equation, it is easy to construct a transformation of witnesses for $\leq_{\sf dm}$.
\end{proof}

It may be useful to keep in mind that all of these $\forall^\downarrow\forall^\infty$-dicomplete problems are related to limits.


\subsection{$\forall^\infty\forall\exists$: Global boundedness}\label{sec:example-global-boundedness}
Next, let us look for a $\forall^\infty\forall\exists$-complete problem.
As a candidate, let us consider the problem ${\sf FinDiam}$ of determining whether the diameter of a graph is finite.
Here, the distance $d(u,v)$ between two vertices $u,v$ in a graph $G$ is the length of a shortest path connecting $u$ and $v$.
If there is no path connecting $u$ and $v$, then $d(u,v)$ is defined to be $\infty$.
Then the diameter of a graph $G=(V,E)$ is defined as $\max\{d(u,v):u,v\in V\}$.
The finite-diameter problem ${\sf FinDiam}$ is defined as follows:
\[{\sf FinDiam}:\quad\exists r\forall u,v\in V\exists\gamma\;[(\mbox{$\gamma$ connects $u$ and $v$})\land |\gamma|\leq r].\]

Here, $|\gamma|$ is the length of a path $\gamma$.

\begin{prop}\label{prop:findiam-reduction-complete}
${\sf FinDiam}$ is $\forall^\infty\forall\exists$-complete.
\end{prop}

\begin{proof}
First, in the definition of ${\sf FinDiam}$, the quantification $\exists r$ can be replaced with ${\forall^\infty r}$; thus, ${\sf FinDiam}\leq_{\sf dm}\comp{\forall^\infty\forall\exists}$.
Therefore, it remains to show that $\comp{\forall^\infty\forall\exists}\leq_{\sf m}{\sf FinDiam}$.
Given $x=(x_{n,m})_{n,m\in\om}$, we construct a graph $\eta(x)=(V,E)$ such that
\[
\exists r\forall n\geq r\forall m\exists t.\ x_{n,m}(t)\not=0\iff \mbox{the diameter of $(V,E)$ is finite.}
\]

The graph has vertices $\ep\in V$ and $a^{n,m}_s\in V$ for any $n,m\in\om$ and $s\leq n$.
For each $s<n$, put $(\ep,a^{n,m}_0),(a^{n,m}_s,a^{n,m}_{s+1})\in E$, which yields infinitely many paths of length $n+1$ connecting $\ep$ and $a^{n,m}_n$.
Moreover, if $x_{n,m}(t)\not=0$ for some $t$, then put $b^{n,m}_t\in V$ and $(\ep,b^{n,m}_t),(a^{n,m}_s,b^{n,m}_t)\in E$ for each $s\leq n$.

Let $W=\{(n,m):\exists t.\ x_{n,m}(t)\not=0\}$.
If $(n,m)\not\in W$ then the distance between $\ep$ and $a^{n,m}_n$ is $n+1$.
If $(n,m)\in W$ then the distance is $2$ via $b^{n,m}_t$ for some $t$, since $(\ep,b^{n,m}_t),(b^{n,m}_t,a^{n,m}_n)\in E$.
Indeed, the distance between any two vertices in $S_{n,m}:=\{\ep,a^{n,m}_n,b^{n,m}_t:x_{n,m}(t)\not=0\}$ is at most $2$.
Therefore, if $(n,m),(n',m')\in W$, the distance between any two vertices in $S_{n,m}\cup S_{n',m'}$ is at most $4$.

Let $r$ be a witness for $\comp{\forall^\infty\forall\exists}(x)$.
Then, for any $n\geq r$ and $m$, we have $(n,m)\in W$.
In this case, for any $n\geq r$, the distance between any two vertices in $S_{n,m}$ is at most $2$, and for $n<r$, the distance between any two vertices in $S_{n,m}$ is at most $r$.
Therefore, the distance between any two vertices in this graph is at most $2r$.
Hence, $2r$ is a witness for ${\sf FinDiam}(V,E)$.

Conversely, if $r$ is a witness for ${\sf FinDiam}(V,E)$, then the distance between any two vertices in this graph is at most $r$.
Thus, for any $n>r$ and $m$, we must have $(n,m)\in W$.
This implies that $r+1$ is a witness for $\comp{\forall^\infty\forall\exists}(x)$.
\end{proof}

However, there are difficulties with the completeness of the dual of ${\sf FinDiam}$.
For this reason, let us consider a slightly modified problem.
Note that, even if the diameter is infinite, there can be an upper bound of the diameter of any connected component.

The set of all paths of a graph $G=(V,E)$ is denote by ${\rm Path}_G$, and for a path $\gamma\in{\rm Path}_G$, its start point is denoted by $\gamma_{\rm start}$, and its end point is denoted by $\gamma_{\rm end}$.
The bounded-diameter problem ${\sf FinDiam}_{\sf conn}$ for connected components is defined as follows:
\[{\sf FinDiam}_{\sf conn}\colon\quad\exists r\forall\gamma\in{\rm Path}_G\exists\delta\in{\rm Path}_G\;[\{\gamma_{\rm start},\gamma_{\rm end}\}=\{\delta_{\rm start},\delta_{\rm end}\}\;\land\;|\delta|\leq r].\]

This formula expresses that if two vertices $a,b\in V$ belong to the same connected component (that is, connected by a finite path $\gamma$) then the distance between $a$ and $b$ is at most $r$ (that is, connected by a path $\delta$ of length at most $r$).

\begin{theorem}\label{Fin-diam-conn-complete}
${\sf FinDiam}_{\sf conn}$ is $\forall^\infty\forall\exists$-dicomplete.
\end{theorem}

\begin{proof}
First, in the definition of ${\sf FinDiam}_{\sf conn}$, the quantification $\exists r$ can be replaced with $\forall^\infty r$; thus, ${\sf FinDiam}_{\sf conn}\leq_{\sf bm}\comp{\forall^\infty\forall\exists}$.
Therefore, it remains to show that $\comp{\forall^\infty\forall\exists}\leq_{\sf dm}{\sf FinDiam}_{\sf conn}$.
Given $x=(x_{n,m})_{n,m\in\om}$, construct a graph $\eta(x)=(V,E)$ as follows:
\[
\exists r\forall n\geq r\forall m\exists t.\ x_{n,m}(t)\not=0\iff (V,E)\in{\sf FinDiam}_{\sf conn}.
\]

The graph has vertices $a^{n,m}_\ell$ for any $n,m\in\om$ and $\ell\leq n$.
For each $\ell<n$, put $(a^{n,m}_\ell,a^{n,m}_{\ell+1})\in E$, which yields infinitely many paths of length $n$ connecting $a^{n,m}_0$ and $a^{n,m}_n$.
Moreover, if $x_{n,m}(t)\not=0$ for some $t$, then put $b^{n,m}_t\in V$ and $(a^{n,m}_\ell,b^{n,m}_t)\in E$ for each $\ell\leq n$.

Let $W=\{(n,m):\exists t.\ x_{n,m}(t)\not=0\}$.
If $(n,m)\not\in W$ then the distance between $a^{n,m}_0$ and $a^{n,m}_n$ is $n$.
If $(n,m)\in W$ then the distance between any two vertices in $S_{n,m}:=\{a^{n,m}_n,b^{n,m}_t:x_{n,m}(t)\not=0\}$ is at most $2$ via $b^{n,m}_t$.

Let $r$ be a witness for $\comp{\forall^\infty\forall\exists}(x)$.
Then, for any $n\geq r$ and $m$, we have $(n,m)\in W$.
In this case, for any $n\geq r$, the diameter of the connected component of $a_0^{n,m}$ is at most $2$, and for $n<r$, the diameter is at most $r$.
Hence, $r$ is a witness for ${\sf FinDiam}_{\sf conn}(V,E)$.

Conversely, if $r$ is a witness for ${\sf FinDiam}_{\sf conn}(V,E)$, then the diameter of any connected component is at most $r$.
Thus, for any $n>r$ and $m$, we must have $(n,m)\in W$.
This implies that $r+1$ is a witness for $\comp{\forall^\infty\forall\exists}(x)$.

For the dual, let $r\mapsto (n_r,m_r)$ be a witness for $\comp{\exists^\infty\exists\forall}(x)$; that is, $x_{n_r,m_r}(t)=0$ for any $t$.
In this case, we have $(n_r,m_r)\not\in W$, so the distance between $a^{n_r,m_r}_0$ and $a^{n_r,m_r}_{n_r}$ is $n_r\geq r$.
Take any path $\gamma_r$ connecting $a^{n_r,m_r}_0$ and $a^{n_r,m_r}_{n_r}$.
Then $r\mapsto\gamma_r$ is a witness for ${\sf FinDiam}_{\sf conn}^{\sf d}(V,E)$.

Conversely, $r\mapsto\gamma_r$ is a witness for ${\sf FinDiam}_{\sf conn}^{\sf d}(V,E)$; that is, the distance between two end points of $\gamma_r$ is at least $r$.
The connected component containing the path $\gamma_r$ has a vertex of the form $a_0^{n(r),m(r)}$.
For $r\geq 3$, we must have $(n(r),m(r))\not\in W$.
Then the end points of $\gamma_r$ are of the forms $a_i^{n(r),m(r)}$ and $a_j^{n(r),m(r)}$, whose distance is at least $r$.
Therefore, the distance between $a_0^{n(r),m(r)}$ and $a_{n(r)}^{n(r),m(r)}$ is $n(r)\geq r$.
Then $r\mapsto (n(r),m(r))$ is a witness for $\comp{\exists^\infty\exists\forall}(x)$.
\end{proof}

The exact complexity of ${\sf FinDiam}^{\sf d}$ has not been determined yet, but we give a lower bound here.
From now on, ${\sf FinDiam}^{\sf d}$ is denoted by ${\sf InfDiam}$.

\begin{prop}\label{prop:AAinftyA-vs-FinDiam}
$\comp{\forall\forall^\infty\forall}\leq_{\sf m}{\sf InfDiam}$.
\end{prop}

\begin{proof}
By Observation \ref{obs:all-bdd-dicomplete}, $\comp{\forall\forall^\infty\forall}\equiv_{\sf m}\forall{\sf Bdd}$, so it suffices to show $\forall{\sf Bdd}\leq_{\sf m}{\sf InfDiam}$.
By a construction similar to that of the reduction $\eta$ in Proposition \ref{prop:findiam-reduction-complete}, given $x=(x_n)_{n\in\om}$ we can guarantee that $\eta(x)=(V,E)$ satisfies the following:
\[
\forall n\exists m\forall t.\;x_n(t)\leq m\iff \mbox{the diameter of $(V,E)$ is infinite.}
\]

The graph has vertices $\ep$ and $a^{n,m}_\ell$ for any $n,m\in\om$ and $\ell\leq n$.
For each $\ell<n$, put $(\ep,a^{n,m}_0),(a^{n,m}_\ell,a^{n,m}_{\ell+1})\in E$, which yields infinitely many paths of length $n+1$ connecting $\ep$ and $a^{n,m}_n$.
Moreover, for any $k\leq n$ and $t$, if $x_{k}(t)>m$ then the add a new vertex $b^{n,m}_{k,t}\in V$, and put $(\ep,b^{n,m}_{k,t}),(a^{n,m}_\ell,b^{n,m}_{k,t})\in E$ for each $\ell\leq n$.
Then consider $W=\{(n,m):\exists k\leq n\exists t.\ x_k(t)>m\}$ and $S_{n,m}=\{\ep,a^{n,m}_n,b^{n,m}_{k,t}:k\leq n\mbox{ and }x_{k}(t)>m\}$ as before.
If $(n,m)\not\in W$ then the distance between $\ep$ and $a^{n,m}_n$ is $n+1$.
If $(n,m)\in W$ then the distance between any two vertices in $S_{n,m}$ is at most $2$ via some $b^{n,m}_{k,t}$.

Let $h$ be a witness for $\forall{\sf Bdd}(x)$; that is, $x_n(t)\leq h(n)$ for any $t$.
Putting $m(n)=\max_{k\leq n}h(k)$, we have $x_k(t)\leq m(n)$ for any $k\leq n$ and $t$.
Hence, $(n,m(n))\not\in W$, so the distance between $\ep$ and $a^{n,m(n)}_{n}$ is $n+1$.
Therefore, $n\mapsto(\ep,a^{n,m(n)}_n)$ is a witness for ${\sf InfDiam}(V,E)$.

Conversely, let $n\mapsto (u_n,v_n)$ be a witness for ${\sf InfDiam}(V,E)$; that is, the distance between $u_n$ and $v_n$ is at least $n$.
For $r=\max\{n,5\}$, compute indices $a(r),b(r),c(r),d(r)$ such that $u_r\in S_{a(r),b(r)}$ and $v_r\in S_{c(r),d(r)}$.
These indices are uniquely determined for vertices other than $\ep$, and if it is $\ep$, choose any indices.
Also, since $u_r\not=v_r$, one of them is not $\ep$.
Since the distance between $u_r$ and $v_r$ are at least $5$ and $\ep\in S_{a(r),b(r)}\cap S_{c(r),d(r)}$, the diameter of either $S_{a(r),b(r)}$ or $S_{c(r),d(r)}$ is at least $3$.
If the diameter of $S_{a(r),b(r)}$ is at least $3$, then we have $(a(r),b(r))\not\in W$; that is, for any $n\leq a(r)$, $x_n(t)\leq b(r)$ for any $t$.

Now, the distance between $u_{2r}$ and $v_{2r}$ is at least $2r$, so the diameter of either $S_{a(2r),b(2r)}$ or $S_{c(2r),d(2r)}$ is at least $r$.
This implies that either $a(2r)\geq r$ and $(a(2r),b(2r))\not\in W$ or $c(2r)\geq r$ and $(c(2r),d(2r))\not\in W$.

Then put $h(n)=\max\{b(2r),c(2r)\}$.
If $a(2r)\geq r$ and $(a(2r),b(2r))\not\in W$ then, since $n\leq r\leq a(2r)$, we get $x_n(t)\leq b(2r)$ for any $t$.
Similarly, if $c(2r)\geq r$ and $(c(2r),d(2r))\not\in W$ then we get $x_n(t)\leq c(2r)$ for any $t$.
In any case, we obtain $x_n(t)\leq h(n)$ for any $t$.
Hence, $h$ is a witness for $\forall{\sf Bdd}(x)$.
\end{proof}

Another lower bound is the disconnectedness problem ${\sf DisConn}$ for graphs, which is the $\exists\forall$-problem of determining whether a graph $G=(V,E)$ is disconnected or not.
\[{\sf DisConn}\colon\quad\exists u,v\in V\forall \gamma\in{\rm Path}_G.\ \{\gamma_{\rm start},\gamma_{\rm end}\}\not=\{u,v\}.\]

In \cite{Kihara}, it has been shown that $\comp{\forall^\infty\forall}<_{\sf m}{\sf DisConn}<_{\sf m}\comp{\exists\forall}$ holds.
As we will see later, ${\sf DisConn}$ and $\comp{\forall\forall^\infty\forall}$ are incomparable.

\begin{prop}\label{prop:DisConn-vs-InfDiam}
${\sf DisConn}\leq_{\sf m}{\sf InfDiam}$.
\end{prop}

\begin{proof}
Given a graph $G=(V,E)$, consider its transitive closure $G^\ast$; that is, for any $\gamma$ connecting vertices $u,v\in V$, add a vertex $a_\gamma$ and edges $(u,a_\gamma),(a_\gamma,v)$ to $G^\ast$.

If $G$ is connected then the diameter of $G^\ast$ is at most $2$, and if $G$ is disconnected then the diameter of $G^\ast$ is infinite.
Therefore, any witness $(u,v)$ for disconnectedness of $G$ is a witness for that the diameter of $G^\ast$ is at least $r$ for any $r$.

Conversely, let $(u_r,v_r)_{r\in\om}$ be a witness for ${\sf InfDiam}(G^\ast)$; that is, the distance between $u_r$ and $v_r$ is at least $r$.
In particular, the distance between $(u_5,v_5)$ is at least $5$.
If $u_5$ is a vertex of $G$ then put $u_5'=u_5$, and if $u_5$ is of the form $a_\gamma$ then let $u_5'$ be one of the end points of the path $\gamma$.
Note that the distance between $(u_5,u_5')$ in $G^\ast$ is at most $1$.
In a similar manner, we also define $v_5'$.
Then $u_5'$ and $v_5'$ are vertices of $G$.
If $u_5'$ and $v_5'$ belong to the same connected component in $G$, then the distance between $(u_5',v_5')$ is at most $2$ in $G^\ast$.
Therefore, the distance between $(u_5,v_5)$ must be at most $4$, which is impossible.
Consequently, $(u_5',v_5')$ is a witness for disconnectedness of $G$.
\end{proof}

In the later Section \ref{sec:separation-concrete-problem}, we will show that ${\sf InfDiam}$ is not $\exists^\infty\exists\forall$-complete, and in particular, ${\sf FinDiam}$ is not $\forall^\infty\forall\exists$-dicomplete.

Let us look at another example of the $\forall^\infty\forall\exists$-dicomplete problem.
The following example is about a preorder $\leq_R$ generated by a binary relation $R$.
In other words, $\leq_R$ is the reflexive transitive closure of $R$, or to be precise, the smallest binary relation that satisfies the following condition:
\begin{align*}
(a,b)\in R\implies a\leq_Ra,\ b\leq_R b,\ a\leq_Rb;
& &
a\leq_Rb\leq_Rc\implies a\leq_Rc.
\end{align*}

If we identify a binary relation $R$ with a directed graph, then $a\leq_Rb$ means that there exists a directed path from $a$ to $b$ in $R$.
The width of a preorder $P$ is the cardinality of a maximal antichain of $P$.
If no such value exists, the width is assumed to be infinite.
The following is a problem that asks whether the width of a preorder $R^\ast$ generated by a given binary relation $R$ is finite.
\[{\sf FinWidth}_\ast\colon\exists r\forall a_1,\dots,a_r\in R\exists j<k\leq r\;(a_j\leq_R a_k\mbox{ or }a_k\leq_R a_j).\]

Note that there is one existential quantifier hidden inside this formula.
In other words, the above formula can be rewritten as follows.
\[{\sf FinWidth}_\ast\colon\exists r\forall a_1,\dots,a_r\in R\exists\gamma\in{\sf Path}_R\exists i<j\;[\{\gamma_{\rm start},\gamma_{\rm end}\}=\{a_i,a_j\}].\]

\begin{theorem}\label{Fin-Width-complete}
${\sf FinWidth}_\ast$ is $\forall^\infty\forall\exists$-dicomplete.
\end{theorem}

\begin{proof}
First, in the definition of ${\sf FinWidth}_\ast$, the quantification $\exists r$ can be replaced with $\forall^\infty r$; thus, ${\sf FinWidth}_\ast\leq_{\sf dm}\comp{\forall^\infty\forall\exists}$.
Therefore, it remains to show $\comp{\exists^\infty\exists\forall}\leq_{\sf dm}{\sf FinWidth}_\ast^{\sf d}$.
Given $x=(x_{n,m})_{n,m\in\om}$, construct a binary relation $\eta(x)=R$ as follows:
\[
\forall r\exists n\geq r\exists m.\ x_{n,m}=0^\infty\iff \mbox{the width of the preorder $\leq_R$ is infinite.}
\]

For each $n,m$, prepare $(a^{n,m}_i)_{i<n}$.
If $\pair{n,m}<\pair{n',m'}$ with respect to the standard ordering of (pairs of) natural numbers, put $a^{n,m}_i<_Ra^{n',m'}_j$ for any $i<n$ and $j<n'$.
For each $n,m$, for the first $t$ such that $x_{n,m}(t)\not=0$, add $(b^{n,m}_{i,t})_{i<n-1}$ so that $a^{n,m}_i<_Rb^{n,m}_{i,t}<_Ra^{n,m}_{i+1}$.
Note that $b^{n,m}_{i,t}$ is comparable with any other element since $b^{n,m}_{i,t}<_Ra^{n,m}_{n-1}<_Ra^{n',m'}_j$ if $\pair{n,m}<\pair{n',m'}$ and $a^{n',m'}_j<_Ra^{n,m}_{0}<_Rb^{n,m}_{i,t}$ if $\pair{n',m'}<\pair{n,m}$.

Let $(n_r,m_r)_{r\in\om}$ be a witness for $\comp{\exists^\infty\exists\forall}(x)$, i.e., $x_{n_r,m_r}=0^\infty$.
Then there is no relation on $(a^{n_r,m_r}_i)_{i<r}$; that is, $(a^{n_r,m_r}_i)_{i<r}$ is an antichain, so this is a witness for that the width of the preorder $\leq_R$ is at least $r$.

Conversely, let $r\mapsto(c^r_i)_{i<r}$ be a witness for ${\sf FinWidth}_\ast^{\sf d}$; that is, $(c^r_i)_{i<r}$ is antichain in $\leq_R$.
Consider $r\geq 2$.
Since $(c^r_i)_{i<r}$ is an antichain and each $b^{n,m}_{k,t}$ is comparable with any other element, each $c^r_i$ must be of the form $a^{n(i),m(i)}_{k(i)}$.
If $\pair{n(i),m(i)}\not=\pair{n(j),m(j)}$ for some $i,j<r$ then $c^r_i$ is comparable with $c^r_j$ by our construction.
Hence, there exist $n(r)$ and $m(r)$ such that $c^r_i=a^{n(r),m(r)}_i$ for any $i$.
Since $(c^r_i)_{i<r}$ is an antichain, this means $x_{n(r),m(r)}=0^\infty$.
Therefore, $r\mapsto n(r),m(r)$ is a witness for $\comp{\exists^\infty\exists\forall}(x)$.

A similar argument applies to the dual.
\end{proof}

Incidentally, as we will see later, the $\exists^\infty\exists\forall$-completeness and the $\forall\exists\forall$-completeness coincide.
Thus, some of the duals of the examples given in Section \ref{sec:example-global-boundedness} are $\forall\exists\forall$-complete.
However, interestingly, assuming the later results, one can see that these are $\forall\exists\forall$-complete but not $\forall\exists\forall$-dicomplete, by showing that $\exists^\infty\exists\forall$-dicompleteness and the $\forall\exists\forall$-dicompleteness do not coincide.

\begin{prop}
${\sf FinDiam}_{\sf conn}^{\sf d}$ and ${\sf FinWidth}_\ast^{\sf d}$ are $\forall\exists\forall$-complete, but not $\forall\exists\forall$-dicomplete.
\end{prop}

\begin{proof}
Let ${\sf F}$ be either ${\sf FinDiam}_{\sf conn}^{\sf d}$ or ${\sf FinWidth}_\ast^{\sf d}$.
By Theorem \ref{Fin-diam-conn-complete} and Theorem \ref{Fin-Width-complete}, ${\sf F}$ is $\exists^\infty\exists\forall$-dicomplete.
By Theorem \ref{thm:formula-class-collapse} (4) proven later, we get $\comp{\forall\exists\forall}\equiv_{\sf m}\comp{\exists^\infty\exists\forall}$, so ${\sf F}$ is $\forall\exists\forall$-complete.
For the dual, we have ${\sf F}^{\sf d}\leq_{\sf m}\comp{\forall^\infty\forall\exists}$, and by Proposition \ref{prop:separation1} proven later, we get $\comp{\lor\forall}\not\leq_{\sf m}\comp{\forall^\infty\forall\exists}$.
However, we also have $\comp{\lor\forall}\leq_{\sf m}\comp{\exists\forall\exists}$; hence we get $\comp{\exists\forall\exists}\not\leq_{\sf m}{\sf F}^{\sf d}$.
Therefore, ${\sf F}$ cannot be $\forall\exists\forall$-dicomplete.
\end{proof}

\subsection{$\forall\exists\forall$: Independent alternation of quantifiers}
Although there are many known classical $\Pi_3$-complete problems, of those we have already seen, the only one that is $\forall\exists\forall$-dicomplete is the complementedness problem ${\sf Compl}$.
Thus, it seems that examples of $\forall\exists\forall$-dicomplete problems about a single structure are somewhat rare.
However, there is a simple way to find a $\forall\exists\forall$-dicomplete problem.
This is to consider a sequence of structures rather than a single structure.
In Proposition \ref{prop:Q-complete}, we have introduced the addition ${\sf P}\varphi$ of a quantifier ${\sf P}$ to an arithmetical formula $\varphi$.
In particular, the formula $\forall\varphi$ is described as follows:
\[(\forall\varphi)(\pair{x_n}_{n\in\om})\quad\equiv\quad\forall n\varphi(x_n).\]

\begin{obs}\label{obs:creating-AEA-complete}
If a formula $\varphi$ is $\exists\forall$-complete, then $\forall\varphi$ is $\forall\exists\forall$-dicomplete.
\end{obs}

\begin{proof}
As seen in the proof of Proposition \ref{prop:quantifier-bi-complete}, if $\varphi\leq_{\sf dm}\psi$ then ${\sf P}\varphi\leq_{\sf dm}{\sf P}\psi$.
In particular, $\comp{\exists\forall}\leq_{\sf dm}\psi$ implies $\comp{\forall\exists\forall}\leq_{\sf dm}\forall\psi$.
By our assumption, since $\varphi$ is $\exists\forall$-complete, we have $\pair{\exists\forall}\leq_{\sf m}\varphi$ via some $\eta$; that is, $\pair{\forall\exists}(x)$ holds if and only if $\varphi^{\sf d}(\eta(x))$ holds.
For a $\forall\exists$-formula $\psi$, one can always compute its witness for $\psi(x)$, so we automatically obtain $\pair{\forall\exists}\leq_{\sf m}\varphi^{\sf d}$.
Therefore, we get $\pair{\exists\forall}\leq_{\sf dm}\varphi$, which implies $\pair{\forall\exists\forall}\leq_{\sf dm}\forall\varphi$ as mentioned above.
\end{proof}

In order to obtain a $\forall\exists\forall$-dicomplete problem, consider the following formula concerning density of a linear order $L$.
\[{\sf Dense}:\quad\forall a,b\in L\,\exists c\in L\;(a<_Lb\to a<_Lc<_Lb).\]

Clearly, ${\sf Dense}$ is a $\forall\exists$-formula, so its dual ${\sf Dense}^{\sf d}$ is a $\exists\forall$-formula.
Therefore, $\forall({\sf Dense}^{\sf d})$ is a $\forall\exists\forall$-formula.
The decision problem $\forall({\sf Dense}^{\sf d})$ is the problem of determining whether or not ``all $L_n$ are not dense'' for a sequence $L=(L_n)_{n\in\om}$ of linear orders.
\[\forall({\sf Dense}^{\sf d})(L)\iff\mbox{$L_n$ is not dense for any $n$.}\]

\begin{cor}
$\forall({\sf Dense}^{\sf d})$ is $\forall\exists\forall$-dicomplete.
\end{cor}

\begin{proof}
It is shown in \cite{Kihara} that ${\sf Dense}^{\sf d}$ is $\exists\forall$-complete.
Hence, by Observation \ref{obs:creating-AEA-complete}, $\forall({\sf Dense}^{\sf d})$ is $\forall\exists\forall$-complete.
\end{proof}


\subsection{$\forall^\to\exists\forall$: Restricting range of quantification}\label{sec:example-restrict-quantify}
In everyday mathematics, the range of quantification is often specified.
For example, we usually consider expressions such as $\forall x\in A\varphi(x)$ and $\exists x\in A\varphi(x)$.
These two quantifications are abbreviations for the following expressions, respectively.
\begin{align*}
\forall x\;(x\in A\to\varphi(x)),
& &
\exists x\;(x\in A\land\varphi(x)).
\end{align*}

Of course, we fix the domain of discourse (i.e., the set over which $x$ runs) to the natural numbers.
However, even if we are considering a specific domain of discourse, we do not necessarily quantify over the entire domain; that is, the quantification range is often limited by some condition $\gamma$.
To be explicit, we often consider formulas of the following form.
\begin{align*}
\forall x\;(\gamma(x)\to\varphi(x,y)),
& &
\exists x\;(\gamma(x)\land\varphi(x,y)).
\end{align*}

Some of the decision problems we have dealt with so far also have such partial quantifications if we formalize them directly.
However, in any of the previous examples, such $\gamma$ is a bounded formula, so we can move it to the innermost part of the formula, which allows us to formalize it in a way that excludes partial quantification.
Of course, this is not always possible.
Often, $\gamma$ is a complicated formula, and in such cases, the quantification over $\gamma$ cannot be removed.
Let us take up such an example.

A tree $T\subseteq 2^{<\om}$ is perfect if $T$ has no isolated infinite path.
We say that a node $\sigma\in T$ is extendible if there is an infinite path through $T$ extending $\sigma$.
Classically, a tree is perfect iff any extendible node of $T$ can be extended to two incomparable extendible nodes of $T$.
This involves the quantification over extendible nodes, and there is no need to consider other nodes.
If $T$ is a binary tree, under the assumption of weak K\"onig's lemma, the extendibility of a node $\sigma\in T$ is equivalent to the following:
\[{\sf Ext}(\sigma,T)\equiv\forall\ell\,\exists\tau\in 2^\ell.\ \tau\in T\]

Note that $\tau$ ranges over a finite set, so the inner existential formula is a bounded formula.
Thus, the extendibility of a node is described by a $\forall$-formula.
Then, the perfectness of a binary tree can be described as follows:
\[{\sf Perfect}_{\sf bin}:\quad\forall\sigma\in T\ [{\sf Ext}(\sigma,T)\to \exists\tau_0,\tau_1\in T\;(\tau_0\bot\tau_1\ \land\ \forall i<2.\ {\sf Ext}(\tau_i,T))].\]

Here, $\tau_0\bot\tau_1$ denotes that nodes $\tau_0,\tau_1$ are incomparable.
In the definition of the perfectness of a binary tree, the quantification range of $\sigma$ is restricted by a $\forall$-formula.
That is, it is of the type $\forall(\forall\to\exists\forall)$ in the following sense.
\[\forall a[\forall b\gamma(a,b,x)\to\exists c\forall d\theta(a,c,d,x)].\]

Here, $\gamma,\theta$ are bounded formulas.
We call the formula of this form a $\forall^\to\exists\forall$-formula, and its dual (i.e., a formula of the following form) a $\exists^\land\forall\exists$-formula.
\[\exists a[\forall b\gamma(a,b,x)\land\forall c\exists d\theta(a,c,d,x)].\]

\begin{theorem}\label{thm:perfect-dicomplete}
${\sf Perfect}_{\sf bin}$ is $\forall^\to\exists\forall$-dicomplete.
\end{theorem}

\begin{proof}
Given $(p,x)=(p_n,x^n_m)_{n,m\in\om}$, construct a binary tree $T=\eta(p,x)\subseteq 2^{<\om}$ as follows:
\[\forall n\ (p_n=0^\infty\to\exists m.\ x^n_m=0^\infty)\iff T\mbox{ is perfect.}\]

Here, note that the left-hand side expresses a $\forall^\to\exists\forall$-complete formula $\pair{\forall^\to\exists\forall}(p,x)$.
Its witness is a partial function $n\mapsto m$.

First put $T_0=\{0^n,0^n10\sigma:n\in\om,\sigma\in 2^{<\om}\}\subseteq T$.
Next, for each $n$, put $0^n110^s\in T$ for any $s\in\om$ until we see $p_n\not=0^\infty$.
For each $m$, put $0^n110^{\pair{m,s}}1\sigma$ for any $s\in\om$ and $\sigma\in 2^{<\om}$ until we see $p_n\not=0^\infty$ or $x^n_m\not=0^\infty$.
Note that, for each $\sigma\in T_0$, the extensions $\sigma 0$ and $\sigma 1$ are extendible since $0^n10\alpha$ is an infinite path through $T_0$ for any $n\in\om$ and $\alpha\in 2^\om$.
Therefore, let us focus on the behavior of $\sigma\not\in T_0$.

Let $n\mapsto m_n$ be a witness for $\comp{\forall^\to\exists\forall}(p,x)$; that is, if $p_n=0^\infty$ then $x^n_{m_n}=0^\infty$.
In this case, we show that each extendible node can be extended to two incomparable extendible nodes.
As discussed above, we only need to consider $\sigma\not\in T_0$.
Note that if $p_n\not=0^\infty$ then $0^n11$ is not extendible in $T$.
If $\sigma\in T$ is a node of the form $\sigma=0^n110^s$ and if it is extendible, then we must have $p_n=0^\infty$; hence, $x^n_{m_n}=0^\infty$.
In this case, take a sufficiently large $k$ such that $t=\pair{m_n,k}\geq s$.
Then $\sigma=0^n110^s$ is extended to $0^n110^t$, which is extended to $0^n110^t0^\infty$ and $0^n110^t1^\infty$.
Then consider $\gamma(\sigma)=\pair{0^n110^t0,0^n110^t1}$.
If $\sigma\in T$ is a node of the form $\sigma=0^n110^{\pair{m,k}}1\tau$ and if it is extendible, we must have $p_n=0^\infty$ and $x^n_m=0^\infty$.
In this case, $\sigma$ is clearly extendible to $\sigma 0^\infty=0^n110^{\pair{m,k}}1\tau 0^\infty$ and $\sigma 1^\infty=0^n110^{\pair{m,k}}1\tau 1^\infty$.
Then consider $\gamma(\sigma)=\pair{\sigma 0,\sigma 1}$.
Then, $\sigma\mapsto\gamma(\sigma)$ gives a witness for ${\sf Perfect}_{\sf bin}(T)$.

Conversely, let $\sigma\mapsto \tau^\sigma_0,\tau^\sigma_1$ be a witness for ${\sf Perfect}_{\sf bin}(T)$.
Given $n$, consider $\sigma=0^n11$.
If $p_n=0^\infty$ then $\sigma=0^n11$ is extendible, so $\tau^\sigma_0$ and $\tau^\sigma_1$ are defined.
Since $\tau^\sigma_0$ and $\tau^\sigma_1$ are incomparable extensions of $\sigma$, either $\tau^\sigma_0$ or $\tau^\sigma_1$ extend $0^n110^{\pair{m,k}}1$ for some $m,k$.
Since $\tau^\sigma_0$ and $\tau^\sigma_1$ are extendible, this implies $x^n_m=0^\infty$.
Thus, $n\mapsto m$ witnesses $\comp{\forall^\to\exists\forall}(p,x)$.

For the dual, let $n$ be a witness for $\comp{\exists^\land\forall\exists}(p,x)$; that is, $p_n=0^\infty$ and $x^n_m\not=0^\infty$ for any $m\in\om$.
As $p_n=0^\infty$, the node $0^n11$ is extendible to $0^n110^\infty$.
For each $m$, since $x^n_m\not=0^\infty$, the node $0^n110^{\pair{m,k}}1$ is not extendible.
Hence, the only infinite path that extends $0^n11$ is $0^n110^\infty$.
Therefore, $0^n11$ is a witness for ${\sf Perfect}^{\sf d}(T)$.

Conversely, let $\sigma$ be a witness for ${\sf Perfect}^{\sf d}(T)$; that is, there is exactly one infinite path extending $\sigma$.
As mentioned above, $\sigma\not\in T_0$.
If $\sigma$ is of the form $0^n110^s$, since $\sigma$ is extendible, we have $p_n=0^\infty$.
Moreover, in this case, $\sigma$ is extendible to $0^n110^\infty$.
Since $\sigma$ is a witness for $T$ being non-perfect, there is no other infinite path through $T$ extending $\sigma$; hence, $0^n110^{\pair{m,k}}1$ is not extendible for any $\pair{m,k}\geq s$.
In particular, for any $m$, $0^n110^{\pair{m,k}}1$ is not extendible for a sufficiently large $k$, which implies $x^n_m\not=0^\infty$ for any $m$.
If $\sigma$ is of the form $\sigma=0^n110^{\pair{m,k}}1\tau$, since $\sigma$ is extendible, we must have $p_n=0^\infty$ and $x^n_m=0^\infty$; however, in this case, for any $\alpha\in 2^\om$, $\sigma=0^n110^{\pair{m,k}}1\tau\alpha$ is an infinite path through $T$, which contradicts our assumption that $\sigma$ is a witness for $T$ being non-perfect.
Consequently, if $\sigma$ is a witness for ${\sf Perfect}^{\sf d}(T)$, then it is of the form $0^n110^s$, and in this case, $n$ is a witness for $\comp{\exists^\land\forall\exists}(p,x)$.
\end{proof}

In Section \ref{sec:separation-perfectness}, we show $\comp{\forall\exists\forall}<_{\sf m}\comp{\forall^\to\exists\forall}$.

\section{The structure of quantifier-patterns}

\subsection{Equivalence}

Example \ref{exa:Sigma03-absorb} presents sixteen $\Sigma_3$ quantifier-patterns.
However, not all of these are different, and we can expect that many of them are equivalent.
What we want to know is exactly how many $\Sigma_3$- and $\Pi_3$-patterns there are (modulo the many-one equivalence).

A quantifier-pattern $\qf{Q}$ is ${\sf m}$-equivalent to $\qf{Q}'$ if $\comp{\qf{Q}}\equiv_{\sf m}\comp{\qf{Q}'}$.
Similarly, a quantifier-pattern $\qf{Q}$ is ${\sf dm}$-equivalent to $\qf{Q}'$ if $\comp{\qf{Q}}\equiv_{\sf dm}\comp{\qf{Q}'}$.
Note that ${\sf dm}$-reducibility for $\Sigma_3$- and $\Pi_3$-patterns are completely the same, so it is sufficient to consider only one of them.
Our main result in this section is the following:

\begin{theorem}\label{thm:sigma3-quantifier-classification}
Any $\Sigma_3$ quantifier-pattern is ${\sf m}$-equivalent to one of the following $3$ patterns: 
\[\exists\forall\exists,\forall^\infty\exists^\infty,\forall^\infty\exists.\]

Any $\Pi_3$ quantifier-pattern is ${\sf m}$-equivalent to one of the following $5$ patterns:
\[\forall\exists\forall,\exists^\infty\forall^\infty\forall,\exists^\infty\forall,\forall\forall^\infty\forall,\forall\forall^\infty.\]

Any $\Pi_3$ quantifier-pattern is ${\sf dm}$-equivalent to one of the following $7$ patterns:
\[\forall\exists\forall,\exists^\infty\exists\forall,\exists^\infty\forall^\infty\forall,\exists^\infty\forall^\infty,\exists^\infty\forall,\forall\forall^\infty\forall,\forall\forall^\infty.\]
\end{theorem}

\begin{figure}[t]
\begin{center}
\includegraphics[width=90mm]{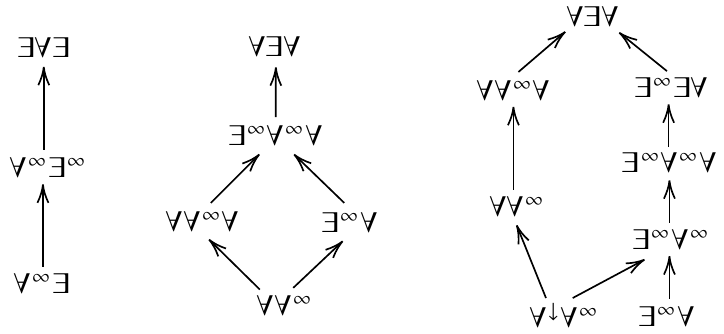}
\end{center}
\caption{(left) The $\equiv_{\sf m}$-degrees of $\Sigma_3$-patterns; (center) The $\equiv_{\sf m}$-degrees of $\Pi_3$-patterns; (right) The $\equiv_{\sf dm}$-degrees of $\Pi_3$-patterns.}\label{fig:quantifier-class}
\end{figure}

Figure \ref{fig:quantifier-class} shows, from left to right, the $\equiv_{\sf m}$-degrees of $\Sigma_3$ quantifier-patterns, the $\equiv_{\sf m}$ quantifier-degrees of $\Pi_3$-patterns, and the $\equiv_{\sf dm}$-degrees of $\Pi_3$ quantifier-patterns.
In Section \ref{sec:separation}, we will show that Figure \ref{fig:quantifier-class} is complete (that is, no further arrow is added).

In order to prove Theorem \ref{thm:sigma3-quantifier-classification}, it is necessary to check which quantifier-patterns are $\equiv_{\sf m}$-equivalent.

\begin{theorem}\label{thm:formula-class-collapse}~
\begin{enumerate}
\item $\comp{\exists\forall\exists}\equiv_{\sf dm}\comp{\exists\forall\exists^\infty}\equiv_{\sf dm}\comp{\exists\forall^\infty\exists^\infty}\equiv_{\sf dm}\comp{\exists\forall^\infty\exists}$.
\item $\comp{\exists\forall\exists}\equiv_{\sf m}\comp{\exists\exists^\infty\exists}\equiv_{\sf m}\comp{\exists\exists^\infty}$.
\item $\comp{\forall^\infty\forall\exists}\equiv_{\sf m}\comp{\forall^\infty\forall\exists^\infty}\equiv_{\sf m}\comp{\forall^\infty\exists^\infty\exists}\equiv_{\sf m}\comp{\forall^\infty\exists^\infty}$.
\item $\comp{\forall\exists\forall}\equiv_{\sf m}\comp{\exists^\infty\exists\forall}\equiv_{\sf m}\comp{\exists^\infty\exists\forall^\infty}$.
\item $\comp{\exists^\infty\forall^\infty}\equiv_{\sf m}\comp{\exists^\infty\forall}$.
\end{enumerate}
\end{theorem}

We first give a proof of the relatively trivial parts that follows from the observations so far.

\begin{proof}[Proof of Theorem \ref{thm:formula-class-collapse} (2), (3)]
A witness for a $\forall\exists$-formula is always computable, so we get $\comp{\forall\exists}\equiv_{\sf m}\comp{\forall\exists^\infty}\equiv_{\sf m}\comp{\exists^\infty\exists}\equiv_{\sf m}\comp{\exists^\infty}$.
Then, by adding the quantifiers $\exists$ and $\forall^\infty$ to the prefix, we get (2) and (3) by Lemma \ref{lem:quantifier-add-lemma}.
\end{proof}

The proof of this theorem is divided into several propositions.

\begin{prop}\label{prop:collapse1}
$\comp{\exists\forall\exists}\leq_{\sf dm}\comp{\exists\forall^\infty\exists}$.
\end{prop}

\begin{proof}
Given $x=(x_n)_{n\in\om}$, we construct $y=\eta(x)$ such that
\[\exists n\forall t.\ x_n(t)\not=0^\infty\iff\exists n\forall^\infty t.\ y_n(t)\not=0^\infty.\]

The left-hand side means that $\comp{\exists\forall\exists}(x)$ is true, and the right-hand side means that $\comp{\exists\forall^\infty\exists}(y)$ is true.
For $t,u$, define $y_n(\pair{t,u})=x_n(t)$.
Then, the property $\exists t\,x_n(t)=0^\infty$ is classically equivalent to $\exists^\infty t\,y_n(t)=0^\infty$, so $n$ is a witness for the former iff it is a witness for the latter.
This implies $\comp{\exists\forall\exists}\leq_{\sf m}\comp{\exists\forall^\infty\exists}$.

For the dual $\comp{\forall\exists\forall}\leq_{\sf m}\comp{\forall\exists^\infty\forall}$, let $n\mapsto t_n$ be a witness for $\comp{\forall\exists\forall}(x)$, i.e., $x_n(t_n)=0^\infty$.
Then $y_n(\pair{t_n,u})=0^\infty$ for any $u$ by definition, so $n\mapsto (\pair{t_n,u})_{u\in\om}$ is a witness for $\comp{\forall\exists^\infty\forall}(y)$.
Conversely, let $n\mapsto(t^n_i,u^n_i)_{i\in\om}$ be a witness for $\comp{\forall\exists^\infty\forall}(y)$.
Then $y_n(\pair{t^n_i,u^n_i})=0^\infty$; thus, in particular, we get $x_n(t^n_0)=0^\infty$.
Hence, $n\mapsto t^n_0$ gives a witness for $\comp{\forall\exists\forall}(x)$.
\end{proof}

\begin{proof}[Proof of Theorem \ref{thm:formula-class-collapse} (1)]
First we have the absorption relation $\exists\forall^\infty\exists^\infty\to\exists\exists\forall\exists^\infty\to\exists\forall\exists^\infty\to\exists\forall\forall\exists\to\exists\forall\exists$.
Therefore, by Observation \ref{obs:absorb-bm-relation}, we get $\comp{\exists\forall^\infty\exists^\infty}\leq_{\sf dm}\comp{\exists\forall\exists^\infty}\leq_{\sf dm}\comp{\exists\forall\exists}$.
Moreover, by Proposition \ref{prop:ex-ex-infty-bm-reducible}, $\comp{\bar{\sf Q}\exists}\leq_{\sf bm}\comp{\bar{\sf Q}\exists^\infty}$, and thus, $\comp{\exists\forall^\infty\exists}\leq_{\sf bm}\comp{\exists\forall^\infty\exists^\infty}$.
Then, by Proposition \ref{prop:collapse1}, we get $\comp{\exists\forall\exists}\leq_{\sf bm}\comp{\exists\forall^\infty\exists}$.
This completes the proof.
\end{proof}

\begin{prop}\label{prop:collapse2}
$\comp{\forall\exists\forall}\leq_{\sf m}\comp{\exists^\infty\exists\forall}$.
\end{prop}

\begin{proof}
Given $x=(x_n)_{n\in\om}$, we construct $y=\eta(x)$ such that
\[\forall n\exists m\forall s.\ x_n(m,s)=0\iff\exists^\infty n\exists \sigma\forall s.\ y_n(\sigma,s)=0.\]

The left-hand side means that $\comp{\forall\exists\forall}(x)$ is true, and the right-hand side means that $\comp{\exists^\infty\exists\forall}(y)$ is true.
Given $s$ and a sequence $\sigma$ of length $n+1$, define $y_n(\sigma,s)=\max_{i\leq n}x_i(\sigma(i),s)$.
Now let $n\mapsto m_n$ be a witness for $\comp{\forall\exists\forall}(x)$; that is, $x_n(m_n,s)=0$ for any $s$.
For each $n$, consider $\sigma_n=(m_i)_{i\leq n}$.
Then $y_n(\sigma_n,s)=\max_{i\leq n}x_i(m_i,s)=0$ for any $s$.
Hence, $n\mapsto(n,\sigma_n)$ is a witness for $\comp{\exists^\infty\exists\forall}(y)$.

Conversely, let $n\mapsto(a(n),\sigma_n)$ be a witness for $\comp{\exists^\infty\exists\forall}(y)$; that is, $y_{a(n)}(\sigma_n,s)=0$ for any $s$.
Then, for any $m$, there exists $n_m$ such that $m\leq a(n_m)$.
By definition, $x_m(\sigma_{n_m}(m),s)\leq y_{a(n_m)}(\sigma_{n_m},s)=0$ for any $s$.
Hence, $m\mapsto\sigma_{n_m}$ is a witness for $\comp{\forall\exists\forall}(x)$.
\end{proof}

\begin{prop}\label{prop:collapse3}
$\comp{\exists^\infty\forall^\infty}\leq_{\sf m}\comp{\exists^\infty\forall}$.
\end{prop}

\begin{proof}
Given $x=(x_n)_{n\in\om}$, we construct $y=\eta(x)$ such that
\[\exists^\infty n\exists s\forall t\geq s.\ x_n(t)=0\iff\exists^\infty n.\ y_n=0^\infty.\]

The left-hand side means that $\comp{\exists^\infty\forall^\infty}(x)$ is true, and the right-hand side means that $\comp{\exists^\infty\forall}(y)$ is true.
For an increasing sequence $\sigma=\pair{n_0,\dots,n_\ell}$ and a sequence $\tau=\pair{s_0,\dots,s_\ell}$, we construct an infinite sequence $y_{\sigma,\tau}$.
We want to ensure that $y_{\sigma,\tau}=0^\infty$ iff, for any $i<|\tau|$, $s_i$ is the least witness for $\comp{\forall^\infty}(x_{n_i})$.
The construction procedure of $y_{\sigma,\tau}$ checks if $\tau$ is the least witness for $\sigma$ in the above sense.
To be precise, until the minimality is refuted, inductively assume that $y_{\sigma,\tau}$ is the initial segment of $0^\infty$ of length $s$, at each stage $s$.
If the minimality is refuted; that is, there exists $i\leq\ell$ and $t<s_i$ such that $x_{n_i}(t)=0$ for any $u\in[t,s_i)$, then the construction ensures $y_{\sigma,\tau}\not=0^\infty$ by putting $y_{\sigma,\tau}(u)\not=0$ for $u\geq s$.
Otherwise, $\tau$ is still the least witness for $\sigma$, so $y_{\sigma,\tau}$ is now the initial segment of $0^\infty$ of length $s+1$.

Now let $(n_i,s_i)_{i\in\om}$ be a witness for $\comp{\exists^\infty\forall^\infty}(x)$; that is, $n_i\geq i$ and $x_{n_i}(t)=0$ for any $t\geq s_i$.
By taking a subsequence, without loss of generality, we may assume that $(n_i)_{i\in\om}$ is a strictly increasing sequence.
Moreover, from the information on $(s_i)_{i\in\om}$ one can obtain the least witnesses $(s_i')_{i\in\om}$ by taking $s_i'=\min\{s:\forall t\in[s,s_i].\ x_{n_i}(t)=0\}$.
By our construction, we have $y_{(n_i)_{i\leq\ell},(s_i')_{i\leq\ell}}=0^\infty$ for any $\ell$.
Hence, $\ell\mapsto((n_i)_{i\leq\ell},(s_i')_{i\leq\ell})$ is a witness for $\comp{\exists^\infty\forall}(y)$.

Conversely, let $(\sigma_i,\tau_i)_{i\in\om}$ be a witness for $\comp{\exists^\infty\forall}(y)$, i.e., $y_{\sigma_i,\tau_i}=0^\infty$.
Then, note that there are infinitely many numbers that appear in $(\sigma_i)_{i\in\om}$.
Otherwise, $\{\sigma_i(k):i\in\om,k<|\sigma_i|\}$ is finite, and each $\sigma_i$ corresponds to its finite subset, so $\{\sigma_i:i\in\om\}$ is finite.
However, for each $\sigma_i$, $\tau_i$ is the corresponding least witness, so it is unique by minimality.
In other words, if $\sigma_i=\sigma_j$, then $\tau_i=\tau_j$, but this means that $\{\sigma_i,\tau_i:i\in\om\}$ is finite.
However, this is a witness for $\exists^\infty$, so this cannot be the case.
Hence, for any $n$, search for $i,k$ such that $\sigma_{i}(k)>n$ and put $m_n=\sigma_i(k)$.
By definition, we have $x_{m_n}(t)=0$ for any $t\geq \tau_i(k)$, so put $s_n=\tau_i(k)$.
Hence, $n\mapsto(m_n,s_n)$ is a witness for $\comp{\exists^\infty\forall^\infty}(x)$; that is, $m_n\geq n$ and $x_{m_n}(t)=0$ for any $t\geq s_n$.
\end{proof}

\begin{proof}[Proof of Theorem \ref{thm:formula-class-collapse} (4), (5)]
First we have the absorption relation $\exists^\infty\exists\forall^\infty\to\exists^\infty\exists\exists\forall\to\exists^\infty\exists\forall\to\forall\exists\exists\forall\to\forall\exists\forall$.
Therefore, by Observation \ref{obs:absorb-bm-relation}, we get $\comp{\exists^\infty\exists\forall^\infty}\leq_{\sf dm}\comp{\exists^\infty\exists\forall}\leq_{\sf dm}\comp{\forall\exists\forall}$.
Moreover, by Proposition \ref{prop:ex-ex-infty-bm-reducible}, $\comp{\bar{\sf Q}\forall}\leq_{\sf dm}\comp{\bar{\sf Q}\forall^\infty}$, and thus, $\comp{\exists^\infty\exists\forall}\leq_{\sf dm}\comp{\exists^\infty\exists\forall^\infty}$.
Then,  by Proposition \ref{prop:collapse2}, we get $\comp{\forall\exists\forall}\leq_{\sf m}\comp{\exists^\infty\exists\forall}$.
This verifies the item (4).

Similarly, by Proposition \ref{prop:ex-ex-infty-bm-reducible}, we have $\comp{\exists^\infty\forall}\leq_{\sf dm}\comp{\exists^\infty\forall^\infty}$ and by Proposition \ref{prop:collapse3} we get $\comp{\exists^\infty\forall^\infty}\leq_{\sf m}\comp{\exists^\infty\forall}$.
This verifies the item (5).
\end{proof}

There are several other reducibility results, so let us mention them here.

\begin{prop}
$\comp{\forall\forall^\infty\forall}\leq_{\sf m}\comp{\exists^\infty\forall^\infty\forall}$.
\end{prop}

\begin{proof}
As in Observation \ref{obs:all-bdd-dicomplete}, we have $\comp{\forall\forall^\infty\forall}\equiv_{\sf dm}\forall{\sf Bdd}$ and $\comp{\exists^\infty\forall^\infty\forall}\equiv_{\sf dm}\exists^\infty{\sf Bdd}$.
Hence, it suffices to show $\forall{\sf Bdd}\leq_{\sf m}\exists^\infty{\sf Bdd}$.
Given $x=(x_n)_{n\in\om}$, we construct $y=\eta(x)$ such that
\[\forall n\exists b\forall t.\ x_n(t)\leq b\iff\exists^\infty n\exists b.\ y_n(t)\leq b.\]

The left-hand side means that ${\forall{\sf Bdd}}(x)$ is true, and the right-hand side means that ${\exists^\infty{\sf Bdd}}(y)$ is true.
Define $y_n(t)=\max_{i\leq n}x_i(t)$.
Let $b=(b_n)_{n\in\om}$ be a witness for ${\forall{\sf Bdd}}(x)$; that is, $b_n$ is an upper bound for $x_n$.
Clearly, the maximum $c_n=\max_{i\leq n}b_i$ of upper bounds of $(x_i)_{i\leq n}$ is an upper bound of $y_n$.
Conversely, if we have upper bounds for $y_n$ for infinitely many $n$; that is, we have $k$ such that $k(n)\geq n$ and $c_n$ is an upper bound for $y_{k(n)}$ for any $n$.
Then $c_n$ is also an upper bound for $x_n$.
\end{proof}

\begin{prop}
$\comp{\forall\forall^\infty}\leq_{\sf m}\comp{\exists^\infty\forall^\infty}$.
\end{prop}

\begin{proof}
Given $x=(x_n)_{n\in\om}$, we construct $y=\eta(x)$ such that
\[\forall n\forall^\infty t.\ x_n(t)=0\iff\exists^\infty n\forall^\infty t.\ y_n(t)=0.\]

Again, we define $y_n(t)=\max_{i\leq n}x_i(t)$.
Let $n\mapsto t_n$ be a witness for $\comp{\forall\forall^\infty}(x)$.
Then one can easily see that $n\mapsto(n,\max_{i\leq n}t_n)$ is a witness for $\comp{\exists^\infty\forall^\infty}(y)$.
Conversely, let $n\mapsto (a(n),s_n)$ be a witness for $y\in\comp{\exists^\infty\forall^\infty}$; that is, $y_{a(n)}(t)=0$ for any $t\geq s_n$, and for any $m$ there is $n_m$ such that $m\leq a(n_m)$.
By definition, we have $x_m(t)\leq y_{a(n_m)}(t)=0$.
In particular, $x_m(t)=0$ for any $t\geq s_{n_m}$.
Hence, $m\mapsto s_{n_m}$ is a witness for $\comp{\forall\forall^\infty}(x)$.
\end{proof}

Using the above, we can now see that all of the $\Sigma_3$- and $\Pi_3$-patterns have already been presented in Figure \ref{fig:quantifier-class}.

\begin{proof}[Proof of Theorem \ref{thm:sigma3-quantifier-classification}]
Let $\mathcal{E}$ be the set of all quantifier-patterns presented in Example \ref{exa:Sigma03-absorb}.
By Proposition \ref{prop:Sigma3-bluteforce-list}, any $\Sigma_3$-pattern $\bar{\sf Q}$ is absorbable into some $\bar{\sf P}\in\mathcal{E}$ and vice versa, so by Observation \ref{obs:absorb-bm-relation} we have $\bar{\sf Q}\equiv_{\sf dm}\bar{\sf P}$.
Now let us see that any quantifier-pattern $\bar{\sf Q}\in\mathcal{E}$ of length $4$ or more is ${\sf dm}$-equivalent to a quantifier-pattern of length at most $3$.
Any such pattern $\bar{\sf Q}$, except for $\forall^\infty\forall\exists^\infty\exists$, contains either $\exists\forall\exists$ or $\exists\forall^\infty\exists^\infty$ as a subpattern, so by Theorem \ref{thm:formula-class-collapse} (1), we get $\comp{\exists\forall\exists}\leq_{\sf dm}\comp{\bar{\sf Q}}$.
If $\bar{\sf Q}$ is a $\Sigma_3$-pattern, then it is clearly absorbable into $\exists\forall\exists$, so by Observation \ref{obs:absorb-bm-relation}, we get $\comp{\bar{\sf Q}}\equiv_{\sf dm}\comp{\exists\forall\exists}$.
For $\bar{\sf Q}=\forall^\infty\forall\exists^\infty\exists$, we have the absorption relations $\forall^\infty\forall\exists^\infty\exists\to\forall^\infty\forall\forall\exists\exists\to\forall^\infty\forall\exists$ and $\forall^\infty\forall\exists\to\forall^\infty\forall\exists^\infty\exists$, so by Observation \ref{obs:absorb-bm-relation} we get $\comp{\forall^\infty\forall\exists}\equiv_{\sf dm}\comp{\forall^\infty\forall\exists^\infty\exists}$.

Thus, every $\Sigma_3$-pattern is ${\sf dm}$-equivalent to a pattern $\bar{\sf P}\in\mathcal{E}$ of length at most $3$, and every $\Pi_3$-pattern is ${\sf dm}$-equivalent to the $\bar{\sf P}^{\sf d}$ of a pattern $\bar{\sf P}\in\mathcal{E}$ of length at most $3$.
Any pattern $\bar{\sf P}\in\mathcal{E}$ of length at most $3$ has already been presented in the left diagram of Figure \ref{fig:arith2}.
First, by Theorem \ref{thm:formula-class-collapse} (1),(2), all the patterns belonging to the polygon in the left diagram of Figure \ref{fig:arith2} are ${\sf m}$-equivalent.
Moreover, by Theorem \ref{thm:formula-class-collapse} (3), all the patterns belonging to the ellipse in the left diagram of Figure \ref{fig:arith2} are ${\sf m}$-equivalent.
Similarly, any pattern which is the dual of some $\bar{\sf P}\in\mathcal{E}$ of length at most $3$ has already been presented in the right diagram in Figure \ref{fig:arith2}.
By Theorem \ref{thm:formula-class-collapse} (1),(4), all the patterns belonging to the polygon in the right diagram of Figure \ref{fig:arith2} are ${\sf m}$-equivalent.
Finally, by Theorem \ref{thm:formula-class-collapse} (5), all the patterns belonging to the ellipse in the right diagram of Figure \ref{fig:arith2} are ${\sf m}$-equivalent.
Therefore, the structures of the $\Sigma_3$- and $\Pi_3$-patterns in Figure \ref{fig:arith2} collapse as shown in Figure \ref{fig:quantifier-class}, and only the three $\equiv_{\sf m}$-classes of $\Sigma_3$-patterns and five $\equiv_{\sf m}$-classes of $\Pi_3$-patterns survive.

Now, consider the number of ${\sf dm}$-equivalence classes of $\Pi_3$-patterns.
For the same reason as above, we only need to consider the duals of quantifier-patterns of length at most $3$ that belong to $\mathcal{E}$.
By Theorem \ref{thm:formula-class-collapse}, all the patterns in the common part of the polygonal regions in the left and right diagrams of Figure \ref{fig:arith2} are $\equiv_{\sf dm}$-equivalent.
Also, the proof of Theorem \ref{thm:formula-class-collapse} (4) actually shows $\comp{\exists^\infty\exists\forall}\equiv_{\sf dm}\comp{\exists^\infty\exists\forall^\infty}$.
Therefore, the polygonal region in the right diagram of Figure \ref{fig:arith2} is divided into at most two $\equiv_{\sf dm}$-equivalence classes.
Together with the other five quantifier-patterns, this gives at most seven $\equiv_{\sf dm}$-equivalence classes.
\end{proof}

\subsection{Separation}\label{sec:separation}

In order to argue that the classification of formulas based on quantifier-patterns is meaningful, we should show that quantifier-patterns provide a separation of many-one complexity.
Below the disjunction $\lor$ is always expressed by the quantifier $\exists i\in\{0,1\}$.
Then, the $\lor\forall$-complete problem $\comp{\lor\forall}$ is given as follows:
\[x_0=0^\infty\ \lor\ x_1=0^\infty.\]

\begin{definition}[\cite{Kihara}]
A formula $\varphi$ is {\em amalgamable} if there exists a partial computable function $\kappa$ such that if at least one of $a_0,\dots,a_\ell$ is a witness for $\varphi(x)$, then $\kappa(a_0,\dots,a_\ell,x)$ is a witness for $\varphi(x)$.
\end{definition}

\begin{lemma}[\cite{Kihara}]\label{lem:amalgamable}
If $\varphi$ is amalgamable, then $\comp{\lor\forall}\not\leq_{\sf m}\varphi$.
\end{lemma}

\begin{prop}\label{prop:separation1}
$\comp{\lor\forall}\not\leq_{\sf m}\comp{\forall^\infty\forall\exists}$.
\end{prop}

\begin{proof}
By Lemma \ref{lem:amalgamable}, it suffices to show that $\comp{\forall^\infty\forall\exists}$ is amalgamable.
Note that a $\forall^\infty\forall\exists$-complete problem is of the following form:
\[\exists n\forall m\geq n\forall k\exists\ell.\ x(m,k,\ell)\not=0\]

Note that a witness for the $\forall\exists$-part can be computably recovered, so the essential information in a witness is only the $\exists n$ part.
Also, if $n$ is a witness for the above formula, then so is any $n'\geq n$.
To show amalgamability, assume that at least one of $a_0,\dots,a_\ell$ is a witness for the above formula.
Then, $\max_{i\leq\ell}a_i$ must also be a witness.
Hence, by considering $\kappa(a_0,\dots,a_\ell,x)=\max_{i\leq\ell}a_i$, conclude that $\comp{\forall^\infty\forall\exists}$ is amalgamable.
\end{proof}

This shows $\comp{\forall^\infty\forall\exists}\equiv_{\sf m}\comp{\forall^\infty\exists^\infty}<_{\sf m}\comp{\exists\forall\exists}$ since clearly $\comp{\lor\forall}<_{\sf m}\comp{\exists\forall\exists}$.

\begin{prop}\label{prop:separation2}
$\comp{\lor\forall}\not\leq_{\sf m}\comp{\forall\forall^\infty\forall}$.
\end{prop}

\begin{proof}
By Lemma \ref{lem:amalgamable}, it suffices to show that $\comp{\forall\forall^\infty\forall}$ is amalgamable.
Note that a $\forall\forall^\infty\forall$-complete problem is of the following form:
\[\forall n\exists s\forall t\geq s.\ x_n(t)=0^\infty.\]

A witness for the above formula is of the form $n\mapsto s$, so it is a function $f\colon\om\to\om$.
Moreover, if $f$ is a witness for the above formula, then so is any $g$ majorizing $f$ (i.e., $g(n)\geq f(n)$ for any $n$).
To show amalgamability, assume that at least one of $f_0,\dots,f_\ell$ is a witness for the above formula.
Define $\kappa(f_0,\dots,f_\ell,x)=g$ by $g(n)=\max_{i\leq\ell}f_i(n)$.
Then $g$ majorizes any $f_i$, so $g$ is a witness for the above formula.
This shows that $\comp{\forall\forall^\infty\forall}$ is amalgamable.
\end{proof}

Note that one can easily see $\comp{\lor\forall}\leq_{\sf m}\comp{\exists^\infty\forall}$ by considering $y_{2n}=x_0$ and $y_{2n+1}=x_1$.
Hence, we get $\comp{\exists^\infty\forall}\not\leq_{\sf m}\comp{\forall\forall^\infty\forall}$.

\begin{theorem}\label{thm:separation3}
$\comp{\forall^\infty\exists^\infty}\not\leq_{\sf m}\comp{\forall^\infty\exists}$.
\end{theorem}

\begin{proof}
Assume $\comp{\forall^\infty\exists^\infty}\leq_{\sf m}\comp{\forall^\infty\exists}$ via $\eta,r_-,r_+$.
In particular, for any $x=(x_n)_{n\in\om}$ and $y=\eta(x)$, we have the following:
\[\forall^\infty n\exists^\infty t.\ x_n(t)\not=0\iff \forall^\infty n\exists t.\ y_n(t)\not=0.\]

For each $n$, begin with $x_n=1^\infty$.
Then, for instance, $0$ is a witness for $\comp{\forall^\infty\exists^\infty}(x)$.
Hence, $a=r_-(0,x)$ is a witness for $\comp{\forall^\infty\exists}(y)$.
Then $n=r_+(a,x)$ is a witness for $\comp{\forall^\infty\exists^\infty}(x)$, which means that, for any $m\geq n$, $x_m$ has infinitely many nonzero values.
Thus, $n+1$ is also a witness for $\comp{\forall^\infty\exists^\infty}(x)$.
Hence, $c=r_-(n+1,x)$ is a witness for $\comp{\forall^\infty\exists}(y)$.

Now, since $a$ is a witness for $\comp{\forall^\infty\exists}(y)$, for any $b\geq a$ there is $t$ such that $y_b(t)\not=0$.
By continuity, there is $s$ such that $r_+(a,x\upto s)=n$ and $r_-(n+1,x\upto s)=c$.
Moreover, if $s$ is sufficiently large, then for $y^s=\eta(x\upto s)$, for any $b\in[a,c]$ we already see a value $t<s$ such that $y^s_b(t)\not=0$.
Then, let $x_n'$ be the result of switching all values of $x_n$ after $s$ to $0$; that is, $x_n'\upto s=x\upto s$ and $x'(u)=0$ for any $u\geq s$.
For $m\not=n$, keep $x'_m=x_m=1^\infty$.
Now, let us consider $y'=\eta(x')$.

By our construction, $n$ is a not witness for $\comp{\forall^\infty\exists^\infty}(x')$ since $x'_n$ contains only finitely many nonzero values, but $n+1$ is a witness.
Our choice of $s$ keeps $r_-(n+1,x')=r_-(n+1,x)=c$, which is a witness for $\comp{\forall^\infty\exists}(y')$ by the property of $r_-$.
That is, for any $b\geq c$, there is $t$ such that $y'_b(t)\not=0$.
Moreover, by our choice of $s$, $\eta(x')$ extends $y^s=\eta(x\upto s)$, so for any $b\in[a,c]$, there is $t$ such that $y'_b(t)\not=0$.
Put these together, for any $b\geq a$, we get $y'_b\not=0^\infty$; hence $a$ is a witness for $\comp{\forall^\infty\exists}(y')$.
Again, our choice of $s$ keeps $r_+(a,x')=r_+(a,x)=n$, which is a witness for $\comp{\forall^\infty\exists^\infty}(x')$ by the property of $r_+$.
This means that $x_n'$ contains infinitely many nonzero values, which is not true.
\end{proof}

\begin{theorem}
$\comp{\forall^\infty\forall}\not\leq_{\sf m}\comp{\forall\forall^\infty}$.
\end{theorem}

\begin{proof}
Recall $\comp{\forall^\infty\forall}\equiv_{\sf m}{\sf Bdd}$; hence it suffices to show  ${\sf Bdd}\not\leq_{\sf m}\comp{\forall\forall^\infty}$.
Assume $\comp{\forall^\infty\forall}\leq_{\sf m}\comp{\forall\forall^\infty}$ via $\eta,r_-,r_+$.
In particular, for any $x=(x_n)_{n\in\om}$ and $y=\eta(x)$, we have the following:
\[\exists a\forall n.\ x_n\leq a\iff \forall n\exists s\forall t\geq s.\ y_n(t)=0.\]

For each $n$, begin with $x_n=0$.
Clearly ${\sf Bdd}(x)$ holds, so $\comp{\forall\forall^\infty}(\eta(x))$ also holds.
Let $u=(u_n)_{n\in\om}$ be the least witness for $y=\eta(x)$; that is, $u_n$ is the least such that $y_n(t)=0$ for any $t\geq u_n$.
If $r_-(a+1,x)=(v_n)_{n\in\om}$, by the minimality of $u$, we have $u_n\leq v_n$ for any $n$.
By continuity, if $\ell$ is sufficiently large, then $r_+(u,x\upto\ell)=r_+(u,x)=a$.
Let $m\geq \ell$ be sufficiently large such that $r_-(a+1,x\upto m)$ extends $(v_n)_{n<\ell}$ and $\eta(x\upto m)$ extends $\{y_n(t):n<\ell\mbox{ and }t\leq v_n\}$; that is, for $y^m=\eta(x\upto m)$, the value $y^m_n(t)=y_n(t)$ is determined for any $n<\ell$ and $t\leq v_n$.
Then, let $x'$ be the result of changing some value of $x_n$ after $m$ to $a+1$; that is, for some $t>m$, put $x'_t=a+1$  and $x'_n=0$ for any other $n\not=t$.

Then $a+1$ is still an upper bound for $x'$, so it is a witness for ${\sf Bdd}(x')$, and since $t>m\geq\ell$, $r_-(a+1,x')=(v_n')_{n\in\om}$ extends $(v_n)_{n<\ell}$; that is, $v'_n=v_n$ for any $n<\ell$.
Recall that $y_n(t)=0$ for any $t\geq u_n$, in particular, for any $[u_n,v_n]$.
By our choice of $m$, $y'=\eta(x')$ extends $y^m$, so $y'_n(t)=y^m_n(t)=y_n(t)$ for any $n<\ell$ and $t\leq v_n$.
This implies $y'_n(t)=0$ for any $t\in[u_n,v_n]$.
By the property of $r_-$, $(v_n')_{n\in\om}$ is a witness for $\comp{\forall\forall^\infty}(y')$, and recall that, $(u_n)_{n<\ell}$ is a witness for $(y_n)_{n<\ell}=(y_n')_{n<\ell}$.
Thus, the result $u'$ of replacing the first $\ell$ values of $v'$ with $(u_n)_{n<\ell}$ is also a witness for $\comp{\forall\forall^\infty}(y')$.
By the property of $r_+$, $r_+(u',x')$ must be a witness for ${\sf Bdd}(x')$.
However, as $t>m\geq\ell$, we have $r_+(u',x')=r_+(u',x)=a$, which cannot be a witness for ${\sf Bdd}(x')$ since $x'(t)=a+1$.
\end{proof}

\begin{theorem}\label{thm:separation-unbounded}
$\comp{\forall\forall^\infty\forall}\not\leq_{\sf m}\comp{\exists^\infty\forall}$.
\end{theorem}

\begin{proof}
Recall $\comp{\forall\forall^\infty\forall}\equiv_{\sf m}\forall{\sf Bdd}$; hence it suffices to show  $\forall{\sf Bdd}\not\leq_{\sf m}\comp{\exists^\infty\forall}$.
Assume $\forall{\sf Bdd}\leq_{\sf m}\comp{\exists^\infty\forall}$ via $\eta,r_-,r_+$.
In particular, for any $x$ and $\eta(x)=(y_k)_{k\in\om}$, we have the following:
\[\forall n\exists a\forall t.\ x(n,t)\leq a\iff \exists^\infty k.\ y_k=0^\infty.\]

A witness for the left-hand side is a function majorizing $\tilde{x}(n)=\sup_tx(n,t)$, and a witness for the right-hand side is an infinite increasing sequence of elements in $Y=\{k:y_k=0^\infty\}$.
For $n,t$, begin with $x(n,t)=0$.
For such $x$, any $g\in\om^\om$ is a witness for $\forall{\sf Bdd}(x)$, so $r_-(g,x)$ is defined.
A value $k$ is {\em $a$-bounded at $n$} if, for any $g\in\om^\om$, whenever $r_-(g,x)(m)=k$ for some $m$, we have $g(n)\leq a$, and $k$ is {\em unbounded at $n$} if $k$ is not $a$-bounded at $n$ for any $a$.

\begin{claim}
There exists $n$ such that infinitely many values $k$ are unbounded at $n$.
\end{claim}

\begin{proof}
Suppose not.
Then, for any $n$, there are at most finitely many $k$ which are unbounded at $n$.
Let $A_n$ be the finite set consisting of all such $k$'s, and put $A_{\leq n}=\bigcup_{m\leq n}A_m$.
Then each $B_n=A_{n+1}\setminus A_{\leq n}$ is finite, and for any $z\in B_n$, we must have $z\not\in A_n$, so $z$ is $a_z$-bounded at $n$ for some $a_z$.
Put $b_n=\max_{z\in B_n}a_z$.
By definition, if $g(n)>b_n$ then $B_n$ does not intersect with the image of $r_-(g,x)$.
Hence, if $g(n)>b_n$ for any $n$, then the image of $r_-(g,x)$ does not intersect with $\bigcup_nB_n=\bigcup_{n>0}A_n$.
Such a $g$ exists since each $B_n$ is finite, so fix such a $g$.
As mentioned above, $r_-(g,x)$ is defined, say $r_-(g,x)(0)=z$, so by continuity, for a sufficiently large $n>0$, we have $r_-(g\upto n,x)(0)=z$.
Hence, for any $c$, we get $g_c$ extending $g\upto n$ such that $g_c(n)=c$, while keeping $r_-(g_c,x)(0)=z$.
Considering an arbitrarily large $c$, this means that the value $z$ is unbounded at $n$.
Therefore, by the definition of $A_n$, we get $z\in A_n$.
Then the image of $r_-(g,x)$ and $A_n$ have an intersection, but this is impossible.
\end{proof}

By the property of $r_-$, note that $\{r_-(g,x)(m):g\in\om^\om\mbox{ and }m\in\om\}\subseteq Y$.
If $k$ is unbounded at some $n'$ then, in particular, $k$ appears in some increasing sequence $r_-(g,x)$ of elements in $Y$, so we have $k\in Y$.
Let $n$ be as in Claim.
Then there is an increasing sequence $\bar{k}=(k_i)_{i\in\om}$ consisting of values which are unbounded at $n$.
This also gives an increasing sequence of elements in $Y$; that is, $\bar{k}$ is a witness for $\comp{\exists^\infty\forall}(y)$.
Hence, $r_+(\bar{k},x)$ is defined, say $r_+(\bar{k},x)(n)=a$, so by continuity, for a sufficiently large $s_0$, we have $r_+(\bar{k}\upto s_0,x\upto s_0)(n)=a$.
Since $k_i$ is not $a$-bounded at $n$, there exists $g_i$ such that $g_i(n)\geq a+1$ and $r_-(g_i,x)(m_i)=k_i$ for some $m_i$.
By continuity, if $s_1\geq s_0$ is sufficiently large, we have $r_-(g_i\upto s_1,x\upto s_1)(m_i)=k_i$ for each $i<s_0$.
Then define $x'(n,t)=a+1$ for some $t\geq s_1$, and keep other values as $0$, i.e., $x'(m,s)=0$ for any $(m,s)\not=(n,t)$.

Then consider $Y'=\{k:y'_k=0^\infty\}$ for $y'=\eta(x')$.
By our choice of $s_1$, the computation $r_-(g_i,x')(m_i)=k_i$ is maintained for each $i<s_0$,
Moreover, $g_i$ majorizes $\tilde{x}'(n)=\sup_tx'(n,t)$, which means that $g_i$ is still a witness for $\forall{\sf Bdd}(x')$, and so $r_-(g_i,x')(m_i)=k_i\in Y'$ by the property of $r_-$.
Hence, $(k_i)_{i<s_0}$ extends to a witness $\bar{k}'$ for $\comp{\exists^\infty\forall}(y')$.
By the property of $r_+$, $r_+(\bar{k}',x')$ must be a witness for $\forall{\sf Bdd}(x')$, which means that it must majorize $\tilde{x}'$.
However, by our choice of $s_0$, the computation $r_+(\bar{k}',x')(n)=a$ is maintained, so $r_+(\bar{k}',x')(n)=a<x'(n,t)\leq\tilde{x}'(n)$, which leads to a contradiction.
\end{proof}

\begin{theorem}\label{thm:separation-unconcentrated}
$\comp{\forall\exists\forall}\not\leq_{\sf m}\comp{\exists^\infty\forall^\infty\forall}$.
\end{theorem}

\begin{proof}
Recall $\comp{\exists^\infty\forall^\infty\forall}\equiv_{\sf m}\exists^\infty{\sf Bdd}$; hence it suffices to show $\comp{\forall\exists\forall}\not\leq_{\sf m}\exists^\infty{\sf Bdd}$.
Assume $\comp{\forall\exists\forall}\leq_{\sf m}\exists^\infty{\sf Bdd}$ via $\eta,r_-,r_+$.
In particular, for any $x$ and $\eta(x)=y$, we have the following:
\[\forall n\exists a\forall t.\ x(n,a,t)=0\iff \exists^\infty k\exists b\forall t.\ y(k,t)\leq b.\]

A witness for the right-hand side is of the form $m\mapsto(k_m,b_m)$, where $k_m\geq m$ and $y(k_m,t)\leq b_m$ for any $t$.
For each $n,a,t$, begin with $x(n,a,t)=0$.
For such $x$, any $g\in\om^\om$ is a witness for $\comp{\forall\exists\forall}(x)$, so $r_-(g,x)$ is defined.
A value $k$ is {\em $a$-concentrated at $n$} if, for any $g\in\om^\om$, whenever $r_-(g,x)(m)=(k,b)$ for some $m,b$, we have $g(n)=a$, and $k$ is {\em unconcentrated at $n$} if $k$ not $a$-concentrated at $n$ for any $a$.

\begin{claim}
There exists $n$ such that infinitely many values $k$ are unconcentrated at $n$.
\end{claim}

\begin{proof}
Suppose not.
Then, for any $n$, there are at most finitely many $k$ which are unconcentrated at $n$.
Let $A_n$ be the finite set consisting of all such $k$'s, and put $A_{\leq n}=\bigcup_{m\leq n}A_m$.
Then each $B_n=A_{n+1}\setminus A_{\leq n}$ is finite, and for any $z\in B_n$, we must have $z\not\in A_n$, so $z$ is $a_z$-concentrated at $n$ for some $a_z$.
By definition, if $g(n)\not\in\{a_z:z\in B_n\}$ then $B_n$ does not intersect with the image of $\pi_0\circ r_-(g,x)$.
Hence, if $g(n)\not\in\{a_z:z\in B_n\}$ for any $n$, then the image of $\pi_0\circ r_-(g,x)$ does not intersect with $\bigcup_nB_n=\bigcup_{n>0}A_n$.
Such a $g$ exists since each $B_n$ is finite, so fix such a $g$.
As mentioned above, $r_-(g,x)$ is defined, say $r_-(g,x)(0)=(z,b)$, so by continuity, for a sufficiently large $n>0$, we have $r_-(g\upto n,x)(0)=(z,b)$.
Hence, for any $c$, we get $g_c$ extending $g\upto n$ such that $g_c(n)=c$, while keeping $r_-(g_c,x)(0)=(z,b)$.
Considering an arbitrarily large $c$, this means that the value $z$ is unconcentrated at $n$.
Therefore, by the definition of $A_n$, we get $z\in A_n$.
Then the image of $r_-(g,x)$ and $A_n$ have an intersection, but this is impossible.
\end{proof}

Let $n$ be as in Claim.
Then there is an increasing sequence $(k_i)_{i\in\om}$ consisting of values which are unconcentrated at $n$.
Each $k_i$ is not $a$-concentrated at $n$, so there exist $g_i^0,g_i^1\in\om^\om$ such that $g_i^0(n)\not=g_i^1(n)$, and for each $j<2$, we have $r_-(g_i^j,x)(m_i^j)=(k_i,b_i^j)$ for some $m_i^j$ and $b_i^j$.
As mentioned above, any $g^j_i$ is a witness for $\comp{\forall\exists\forall}(x)$, so $r_-(g_i^j,x)(m_i^j)$ is a witness for $\exists^\infty{\sf Bdd}(y)$.
Then, for $c_i=\max\{b_i^0,b_i^1\}$, one can see that $(\bar{k},\bar{c})=(k_i,c_i)_{i\in\om}$ is also a witness for $\exists^\infty{\sf Bdd}(y)$.
Hence, $r_+((\bar{k},\bar{c}),x)$ is defined, say $r_+((\bar{k},\bar{c}),x)(n)=a$, so by continuity, for a sufficiently large $s_0$, we have $r_+((\bar{k}\upto s_0,\bar{c}\upto s_0),x)(n)=a$.
By continuity, for a sufficiently large $s_1\geq s_0$, we have $r_-(g_i^j,x\upto s_1)(m_i^j)=(k_i,b_i^j)$ for each $i<s_0$ and $j<2$.
Then define $x'(n,a,t)\not=0$ for some $t\geq s_1$, and keep other values as $0$, i.e., $x'(m,b,s)$ for any $(m,b,s)\not=(n,a,t)$.

For each $i<s_0$, since $g_i^0(n)\not=g_i^1(n)$, we get $g_i^j(n)\not=a$ for some $j<2$.
Fix such $j$.
Then $g_i^j$ is a witness for $\comp{\forall\exists\forall}(x')$.
For $y'=\eta(x')$, by our choice of $s_1$, the computation $r_-(g^j_i,x')(m_i^j)=(k_i,b_i^j)$ is maintained for each $i<s_0$.
The chosen $g^j_i$ is still a witness (since $g_i^j(n)\not=a$), so $(k_i,b^j_i)_{i<s_0}$ is extendible to a witness for $\exists^\infty{\sf Bdd}(y')$; thus, $y'(k_i,t)\leq b^j_i\leq c_i$ for any $i<s_0$ and $t$.
Hence, $(k_i,c_i)_{i<s_0}$ is extendible to a witness $(\bar{k}',\bar{c}')$ for $\exists^\infty{\sf Bdd}(y')$.
By our choice of $s_0$, we have $r_+((\bar{k}',\bar{c}'),x')(n)=a$ and $x'(n,a,t)\not=0$.
However, this means that $r_+((\bar{k}',\bar{c}'),x')$ does not give a witness for $\comp{\forall\exists\forall}(x')$, which leads to a contradiction.
\end{proof}

Summarizing the results in Section \ref{sec:separation}, we now see that no more arrows can be added to Figure \ref{fig:quantifier-class}.
For $\Sigma_3$-patterns, Theorem \ref{thm:separation3} and (the comment after) Proposition \ref{prop:separation1} show $\comp{\forall^\infty\exists}<_{\sf m}\comp{\forall^\infty\exists^\infty}<_{\sf m}\comp{\exists\forall\exists}$.
For $\Pi_3$-patterns, (the comment after) Proposition \ref{prop:separation2} and Theorem \ref{thm:separation-unbounded} show $\comp{\forall\forall^\infty}$ and $\comp{\exists^\infty\forall}$ are ${\sf m}$-incomparable.
Combining with Theorem \ref{thm:separation-unconcentrated}, we get $\comp{\forall\forall^\infty}<_{\sf m}\comp{\forall\forall^\infty\forall}\bot_{\sf m}\comp{\exists^\infty\forall}<_{\sf m}\comp{\exists^\infty\forall^\infty\forall}<_{\sf m}\comp{\forall\exists\forall}$, which is exactly the same shape as in Figure \ref{fig:quantifier-class}.
For ${\sf dm}$-reducibility, first note that, since $\qf{P}\not\leq_{\sf m}\qf{Q}$ clearly implies $\qf{P}\not\leq_{\sf dm}\qf{Q}$, the previous observations imply $\comp{\exists^\infty\forall}\not\leq_{\sf dm}\comp{\forall\forall^\infty\forall}$; $\comp{\forall\forall^\infty}<_{\sf dm}\comp{\forall\forall^\infty\forall}$; and $\comp{\exists^\infty\forall}<_{\sf dm}\comp{\exists^\infty\forall^\infty}<_{\sf dm}\comp{\exists^\infty\forall^\infty\forall}$.
By Theorem \ref{thm:formula-class-collapse}, we also have $\comp{\exists^\infty\forall^\infty\forall}<_{\sf m}\comp{\forall\exists\forall}\equiv_{\sf m}\comp{\exists^\infty\exists\forall}$, so $\comp{\exists^\infty\forall^\infty\forall}<_{\sf dm}\comp{\exists^\infty\exists\forall}$.
Moreover, we have $\comp{\forall\forall^\infty}\not\leq_{\sf dm}\comp{\exists^\infty\exists\forall}$ since by Theorem \ref{thm:formula-class-collapse}, the dual of $\comp{\forall\forall^\infty}$ is $\exists\forall\exists$-complete, and the dual of $\comp{\exists^\infty\exists\forall}$ is ${\sf m}$-equivalent to $\comp{\forall^\infty\exists^\infty}<_{\sf m}\comp{\exists\forall\exists}$ by Proposition \ref{prop:separation1}.
Combining the above, we conclude that Figure \ref{fig:quantifier-class} is complete.

\subsection{Other quantifier-petterns}\label{sec:separation-perfectness}
In Section \ref{sec:example-restrict-quantify}, we have dealt with quantifiers other than $\forall,\exists,\forall^\infty,\exists^\infty$.
Here, we consider the strength of the quantifier-pattern $\forall^\to\exists\forall$.
By taking $\gamma$ as a trivial formula, we clearly have $\comp{\forall\exists\forall}\leq_{\sf dm}\comp{\forall^\to\exists\forall}$.
The question is whether these are ${\sf m}$-equivalent.
For the duals, we can see that it holds true.

\begin{prop}
$\comp{\exists^\land\forall\exists}\equiv_{\sf m}\comp{\exists\forall\exists}$.
\end{prop}

\begin{proof}
It suffices to show $\comp{\exists^\land\forall\exists}\leq_{\sf m}\comp{\exists\forall\exists}$.
Given $(p,x)$, we construct $y=\eta(p,x)$ such that
\[\exists n\ (p_n=0^\infty\mbox{ and }\forall m.\ x^n_m\not=0^\infty)\iff\exists n\forall m.\ y_n(m)\not=0^\infty\]

Given $(x,p)$, for each $n,k,m$, we check if $p_n(k)=0$.
If yes, put $y_n(\pair{k,m})=x^n_m$; otherwise, put $y_n(\pair{k,m})=0^\infty$.

Let $n$ be a witness for $\comp{\exists^\land\forall\exists}(x,p)$.
Then $p_n=0^\infty$ and $x^n_m\not=0^\infty$ for any $m$.
In this case, $y_n(k,m)=x_n(m)\not=0^\infty$ for any $k,m$.
Hence, $n$ is also a witness for $\comp{\exists\forall\exists}(y)$.

Conversely, let $n$ be a witness for $\comp{\exists\forall\exists}(y)$.
Then, $y_n(\pair{k,m})\not=0^\infty$ for any $k,m$.
In this case, we must have $p_n(k)=0$, and thus $y_n(\pair{k,m})=x^n_m$.
Therefore, we get $p_n=0^\infty$ and $x_n(m)=y_n(0,m)\not=0^\infty$.
Hence, $n$ is also a witness for $\comp{\exists^\land\forall\exists}(x,p)$.
\end{proof}

\begin{theorem}\label{thm:above-AEA}
$\comp{\forall^\to\exists\forall}\not\leq_{\sf m}\comp{\forall\exists\forall}$.
\end{theorem}

\begin{proof}
Suppose $\comp{\forall^\to\exists\forall}\leq_{\sf m}\comp{\forall\exists\forall}$ via $\eta,r_-,r_+$.
In particular, for any $(p,x)$ and $y=\eta(p,x)$, we have the following:
\[\forall n\ (p_n=0^\infty\to\exists m.\ x^n_m=0^\infty)\iff\forall n\exists m.\ y_n(m)=0^\infty.\]

Begin with $x^n_m=0^\infty$ and $p_n=0^\infty$.
Clearly, $n\mapsto c_0(n)=0$ is a witness for $\comp{\forall^\to\exists\forall}(p,x)$.
Then, one can see $r:=r_+(r_-(c_0,(p,x)),(p,x))=c_0$.
Otherwise, $r(n)\not=0$ for some $n$, which is determined after reading a finite initial segment $(p,x)\upto\ell$ of $(p,x)$ by continuity.
By changing only the value $x^n_{r(n)}(t)$ for $t>\ell$, we get some $\tilde{x}$ such that $\tilde{x}^n_{r(n)}\not=0^\infty$ while keeping $\tilde{x}\upto\ell=x\upto\ell$.
As $r(n)\not=0$, $c_0$ is still a witness for $\comp{\forall^\to\exists\forall}(p,\tilde{x})$.
Therefore, $\tilde{r}=r_-(c_0,(p,\tilde{x}))$ is a witness for $\comp{\forall\exists\forall}(\tilde{y})$, where $\tilde{y}=\eta(p,\tilde{x})$.
In particular, $\tilde{x}^n_{\tilde{r}(n)}=0^\infty$.
However, since $\tilde{x}\upto\ell=x\upto\ell$, we must have $\tilde{r}(n)=r(n)$, and thus $0^\infty=\tilde{x}^n_{\tilde{r}(n)}=\tilde{x}^n_{r(n)}\not=0^\infty$, which is impossible.

Now, $r(0)$ is determined after reading some finite initial segment $(p,x)\upto\ell$ by continuity.
Again, by changing only the value $x^0_{r(0)}(t)$ for $t>\ell$, we get some $\tilde{x}$ such that $\tilde{x}^0_{r(0)}\not=0^\infty$ while keeping $\tilde{x}\upto\ell=x\upto\ell$.
Here, we also keep $\tilde{x}_n(m)=0^\infty$ for any $(n,m)\not=(0,0)$.

Then $\comp{\forall^\to\exists\forall}(p,\tilde{x})$ still holds; thus, $\comp{\forall\exists\forall}(\tilde{y})$ also holds, where $\tilde{y}=\eta(p,\tilde{x})$.
If $q=r_-(c_0,(p,x)\upto \ell)\upto \ell$ is extendible to a witness for $\comp{\forall\exists\forall}(\tilde{y})$, then $\tilde{r}=r_+(q,(p,x))$ is also extendible to a witness for $\comp{\forall^\to\exists\forall}(p,\tilde{x})$; that is, $\tilde{x}^n_{\tilde{r}(n)}=0^\infty$; however, we have $\tilde{r}(0)=r(0)$ and $\tilde{x}^0_{r(0)}\not=0^\infty$.
Hence, $q$ is not extendible to a witness for $\comp{\forall\exists\forall}(\tilde{y})$, which means that there is $n<\ell$ such that $\tilde{y}_n(q(n))\not=0^\infty$, so $\tilde{y}_n(q(n))(u)\not=0$ for some $u$.
By continuity, this value $\tilde{y}_n(q(n))(u)\not=0$ is determined after reading some $(p,\tilde{x})\upto\ell'$ for some $\ell'>\ell$.

Then, by changing only the value $p_0(t)$ for $t>\ell'$, we get some $\tilde{p}$ such that $\tilde{p}_0\not=0^\infty$ while keeping $\tilde{p}\upto\ell'=p\upto\ell'$.
Then, $c_0$ is a witness for $\comp{\forall^\to\exists\forall}(\tilde{p},\tilde{x})$; hence $\tilde{q}=r_-(c_0,(\tilde{p},\tilde{x}))$ must be a witness for $\comp{\forall\exists\forall}(y')$, where $y'=\eta(\tilde{p},\tilde{x})$.
Moreover, $\tilde{q}$ extends $q$, so in particular, $\tilde{q}(n)=q(n)$ for any $n<\ell$, and by our choice of $\ell'$, we must have $y'_n(\tilde{q}(n))(u)=\tilde{y}_n(q(n))(u)\not=0$.
However, this means $y'_n(\tilde{q}(n))\not=0^\infty$, so $\tilde{q}$ is not a witness for $\comp{\forall\exists\forall}(y')$.
This leads to a contradiction.
\end{proof}

\begin{cor}
$\pair{\forall\exists\forall}<_{\sf dm}{\sf Perfect}_{\sf bin}$.
\end{cor}

\begin{proof}
By Theorems \ref{thm:perfect-dicomplete} and \ref{thm:above-AEA}.
\end{proof}

\subsection{Specific problems}\label{sec:separation-concrete-problem}

Next, we perform a detailed analysis of the complexity of the infinite-diameter problem ${\sf InfDiam}={\sf FinDiam}^{\sf d}$ introduced in Section \ref{sec:example-global-boundedness}.
By combining Propositions \ref{prop:findiam-reduction-complete} and \ref{prop:AAinftyA-vs-FinDiam} and Theorem \ref{thm:formula-class-collapse}, we have obtained that ${\sf FinDiam}\equiv_{\sf m}\comp{\forall^\infty\forall\exists}\equiv_{\sf m}\comp{\forall^\infty\exists^\infty}$ and $\comp{\forall\forall^\infty\forall}\leq_{\sf m}{\sf InfDiam}\leq_{\sf dm}\comp{\exists^\infty\exists\forall}$.
Judging from Figure \ref{fig:quantifier-class}, this problem seems to be of intermediate complexity.

Moreover, ${\sf InfDiam}$ is not amalgamable since we have ${\sf DisConn}\leq_{\sf m}{\sf InfDiam}$ by Proposition \ref{prop:DisConn-vs-InfDiam}, and the disconnectedness problem ${\sf DisConn}$ is not amalgamable \cite{Kihara}.

\begin{lemma}[\cite{Kihara}]\label{lem:disconn-vs-amalgamable}
If $\varphi$ is amalgamable, then ${\sf DisConn}\not\leq_{\sf m}\varphi$.
\end{lemma}

In particular, ${\sf InfDiam}\not\leq_{\sf m}\varphi$ for any amalgamable $\varphi$.
Nevertheless, interestingly, $\comp{\lor\forall}$ is not ${\sf m}$-reducible to ${\sf InfDiam}$.

\begin{theorem}\label{thm:separation-InfDiam}
$\comp{\lor\forall}\not\leq_{\sf m}{\sf InfDiam}$.
\end{theorem}

\begin{proof}
Suppose $\comp{\lor\forall}\leq_{\sf m}{\sf InfDiam}$ via $\eta,r_-,r_+$.
Begin with $x_i=0^\infty$.
Then for each $i<2$, $r_-(i,x)=(a^i_n,b^i_n)_{n\in\om}$ is a witness for ${\sf InfDiam}(\eta(x))$; that is, the distance between $(a^i_n,b^i_n)$ is at least $n$.
For each $n$, let $E_n$ be the set of all two-element subsets of $V_n=\{a^0_{3n},b^0_{3n},a^1_{3n},b^1_{3n}\}$.
Counting the number of combinations of selecting two vertices from the set $V_n$ of at most four vertices, we see that the cardinality of $E_n$ is at most $6$.

Now, for any $f\in\prod_nE_n$, if $f$ is a witness for ${\sf InfDiam}(\eta(x))$, then $r_+(f,x)$ gives a witness for $\comp{\lor\forall}(x)$; that is, $x_i=0^\infty$ if $r_+(f,x)=i$.
In particular, either $f$ is not a witness for ${\sf InfDiam}(\eta(x))$ or $r_+(f,x)$ gives a witness for $\comp{\lor\forall}(x)$.
Moreover, which of these holds is determined after reading some finite information of $f$ and $x$.
If the latter holds, this is clear.
If the former holds, there is $n$ such that the distance between $f(n)=\{a,b\}$ is less than $n$; that is, in the graph $\eta(x)$, there is a path of length less than $n$ connecting $a$ and $b$.
After reading some finite information of $x$, it is determined that such a path exists in the graph $\eta(x)$.
Hence, by compactness of $\prod_nE_n$, there is $\ell$ such that either the graph $\eta(x\upto \ell)$ has a path of length less than $n$ connecting two vertices in $f(n)$ for some $n<\ell$ or the value of $r_+(f\upto\ell,x\upto\ell)\in\{0,1\}$ is determined.
Then let $\ell'\geq\ell$ be such that $r_-(i,x\upto\ell')$ extends $(a^i_n,b^i_n)_{n<3\ell}$ for each $i<2$.

Now, consider $T=\prod_{n<\ell}E_n$, and let us think of this as a tree of height $\ell$.
For each leaf $\rho\in T$, we define the {\em acting position} $\alpha(\rho)$ of $\rho$ as follows:
If the graph $\eta(x\upto\ell)$ has a path of length less than $n$ connecting two vertices in $\rho(n)$ for some $n<\ell$, then put $\alpha(\rho)=\rho\upto (n+1)$.
Otherwise, $r_+(\rho,x\upto\ell)$ is defined by our choice of $\ell$, so first, let us label $\rho$ with this value.
Note that if we put $x_i'(t)\not=0$ for a sufficiently large $t\geq \ell'$ (so $x_i'\not=0^\infty$), then for any $i$-labeled leaf $\rho$ there is $n<\ell$ such that the distance of vertices in $\rho(n)$ must be less than $n$ in the graph $\eta(x')$.
Otherwise, $\rho$ is extendible to a witness $f$ for ${\sf InfDiam}(\eta(x'))$, so $r_+(f,x')$ must be a witness for $\comp{\lor\forall}(x')$; however, we have $r_+(f,x')=i$ since $\rho$ is labeled by $i$, and $x_i'\not=0^\infty$, which means that $r_+(f,x')$ is not a correct witness for $\comp{\lor\forall}(x')$.
Hence, we must have $n$ such that the distance between two vertices in $\rho(n)$ is less than $n$ in $\eta(x')$, so for such an $n$, we put $\alpha(\rho)=\rho\upto(n+1)$.

We say that $\sigma\in T$ is a {\em full-acting node} if each of the immediate successor of $\sigma$ in $T$ is the acting position of some leaf.
In other words, for the length $k$ of $\sigma$, for each $a\in E_k$ there is a leaf $\rho\in T$ such that $\alpha(\rho)=\sigma\fr a$.

\begin{claim}
There always exists a full-acting node in $T$.
\end{claim}

\begin{proof}
Suppose that a full-acting node does not exist in $T$.
In particular, the root is not a full-acting node, so by definition, the root has an immediate successor $a_1\in T$ which is not an acting position of any leaf.
By our assumption, $a_1$ is also not a full-acting node, so $a_1$ has an immediate successor $\pair{a_1,a_2}\in T$ which is not an acting position of any leaf.
Repeat this process to obtain a leaf $\rho=\pair{a_1,a_2,\dots a_\ell}$.
Here, $\rho\upto j$ is not an acting position for any $j\geq 1$, so the acting position $\alpha(\rho)$ of $\rho$ must be the root.
However, the acting position of the leaf is always a node of length at least $1$, so this is impossible.
\end{proof}

Now fix a full-acting node $\sigma\in T$, and let $n$ be its length.
Then, for each $a\in E_n$, there is a leaf $\rho_a$ extending $\sigma\fr a$ whose acting position is $\sigma\fr a$.
Now, the value of the label of $a\in E_n$ is set to the value of the label of $\rho_a$, i.e., the value of $r_+(\rho_a,x'\upto\ell)$.
Here, if the label of $\rho_a$ is undefined, choose any value of $0$ or $1$.
In other words, either there is a path of length less than $n$ connecting the two vertices in $a$, or else $a$ is labeled by the value of $r_+(\rho_a,x'\upto\ell)$.

\begin{claim}
$z=\{a^i_{3n},b^i_{3n}\}$ is labeled by $i$.
\end{claim}

\begin{proof}
Suppose not.
Since this $z$ is the $(3n)$th component of $r_-(i,x)$; that is, the $(3n)$th component of a witness for ${\sf InfDiam}(\eta(x))$, the distance between $\{a^i_{3n},b^i_{3n}\}$ is at least $3n$.
Therefore, the label of $z$ must be defined, and its value is $1-i$ by our assumption.
Since $\rho_z$ is labeled by $1-i$, the action position being $\alpha(\rho_z)=\sigma\fr z$ means that the distance between two vertices in $z$ must become less than $n$ in $\eta(x')$ after putting $x_{1-i}'(t)\not=0$.
In this case, $x'_i=0^\infty$ is maintained, so $i$ is a witness for $\comp{\lor\forall}(x')$.
Now, since $t\geq\ell'$, $r_-(i,x')$ extends $(a^i_{n},b^i_{n})_{n<3\ell}$, so in particular the distance between $z=\{a^i_{3n},b^i_{3n}\}$ must be at least $3n$ even in $\eta(x')$.
This is impossible.
\end{proof}

The rest can be done by imitating the proof of $\comp{\lor\forall}\not\leq_{\sf m}{\sf DisConn}$ in \cite{Kihara}, but the notation is slightly different, so we give a complete proof here.
Each edge of the complete graph $G_n=(V_n,E_n)$ with at most $4$ vertices $V_n=\{a^0_{3n},b^0_{3n},a^1_{3n},b^1_{3n}\}$ is labeled in the above way.

\begin{claim}
For some $i<2$, there exists a path $\gamma$ of length at most $3$ connecting $a^i_{3n}$ and $b^i_{3n}$ in $G_n$ such that $\gamma$ consists only of edges with label $1-i$.
\end{claim}

\begin{proof}
Suppose not for $i=0$ (the same argument applies to the case where $i=1$).
For the sake of clarity, edges of a graph are described using the symbol $\to$ below, but note that we consider an undirected graph, so $a\to b$ and $b\to a$ represent the same edge, and the labels are also the same.

First, the vertices $a^0_{3n}$ and $b^0_{3n}$ are connected by the path $a^0_{3n}\to a^1_{3n}\to b^1_{3n}\to b^0_{3n}$ of length $3$ in the complete graph $G_n$.
By our assumption, one of these edges is labeled by $0$.
By the previous claim, the edge $a^1_{3n}\to b^1_{3n}$ is labeled by $1$, so either $a^0_{3n}\to a^1_{3n}$ or $b^1_{3n}\to b^0_{3n}$ must be labeled by $0$.

Case 1.
Assume that the edge $a^0_{3n}\to a^1_{3n}$ is labeled by $0$.
The vertices $a^0_{3n}$ and $b^0_{3n}$ are connected by the path $a^0_{3n}\to b^1_{3n}\to b^0_{3n}$ of length $2$.
By our assumption, either $a^0_{3n}\to b^1_{3n}$ or $b^1_{3n}\to b^0_{3n}$ is labeled by $0$.
In the former case, the path $a^1_{3n}\to a^0_{3n}\to b^1_{3n}$ of length $2$ connects $a^1_{3n}$ and $b^1_{3n}$ and consists only of $0$-labeled edges.
In the latter case, by the above claim, the edge $a^0_{3n}\to b^0_{3n}$ is labeled by $0$, so the path $a^1_{3n}\to a^0_{3n}\to b^0_{3n}\to b^1_{3n}$ of length $3$ connects $a^1_{3n}$ and $b^1_{3n}$ and consists only of $0$-labeled edges.

Case 2.
Assume that the edge $b^1_{3n}\to b^0_{3n}$ is labeled by $0$.
The vertices $a^0_{3n}$ and $b^0_{3n}$ are connected by the path $a^0_{3n}\to a^1_{3n}\to b^0_{3n}$ of length $2$.
By our assumption, either $a^0_{3n}\to a^1_{3n}$ or $a^1_{3n}\to b^0_{3n}$ is labeled by $0$.
In the latter case, the path $a^1_{3n}\to b^0_{3n}\to b^1_{3n}$ of length $2$ connects $a^1_{3n}$ and $b^1_{3n}$ consists only of $0$-labeled edges.
In the former case, by the above claim, the edge $a^0_{3n}\to b^0_{3n}$ is labeled by $0$, so the path $a^1_{3n}\to a^0_{3n}\to b^0_{3n}\to b^1_{3n}$ of length $3$ connects $a^1_{3n}$ and $b^1_{3n}$ and consists only of $0$-labeled edges.
\end{proof}

Fix $i<2$ as in the previous claim.
Then put $x'_{1-i}(t)\not=0$ for a sufficiently large $t$.
For each $a\in E_n$, since the acting position of the leaf $\rho_a\in T$ is $\sigma\fr a$, if $a$ is labeled by $1-i$, then by the definition of the acting position, the distance between two vertices in $\rho_a(n)=a$ is less than $n$ in the graph $\gamma(x')$.
This means that the end points of each $(1-i)$-labeled edge in the complete graph $G_n=(V_n,E_n)$ are connected by a path of length less than $n$ in the graph $\gamma(x')$.
By the previous claim, $a^i_{3n}$ and $b^i_{3n}$ are connected by a $(1-i)$-labeled path of length at most $3$ in $G_n$, so they are connected by a path of length less than $3n$ in $\gamma(x')$.
However, $i$ is still a witness for $\comp{\lor\forall}(x')$, so $r_-(i,x')$ must be a witness for ${\sf InfDiam}(\eta(x'))$, which extends $(a^i_n,b^i_n)_{n<3\ell}$.
In particular, the distance between $(a^i_{3n},b^i_{3n})$ must be at least $3n$ in $\gamma(x')$, which leads to a contradiction.
\end{proof}

\begin{cor}
$\comp{\forall\forall^\infty\forall}<_{\sf m}{\sf InfDiam}<_{\sf m}\comp{\forall\exists\forall}$, and ${\sf InfDiam}<_{\sf dm}\comp{\exists^\infty\exists\forall}$.
\end{cor}

\begin{proof}
Clearly, $\comp{\lor\forall}\leq_{\sf m}\comp{\forall\exists\forall}$, so Theorem \ref{thm:separation-InfDiam} implies ${\sf InfDiam}<_{\sf m}\comp{\forall\exists\forall}$.
The proof of Proposition \ref{prop:separation2} shows that $\comp{\forall\forall^\infty\forall}$ is amalgamable, so Lemma \ref{lem:disconn-vs-amalgamable} implies ${\sf DisConn}\not\leq_{\sf m}\comp{\forall\forall^\infty\forall}$.
Theorefore, by Proposition \ref{prop:AAinftyA-vs-FinDiam} and \ref{prop:DisConn-vs-InfDiam}, we obtain $\comp{\forall\forall^\infty\forall}<_{\sf m}{\sf InfDiam}$.
For the second assertion, the proof of Proposition \ref{prop:findiam-reduction-complete} shows ${\sf InfDiam}\leq_{\sf dm}\comp{\exists^\infty\exists\forall}$.
Moreover, we clearly have $\comp{\lor\forall}\leq_{\sf dm}\comp{\exists^\infty\exists\forall}$, so by Theorem \ref{thm:separation-InfDiam}, we get ${\sf InfDiam}<_{\sf dm}\comp{\exists^\infty\exists\forall}$.
\end{proof}

As a modification of the infinite-diameter problem, if we consider the problem of determining whether the diameter is at least $r$ for a fixed $r\geq 4$, then the modified problem turns out to be $\exists\forall$-complete.
\[{\sf Diam}_{\geq r}:\quad \exists u,v\in V\forall\gamma\;[(\mbox{$\gamma$ connects $u$ and $v$})\to |\gamma|\geq r].\]

\begin{prop}\label{prop:Diam4-EA-complete}
${\sf Diam}_{\geq 4}$ is $\exists\forall$-complete.
\end{prop}

\begin{proof}
Given $x=(x_n)_{n\in\om}$, we construct $\eta(x)=(V,E)$ such that
\[\exists n.\ x_n=0^\infty\iff\mbox{the diameter of $(V,E)$ is at least $4$.}\]

The graph has a path $a^n_0\to a^n_1\to a^n_2\to a^n_3\to a^n_4$ of length $4$ for each $n\in\om$.
Moreover, put $(a^n_i,a^m_i)\in E$ for any $n\not=m$.
If $x_n(t)\not=0$ for some $t$, then add $c^n_t\in V$ for such $t$ and put $(a^n_i,c^n_t)\in E$ for each $i<5$.

The distance between $c^n_t$ and $a^m_j$ is at most $2$ since there is a path $c^n_t\to a^n_j\to a^m_j$ of length $2$.
Hence, if there is a vertex of the form $c^n_t$, then since $a^n_i$ and $c^n_t$ are adjacent, the distance between $a^n_i$ and $a^m_j$ is at most $3$.
If there is no vertex of the form $c^n_t$, the distance between $a^n_0$ and $a^n_4$ is $4$.
This is because, if there is a path of length at most $3$ connecting $a^n_0$ and $a^n_4$, then it is necessary to pass through a vertex of the form $a^m_j$ for some $m\not=n$.
If $|i-j|\geq 2$, the shortest path between $a^n_i$ and $a^n_j$ that passes through such a vertex is $a^n_i\to a^m_i\to c^m_t\to a^m_j\to a^n_j$, whose length is $4$ (even if $c^m_t$ exists).
To summarize this, a vertex of the form $c^n_t$ exists iff, for any $m,i,j$, the distance between $(a^n_i,a^m_j)$ is at most $3$.

Now, if $n$ is a witness for $\comp{\exists\forall}(x)$, then there is no vertex of the form $c^n_t$.
Then the distance between $(a^n_0,a^n_4)$ is at least $4$ as discussed above, so $(a^n_0,a^n_4)$ is a witness for ${\sf Diam}_{\geq 4}(V,E)$.
Conversely, if $(u,v)$ is a witness for ${\sf Diam}_{\geq 4}(V,E)$; that is, the distance between $(u,v)$ is at least $4$, then by the above argument, neither $u$ nor $v$ can be of the form $c^n_t$. 
Hence, $u=a^n_i$ and $v=a^m_j$ for some $n,m,i,j$.
Then, there is no vertex of the form $c^n_t$ as discussed above.
This means that $x_n(t)=0$ for any $t$.
Hence, $n$ is a witness for $\comp{\exists\forall}(x)$.
\end{proof}

\begin{cor}
${\sf Diam}_{\geq 4}\not\leq_{\sf m}{\sf InfDiam}$.
\end{cor}

\begin{proof}
By Theorem \ref{thm:separation-InfDiam} we have $\comp{\lor\forall}\not\leq_{\sf m}{\sf InfDiam}$, and by Proposition \ref{prop:Diam4-EA-complete}, we also have $\comp{\lor\forall}\leq_{\sf m}\comp{\exists\forall}\leq_{\sf m}{\sf Diam}_{\geq 4}$.
Combining these, we obtain ${\sf Diam}_{\geq 4}\not\leq_{\sf m}{\sf InfDiam}$.
\end{proof}

\section{Open Question}

The most straightforward but important open problem is the problem of counting the number of equivalence classes of quantifier patterns at each level $n\geq 4$ of the arithmetical hierarchy.
\begin{question}
How many $\equiv_{\sf m}$-equivalence classes of $\Sigma_4$ quantifier-patterns are there?
How many $\equiv_{\sf m}$-equivalence classes of $\Pi_4$ quantifier-patterns are there?
How many $\equiv_{\sf dm}$-equivalence classes of $\Sigma_4$ quantifier-patterns are there?
\end{question}

Of course, this problem can also be considered for general $n$ other than $n=4$.
It would also be useful to have a general method of determining the relationship between two given quantifier-patterns.
This problem is expected to be difficult, but it is important.

\begin{question}
Find an algorithm which decides if $\comp{\qf{P}}\leq_{\sf m}\comp{\qf{Q}}$ for given quantifier-patterns $\qf{P}$ and $\qf{Q}$.
\end{question}

It is also important to reconsider previous research on the classical arithmetical hierarchy from the realizability-theoretic perspective.
For example, it is important to consider the problem of determining the exact arithmetical complexity (the exact quantifier-patterns) of natural decision problems involving groups \cite{BCR20}; see also \cite{Shpi10} for witness/search problems in group theory.

\begin{ack}
The author wishes to thank Gerard Glowacki, Kaito Takayanagi and Kakeru Yokoyama for their active participation in the spring 2024 seminar on this research.
The author's research was partially supported by JSPS KAKENHI (Grant Numbers 21H03392, 22K03401 and 23H03346).
\end{ack}


\bibliographystyle{plain}
\bibliography{references}

\end{document}